\documentclass[11pt]{amsart}
\usepackage{verbatim,xspace}
\usepackage{amssymb}
\usepackage{amsthm}
\usepackage{amsrefs}
\usepackage{amsxtra}
\usepackage{mathrsfs}
\usepackage{enumerate}
\usepackage{hyperref}
\usepackage[matrix,arrow,curve]{xy} 

\let\tilde\widetilde

\newcommand\xref[1]{\csname#1\endcsname}

\newcommand{\xrhardcode}[2]{\expandafter\def\csname#1\endcsname{#2}}
\xrhardcode{FGL-sec:modules}{2}
\xrhardcode{FGL-eq:restriction}{2.2}
\xrhardcode{FGL-sec:fixed-point-group}{3}
\xrhardcode{FGL-prop:fixed-point-group}{3.5}
\xrhardcode{FGL-rem:T-over-k}{3.6(ii)}
\xrhardcode{FGL-defn:parascopy}{4.1}
\xrhardcode{FGL-ex:fixed-point}{4.2(a)}
\xrhardcode{FGL-sec:weyl-embedding}{5}
\xrhardcode{FGL-lem:multiple}{5.3}
\xrhardcode{FGL-prop:weyl-embedding}{5.5}
\xrhardcode{FGL-prop:parascopy-equivalence}{7.5}
\xrhardcode{FGL-item:all-tori-equally-good}{i}
\xrhardcode{FGL-item:upper-torus-determined}{ii}
\xrhardcode{FGL-defn:equivalent-data}{7.3}
\xrhardcode{FGL-sec:duality}{8}
\xrhardcode{FGL-prop:dual-tori}{8.3}
\xrhardcode{FGL-sec:conorm}{9}
\xrhardcode{FGL-rem:square}{9.1}
\xrhardcode{FGL-prop:centralizer-weyl}{9.4}
\xrhardcode{FGL-thm:main}{11.1}
\xrhardcode{FGL-rem:main}{11.2}

\newcommand\refaprime{a$^\prime$}
\newcommand\refb{b}

\newcommand{\Cl}{\mathcal{C}}

\newcommand{\TTst}{\mathcal{T}_{\mathrm{st}}}
\newcommand{\loccit}{\emph{loc.\ cit.}\xspace}
\newcommand{\bX}{X}
\newcommand{\Z}{{\mathbb{Z}}}

\newcommand{\red}{_{\mathrm{red}}}
\newcommand{\enh}{_{\mathrm{enh}}}
\newcommand{\redwrong}{_{\mathrm{red}^\prime}}
\newcommand{\conn}{^\circ}
\newcommand{\st}{^\mathrm{st}}

\newcommand{\sep}{^{\mathrm{sep}}}

\newcommand{\Long}{^{\mathrm{long}}}
\newcommand{\Short}{^{\mathrm{short}}}

\newcommand{\dual}{{^\wedge}}
\DeclareMathOperator{\id}{id}
\DeclareMathOperator{\Weil}{R}
\DeclareMathOperator{\Rep}{Rep}

\newcommand{\univ}[1]{\underline{#1}} 
\newcommand{\CompOf}[2]{composition of $#1$ and $#2$}
\newcommand{\IsCompOf}[3]{$#1$ is a \CompOf{#2}{#3}}

\newcommand{\liftfactor}[3]{
\begin{xy}
\xymatrix@C=-20pt{
\Rep #3 & \\
& \Rep #2\ar@{~>}[ul] \\
\Rep #1\ar@{~>}[uu]\ar@{~>}[ur]
}
\end{xy}
}

\newcommand{\Q}{\mathbb{Q}}

\newcommand{\norm}[1][]{\NN_{#1}}
\newcommand{\dnorm}[1][]{{\widehat\NN_{#1}}}
\newcommand{\dnormst}[1][]{{\widehat\NN_{#1}\st}}
\newcommand{\normchar}[1][]{\NN^*_{#1}}

\newcommand{\dnormcochar}[1]{\widehat\NN_{#1,*}}
\newcommand{\dnormfunc}{\widehat{N}}

\DeclareMathOperator{\stab}{stab}
\DeclareMathOperator{\Hom}{Hom}
\DeclareMathOperator{\diag}{diag}

\DeclareMathOperator{\Int}{Int}

\DeclareMathOperator{\Aut}{Aut}

\DeclareMathOperator{\GL}{GL}
\DeclareMathOperator{\U}{U}

\DeclareMathOperator{\PGL}{PGL}
\DeclareMathOperator{\SO}{SO}
\DeclareMathOperator{\Sp}{Sp}
\DeclareMathOperator{\PGSp}{PGSp}
\DeclareMathOperator{\Gal}{Gal}
\DeclareMathOperator{\Ind}{Ind}

\DeclareMathOperator{\im}{im}
\DeclareMathOperator{\chr}{char}
\DeclareMathOperator{\pr}{pr}

\newcommand{\sfG}{{\mathsf{G}}}

\newcommand{\NN}{\mathcal{N}}
\newcommand{\inv}{^{-1}}

\newcommand{\set}[2]{ 
        {\left\{\left.
        #1\vphantom{#2\bigl(\bigr)}\,\right|
        \,#2\right\}}}

\newcommand{\lsup}[1]{{}^{#1}}
\newcommand{\lsub}[1]{{}_{#1}}
\newcommand{\llsub}[1]{{}_{#1}^{\phantom{x}}}

\newcommand{\maaap}[4]{\ensuremath{{#2}\colon{#3}#1{#4}}}
\newcommand{\map}{\maaap\longrightarrow}
\newcommand{\abmap}[2]{\ensuremath{{#1}\longrightarrow{#2}}}
\newcommand{\abmapto}[2]{\ensuremath{{#1}\longmapsto{#2}}}

\newtheorem{thm}[equation]{Theorem}  
\newtheorem{prop}[equation]{Proposition}
\newtheorem{lem}[equation]{Lemma} 
 
\theoremstyle{definition} 
\newtheorem{defn}[equation]{Definition} 
\theoremstyle{remark} 
\newtheorem{notation}[equation]{Notation} 
\newtheorem{rem}[equation]{Remark} 
\newtheorem{example}[equation]{Example} 
\newtheorem{examples}[equation]{Examples}
 
\newtheorem{para}[equation]{} 

\numberwithin{equation}{section}
\numberwithin{table}{section}

\title[Lifting: Explicit conorms]%
{Lifting representations of finite reductive groups II: Explicit conorms}

\date{\today}
\author{Jeffrey D.~Adler}
\address{Department of Mathematics and Statistics\\
American University\\
Washington, DC 20016-8050}
\email[Adler]{jadler@american.edu}
\author{Joshua M.~Lansky} 
\email[Lansky]{lansky@american.edu}

\begin{document}

\begin{abstract}
Let $k$ be a field, $\tilde{G}$ a connected reductive $k$-group,
and $\Gamma$ a finite group.
In a previous work, the authors defined what it means for a connected reductive
$k$-group $G$ to be \emph{parascopic} for $(\tilde{G},\Gamma)$.
Roughly, this is a simultaneous generalization of several settings.
For example, $\Gamma$ could act on $\tilde{G}$,
and $G$ could be the connected part of the group of $\Gamma$-fixed points in $\tilde{G}$.
Or $G$ could be an endoscopic group, a pseudo-Levi subgroup, or an isogenous image of $\tilde{G}$. 
If $G$ is such a group, and both $\tilde{G}$ and $G$ are $k$-quasisplit,
then we constructed a map $\dnormst$
from the set of stable semisimple conjugacy classes in the dual $G^\wedge(k)$
to the set of such classes in $\tilde{G}^\wedge(k)$.
When $k$ is finite, this implies a lifting from packets of representations of $G(k)$ to those of $\tilde{G}(k)$.

In order to understand such a lifting better, 
here we describe two ways in which $\dnormst$ can be made more explicit.
First, we can express our map in the general case in terms of simpler cases.
We do so by showing that $\dnormst$ is compatible with isogenies and with Weil restriction,
and also by expressing it as a composition of simpler maps.
Second,
in many cases we can 
construct an explicit $k$-morphism $\map{\dnormfunc}{G^\wedge}{\tilde{G}^\wedge}$
that agrees with $\dnormst$.
As a consequence, our
lifting of representations is seen to coincide with Shintani
lifting in some important cases.
\end{abstract}

\maketitle

\addtocounter{section}{-1}
\section{Introduction}

Suppose that $k$ is a field,
$\tilde{G}$ and $G$ are connected reductive $k$-groups,
$\Gamma$ is a finite group that acts on a maximal $k$-torus $\tilde{T}\subseteq\tilde{G}$ via $k$-automorphisms
that preserve a positive root system for $(\tilde G, \tilde T)$,
and $G$ is \emph{parascopic} for $\tilde{G}$ and the action of $\Gamma$ as in~\cite{adler-lansky:lifting}.
While our action of $\Gamma$ on $\tilde\Psi$ does not in general determine
an action of $\Gamma$ on $\tilde{G}$ (and even when it does,
this action is not unique),
nonetheless $G$ has many of the properties one would expect
of the connected part $(\tilde{G}^\Gamma)\conn$ of the group
of fixed points if such an action were to exist.
For example, the group $\Gamma$ acts on $\tilde{T}$,
and we obtain a $k$-map
$\abmap{T}{(\tilde T^\Gamma)\conn}$.
In addition to fixed-point groups,
some other examples of parascopic groups for $\tilde{G}$ are:
pseudo-Levi subgroups, endoscopic groups, and images under isogenies.
The precise relationship between $\tilde{G}$, $G$, and the action of $\Gamma$ is described
by a pair that we call a \emph{parascopic datum} for $(\tilde{G},\Gamma,G)$
(see Definition \ref{defn:parascopy}).

If $\tilde{G}$ and $G$ are $k$-quasisplit,
then one can form the duals $\tilde{G}^\wedge$ and $G^\wedge$,
and we showed in \cite{adler-lansky:lifting} that the parascopic datum induces
a \emph{conorm}, i.e., a canonical $k$-morphism $\dnorm$
from the variety of semisimple geometric conjugacy classes
in $G^\wedge$ to the analogous variety for $\tilde{G}^\wedge$.
Moreover, this map specializes to give a map $\dnormst$
from the set of stable (in the sense of Kottwitz \cite{kottwitz:rational-conj})
semisimple conjugacy classes in $G^\wedge(k)$ to the analogous
set for $\tilde{G}^\wedge(k)$.
In the special case where $k$ is finite, we have that
$G$ and $\tilde{G}$ are automatically $k$-quasisplit,
and stable and rational conjugacy coincide.
Since $G^\wedge(k)$-conjugacy classes in $G^\wedge(k)$ parametrize
collections
(``Lusztig series'')
of irreducible representations of $G(k)$,
and similarly for $\tilde{G}(k)$,
one obtains a lifting of such collections
from $G(k)$ to $\tilde{G}(k)$.

The description of the map $\dnorm$ in \cite{adler-lansky:lifting}
is explicit in some sense.
Given a semisimple element $s\in G^\wedge(k)$,
choose a maximal $k$-torus $T^\wedge\subseteq G^\wedge$ such that $s\in T^\wedge(k)$.
One can construct a corresponding $k$-torus
$\tilde{T}^\wedge \subseteq \tilde{G}^\wedge$, and an explicit $k$-homomorphism
$\map{\dnorm[T^\wedge]}{T^\wedge}{\tilde{T}^\wedge}$.
If $[s]$ and $[s]\st$ respectively denote the geometric and stable classes of $s$,
then $\dnorm([s])$ is the geometric class, and $\dnormst([s]\st)$ is the stable class, containing $\dnorm[T^\wedge](s)$.
In particular, the choices of $T^\wedge$ and $\tilde{T}^\wedge$ don't matter.

However, there are several situations where one can make
$\dnorm$ (and thus $\dnormst$) even more explicit.
First, one can sometimes express a given parascopic datum in terms
of simpler parascopic data.
For example, a datum can be factored
as a composition of other data
(see Definition \ref{defn:composition})
or realized as a Weil restriction of another datum
(see Definition \ref{defn:weil-restriction-parascopy}).
As a consequence, one can express the associated conorm $\dnorm$
in terms of simpler conorms
(see Propositions
\ref{prop:decomposition-gamma},
\ref{prop:decomp-fixed-pinning},
\ref{prop:product-conorm},
and \ref{lem:conorm-restriction-over-k}
for example).

Second, the $k$-morphism $\dnorm$ of varieties of geometric semisimple conjugacy classes
could be compatible with
a $k$-morphism of varieties $\map{\dnormfunc}{G^\wedge}{\tilde{G}^\wedge}$
that specializes to a map on all geometric conjugacy classes.
A morphism of the latter kind gives us an explicit formula for $\dnormst$.
Moreover,
in the Lusztig classification of representations of finite reductive groups,
each Lusztig series is further partitioned into ``families'',
which can be parametrized by appropriate conjugacy classes
in dual groups.
(See \cite{lusztig:chars-finite}.)
We would eventually like to lift these families of representations, and
having a lifting of all geometric 
conjugacy classes (in particular, those unipotent classes that commute with a given
semisimple element)
seems to be
a necessary step in this program.

We now consider some situations where we can factor $\dnorm$
as mentioned above.
First, given a subnormal series for $\Gamma$,
one can express the map $\dnorm$ as a composition of maps
associated to the subquotients of $\Gamma$
(Proposition \ref{prop:decomposition-gamma}).
Thus, it is enough to understand the situations where $\Gamma$
is simple or trivial.

Second, suppose we fix a parascopic datum
for $(\tilde G, \Gamma, G)$,
and consider all groups $\univ G$ such that our datum is parascopic
for $(\tilde{G},\Gamma,\univ G)$.
If $\tilde{G}$ is $k$-quasisplit, then
under some hypotheses on the action of $\Gamma$,
there is a choice of $\univ{G}$ that is universal in the sense that
our datum for $(\tilde{G},\Gamma,G)$ always factors in a canonical way through our datum for
$(\tilde{G},\Gamma,\univ{G})$.
(See Proposition \ref{prop:decomp-fixed-pinning}.)
Roughly, speaking, $\univ{G}$
is the group we obtain if we extend the action of $\Gamma$ to one on all of $\tilde{G}$
that fixes a pinning
and consider the connected part of the group of fixed points
of this action.

Now let us consider a situation where we can make $\dnorm$ more
explicit by showing that it agrees with a morphism $\abmap{G^\wedge}{\tilde{G}^\wedge}$.
From the point of view of Shintani base change, an important case is where
$E/k$ is a finite, Galois extension,
$\Gamma=\Gal(E/k)$,
$\tilde{G} = \Weil_{E/k} G$,
and the action of $\Gamma$ on $\tilde{G}$ is induced by the Galois action.
That is, as groups over the separable closure $k\sep$ of $k$,
$\tilde{G}$ is a product of $|\Gamma|$ copies of $G$, and $\Gamma$
acts as a simple, transitive permutation group of the coordinates.
One thus has a $k$-homomorphism $\diag_{E/k}$,
analogous to the diagonal inclusion,
from $G$ to $\tilde{G}$, and similarly from $G^\wedge$ to $\tilde{G}^\wedge$,
and this latter map agrees with $\dnorm$
(Proposition \ref{prop:conorm-base-change}).
Suppose $k$ is finite.
Since families of irreducible representations of $G(k)$ are parametrized
by ``special'' conjugacy classes in $G^\wedge(k)$,
our function $\diag_{E/k}$ induces a lifting of families of representations
of $G(k)$ to those of $\tilde{G}(k)$.
Families are singletons in the cases where $G = \GL_n$
(see \cite{dprasad-sanat:restriction-cuspidal}*{Theorem 4.1})
or $\U_n$
(see \cite{lusztig-srinivasan:unitary}*{Corollary 2.4})
and we obtain a lifting of irreducible representations.
For these and related groups, MacDonald's correspondence \cite{macdonald:finite-gln}
associates such representations to certain combinatorial data,
for which there is a natural notion of base change.
It is easy to check that the lifting induced by the function
$\diag_{E/k}$ agrees with MacDonald's base change.
Therefore, our lifting agrees with Shintani lifting for groups of the above type
provided that MacDonald's base change does.
Silberger and Zink \cite{silberger-zink:explicit-base-change}
show that this is indeed the case when $G = \GL_n$.
We expect that our lifting is similarly compatible with Shintani base change
for other groups, though in general
base change might be a several-to-several map.

Our two phenomena---resolving $\dnorm$ into simpler conorms
and explicit realization of $\dnorm$ as a function on groups---come together in several interesting cases.
Suppose (to simplify the present discussion) that $\tilde{G}$ is almost simple,
$\Gamma$ is cyclic and acts on $\tilde{G}$,
$G$ is the identity component of the group of fixed points,
and no nontrivial element of $\Gamma$ acts via inner automorphisms.
As above, let $\univ{G}$ denote the connected part of the
group of fixed points of a related action of $\Gamma$ that fixes a pinning.
Then there is a 
natural embedding $G^\wedge \hookrightarrow \univ{G}^\wedge$,
and $\dnorm$ factors through it
(see
Proposition \ref{prop:quasi-central-dual-closed}
and
\S\ref{para:conorm-dual-closed}).
For example, suppose $\Gamma$ acts on $\tilde G = \GL(2n)$ via an involution,
and $G = \SO(2n)$.  Then $\univ{G} = \Sp(2n)$,
and Proposition~\ref{prop:decomp-fixed-pinning} says that our lifting of (packets of) representations
from $\SO(2n,k)$ to $\GL(2n,k)$ ($k$ finite)
must factor through a lifting from
$\SO(2n,k)$ to $\Sp(2n,k)$,
and this latter lifting is determined by an embedding of dual groups.
Other liftings have analogous factorizations:
\begin{equation}
\label{eq:nice-decomp}
\liftfactor{\SO(2n,k)}{\Sp(2n,k)}{\GL(2n,k)}
\liftfactor{\PGSp(8,k)}{F_4(k)}{E_6(k)}
\liftfactor{\PGL(3,k)}{G_2(k)}{\SO(8,k)}
\end{equation}

One expects
\cites{adler-lansky:bc-u3-unram,adler-lansky:bc-u3-ram}
that base-change liftings of depth-zero representations
of (the group of rational points of)
a reductive group $G$ over a $p$-adic field $F$
can be described in terms of liftings
of representations of certain reductive groups $\sfG$
(over the residue field $k$ of $F$)
that arise as quotients of parahoric subgroups of $G$.
In certain situations, these latter liftings cannot be base change,
but they nonetheless arise from considering a $k$-action of a finite
group $\Gamma$ on a reductive group $\tilde\sfG$,
and lifting representations from $\sfG(k)$
(with $\sfG = (\tilde\sfG^\Gamma)\conn$) to those of $\tilde\sfG(k)$.
Thus, even if one is only concerned with base change,
one is forced to consider group actions more generally.
In order to prove a general relationship between 
liftings of representations of finite and of $p$-adic groups,
a necessary first step is to show that 
for a finite, Galois extension $E/F$, the natural action of
$\Gamma:= \Gal(E/F)$
on the Weil restriction $\tilde{G}:= \Weil_{E/F}G$
leads to a $k$-action of $\Gamma$ on each associated $k$-group
$\tilde\sfG$,
and that the groups $\sfG$ and $G$ are then closely related.
More generally, one would like to replace $\Gal(E/F)$ by any
finite group $\Gamma$ of $F$-automorphisms satisfying reasonable properties.
We address all of these matters elsewhere
\cite{adler-lansky-spice:group-actions}.
Still more generally, one expects that under reasonable
tameness hypotheses, parascopic data for $p$-adic groups,
and the liftings that they determine, should be associated
to parascopic data (and liftings) for finite groups.
In generalizing \loccit to show this, 
the decompositions of the present paper will play a key role.

Liftings of (packets of) representations
arise in other contexts besides base change. Consider three examples.
First, consider ``quadratic unipotent'' blocks of representations,
i.e., those Lusztig series that are parametrized
by order-two semisimple elements in the dual group.
Srinivasan \cite{srinivasan:quadratic-unipotent}
gives a correspondence between the union
of all such blocks for all general linear and unitary groups
(over a finite field),
and the union of all such blocks for all symplectic groups,
and this correspondence satisfies some strong properties.
Second, if $k$ is $p$-adic
and $\theta$ is a $k$-involution of $\tilde{G}=\GL_n$,
then 
Sylvestre 
\cite{sylvestre:thesis}
gives a relation between the
(twisted) depth-zero supercuspidal characters of $\tilde{G}(k)$
and the
depth-zero supercuspidal characters of $G(k)$,
where $G = \tilde{G}^\theta$,
thus lifting representations from $G(k)$ to $\tilde{G}(k)$.
Third,
our setting overlaps with, but does not subsume, that of
\cite{kumar-lusztig-dprasad:characters}.
We believe
that we could generalize our definitions to cover their situation.
For now we have not pursued the matter because our main
result, concerning the explicit form of the conorm map,
is trivial in their situation almost by definition.

We note that the definition of a parascopic datum in the present paper (Definition~\ref{defn:parascopy})
differs from that in~\cite{adler-lansky:lifting}*{Definition 4.1}
in two essential ways
as detailed in Remark~\ref{rem:new-parascopy-defn}.
First,
our current definition implies
a condition that was implicitly assumed in 
Lemma~\xref{FGL-lem:multiple}~\loccit.
Second, 
in all other respects our new definition is more general,
and thus applies to a wider selection of pairs $(\tilde{G},G)$.
In addition, our new definition is much easier to state,
using our notion of a \emph{morphism}
of root data (Definition \ref{defn:datum-morphism}).

\textbf{Outline.}
We work over an arbitrary field $k$.
In \S\ref{sec:basic}, we establish our basic notation and define
the notions of \emph{morphisms} of root data,
and of \emph{subdata}.
We then recall the notion of \emph{parascopic datum} that establishes
a relationship between a root datum $\tilde\Psi$ with an action of
a finite group $\Gamma$, and another root datum $\Psi$,
where all root data are equipped with an action of the absolute
Galois group $\Gal(k)$ of $k$.
If $\tilde{G}$ and $G$ are connected reductive $k$-groups with maximal $k$-tori
$\tilde{T}$ and $T$, such that $\tilde\Psi$ and $\Psi$
are the root data associated to the pairs $(\tilde{G},\tilde{T})$
and $(G,T)$,
then
one thus establishes a relationship between these pairs.
(We then say that we have a \emph{parascopic datum for the triple
$(\tilde{G},\Gamma,G)$
relative to $\tilde{T}$ and $T$}.)

We introduce (Definition \ref{defn:parascopy}) two properties
that a parascopic datum might have.
Recall that our parascopic datum induces a $k$-map
$\abmap{T}{(\tilde T^\Gamma)\conn}$.
Roughly, a datum is
\emph{torus-inclusive} if this map
is an isomorphism,
so that we obtain an inclusion $\abmap{T}{\tilde{T}}$;
and \emph{root-inclusive} if, were $\Gamma$ to act on $\tilde{G}$,
the root system of $(G,T)$ would be contained in
that of $((\tilde{G}^\Gamma)\conn, (\tilde{T}^\Gamma)\conn)$.
Root-inclusivity makes some arguments easier, and it holds in most cases,
but the exceptions are interesting.
See \S\ref{para:wrong-subsystem} for the most basic example.
After introducing some further notation, we recall
(Theorem \ref{thm:lift-geometric-general})
the main result of \cite{adler-lansky:lifting}:
a ``conorm'' map from the set of stable semisimple conjugacy classes of
$G^\wedge(k)$ to those of $\tilde{G}^\wedge(k)$ that depends only on
(the equivalence class of) the parascopic datum.

We then turn to the problem of expressing a parascopic
datum in terms of simpler data.
A reasonable step is 
to understand the subdata of a given root datum,
which is essentially the same as understanding the subsystems
of a root system.
By induction, it would be enough to understand \emph{maximal} subdata.
Fortunately, we show
(Proposition \ref{prop:max-subdatum})
that a maximal subdatum $\Psi$ of a datum $\tilde\Psi$
must be either closed or ``dual-closed'', in the sense that the dual
$\Psi^\wedge$ is a closed subdatum of $\tilde\Psi^\wedge$.
The closed case is interesting because
if $\Psi = \Psi(G,T)$ is closed in
$\tilde\Psi = \Psi(\tilde{G},\tilde{T})$,
where both root data are equipped with compatible $\Gal(k)$-actions,
and both groups $\tilde{G}$ and $G$ are $k$-quasisplit,
then one can construct (Proposition \ref{prop:embedding})
a $k$-embedding $\map{\nu}{G}{\tilde{G}}$
such that the image $\nu(T)$ of $T$ in $\tilde{G}$ is stably conjugate
to $\tilde{T}$.
We will see in 
\S\ref{para:conorm-dual-closed}
that the dual-closed case is interesting because in this case
the conorm map
is essentially an embedding of dual groups
$\abmap{G^\wedge}{\tilde{G}^\wedge}$.

In \S\ref{sec:isogenies},
we show that conorms are compatible with isogenies and, more generally,
certain kinds of isotypies.
Thus, for the purposes of understanding conorms,
we can reduce to the case where $\tilde{G}$ is adjoint.
In \S\ref{sec:decomposition-defn},
we define the notion of a \emph{decomposition} of parascopic data,
and give some of its properties.
In \S\ref{sec:decomposition-gamma}, we show that
a parascopic datum can be expressed as a composition of simpler data
whenever $\Gamma$ has a normal subgroup,
allowing us to reduce to the case where $\Gamma$ is simple or trivial.
In \S\ref{sec:fixed-pinning}, we show that a root-inclusive
parascopic datum has a factorization of the kind illustrated
in \eqref{eq:nice-decomp}.
In particular, we obtain an action of $\Gamma$ on all of $\tilde{G}$
that fixes a pinning
(Proposition \ref{prop:decomp-fixed-pinning}).
This allows us to reduce the root-inclusive case to the case
where $\Gamma$ acts fixing a pinning.

In \S\ref{sec:conormfunc}, we introduce the notion of a
\emph{pointwise conorm},
a $k$-morphism of varieties $\abmap{G^\wedge}{\tilde G^\wedge}$
that agrees with our conorm as a map of conjugacy classes,
and we provide (Proposition \ref{prop:conorm-func-hom})
a convenient way to detect when a morphism has this property.
We apply this result in two ways.
First, we take care of some unfinished business:
determining a conorm for what will turn out to be the basic
example of a parascopic datum that is not root-inclusive
(\S\ref{para:wrong-subsystem}).
Second, we consider the case where $\tilde{G}$ is a direct product
(Proposition \ref{prop:product-conorm}).

In \S\ref{sec:weil-parascopic},
we define the notion of a \emph{Weil restriction} of a root datum
and of a parascopic datum,
and we consider the cases where a parascopic datum
arises via Weil restriction
(Lemma \ref{lem:parascopy-restriction-over-k})
or doesn't
(Proposition \ref{prop:non-induced}).
In \S\ref{sec:weil-conorm},
we consider the conorms that can arise when $\tilde{G}$ is a Weil restriction.

Suppose we have a parascopic datum for a triple $(\tilde{G},\Gamma,G)$.
We would like to express this datum, and thus the associated
conorm map, in terms of simpler special cases.
The reductions we have made above allow us to assume that
$\tilde{G}$ is absolutely simple,
and $\Gamma$ is simple or trivial.
In \S\ref{sec:reduction}
we further reduce the general case to a list of special cases.
With a small number of exceptions listed in Tables
\ref{table:quasi-central} and \ref{table:subdatum},
we then express the conorm in each case as a $k$-morphism
of dual groups.

\textbf{Acknowledgements.}
The authors thank Brian Conrad and Ryan Vinroot for helpful conversations.
In an earlier incarnation of this paper,
we used results of Steven Spallone and Rohit Joshi
\cite{joshi-spallone:spinoriality}*{Proposition 1 and Remark 1}.
We thank them not only for providing us a pre-publication version,
but also for strengthening their results in order to suit our
then-purposes.
We thank an anonymous referee for providing several corrections
and clarifications.
Both authors were partially supported by the National Science Foundation (DMS-0854844),
the National Security Agency (H98230-13-1-0202),
and by
Summer Faculty Research Awards
from the College of Arts and Sciences of American University.

\section{Notation and preliminaries}
\label{sec:basic}
Let $k$ be a field.
Let $\Psi = (X^*,\Phi,X_*,\Phi^\vee)$ be a root datum on which the
absolute Galois group $\Gal(k)$ of $k$ acts.
We will denote the Weyl group of $\Psi$ by $W(\Psi)$.

Denote by $\Psi\red$ the reduced root datum formed 
as in \cite{adler-lansky:data-actions}*{Remark 13}.
That is,
$\Psi\red = (X^*,\Phi\red,X_*,\Phi^\vee\red)$,
where $\Phi\red$ is the set of
non-divisible (resp.\ non-multipliable) roots
in $\Phi$ according as $\chr k$ is not two (resp.\ two),
and $\Phi^\vee\red$ is the set of coroots corresponding
to roots in $\Phi\red$.
Denote by $\Psi\redwrong = (X^*,\Phi\redwrong,X_*,\Phi^\vee\redwrong)$,
the reduced root datum formed
in the opposite way, i.e., the set of
non-divisible (resp.\ non-multipliable) roots
in $\Phi$ according as $\chr k$ is two (resp.\ not two).
Note that $\Psi\red$ and $\Psi\redwrong$
are $\Gal(k)$-stable.

\begin{defn}
\label{defn:restricted}
Suppose that $\map{\phi}{\Gamma}{\Aut(\Psi)}$ is a
$\Gal(k)$-equivariant action of a finite
group $\Gamma$ on $\Psi = (\bX^*, \Phi,\bX_*,\Phi^\vee)$,
and suppose that 
the action of $\Gamma$ preserves a system of positive roots in $\Phi$.
We define the
\emph{restricted root datum $\lsub{\phi}\Psi$ with respect to the action $\phi$}
as in 
\cite{adler-lansky:data-actions}*{\S2}.
Let $\lsub{\phi}X^*$ and $\lsub{\phi}\Phi$ be the
respective images of $ X^*$ and $\Phi$ under the quotient map
from $X^*\otimes\Q$ to its space of $\Gamma$-coinvariants $(X\otimes \Q)_{\phi(\Gamma)}$.
Then $\lsub{\phi}X^*$ and $\lsub{\phi}X_*:= X_*^{\phi(\Gamma)}$
are naturally in duality.
By~\cite{adler-lansky:data-actions}*{Lemma 6},
if $\bar\alpha\in\lsub{\phi}\Phi$, then the set of roots
in $\Phi$ that map to $\bar\alpha$
is a $\Gamma$-orbit $\Theta_{\bar\alpha}$.
By~\cite{adler-lansky:data-actions}*{Theorem 7},
we can form a coroot system $\lsub{\phi}\Phi^\vee\subseteq \bar X_*$,
where for each $\bar\alpha$,
the corresponding coroot $\bar\alpha^\vee$ is either
$\sum_{\theta\in\Theta_{\bar\alpha}}\theta^\vee $ or its double, the latter
case occurring only when the orbit $\phi(\Gamma)\cdot \alpha$ is not an orthogonal set. 
Then $\lsub{\phi}\Psi: = (\lsub{\phi}X^*,\lsub{\phi}\Phi, \lsub{\phi}X_*,\lsub{\phi}\Phi^\vee)$ is a
(not necessarily reduced) root datum with $\Gal(k)$-action. In situations where the action
$\phi$ is clear from context, we will denote $\lsub{\phi}\Psi$,
$\lsub{\phi}\Phi, \ldots$ by $\lsub{\Gamma}\Psi$, $\lsub{\Gamma}\Phi, \ldots$.
\end{defn}

We note that if $\Psi$ is reduced, then $\lsub{\phi}\Psi$ will be reduced unless $\Phi$ contains an irreducible component
of type $A_{2n}$ on which a subgroup of $\Gamma$ acts nontrivially, in which case
$\lsub{\phi}\Phi$ will contain a subsystem of type $BC_n$.

\begin{defn}
\label{defn:datum-morphism}
Suppose $\Psi = (\bX^*, \Phi, \bX_*, \Phi^\vee)$
and
$\dot\Psi = (\dot\bX^*, \dot\Phi, \dot\bX_*, \dot\Phi^\vee)$
are root data.
For us, a \emph{morphism} from $\Psi$ to $\dot\Psi$
will be a
map $\map{f^*}{\bX^*}{\dot\bX^*}$
such that $f^*(\Phi) \subseteq \dot\Phi $,
and such that the transpose map
$\map{f_*}{\dot\bX_*}{\bX_*}$
maps $f^*(\alpha)^\vee$ to $\alpha^\vee$ for all $\alpha \in \Phi$.
(In particular, it follows that $f^*$ must be injective on $\Phi$.)
If $f^*$ is the identity map, then we will say that $\Psi$ is a \emph{subdatum} of $\dot\Psi$,
and that $\dot\Psi$ is a \emph{superdatum} of $\Psi$.
\end{defn}

We note that the above notion of morphism is by no means standard.

\begin{example}
For every root datum $\Psi$, $\Psi\red$ is a subdatum of $\Psi$.
\end{example}

\begin{rem}
\label{rem:morphism-properties}
\ 
\begin{enumerate}[(a)]
\item
The composition of morphisms is again a morphism.
\item
Let $\Psi$ and $\dot\Psi$ be root data equipped with actions of a finite group $\Gamma$.
Let $f^*$ be a $\Gamma$-equivariant morphism from $\Psi$ to $\dot\Psi$. Then $f^*$
induces a morphism $\lsub{\Gamma}f^*$ from $\lsub{\Gamma}\Psi$ to $\lsub{\Gamma}\dot\Psi$.
If both $\Psi$ and $\dot\Psi$ are equipped with an action of $\Gal(k)$,
and $f^*$ is $\Gal(k)$-equivariant, then so is $\lsub{\Gamma}f^*$.
\item A morphism $f^*$ from $\Psi$ to $\dot\Psi$ induces an embedding $\map{i_{f^*}}{W(\Psi)}{W(\dot\Psi)}$
taking the reflection of $X^*$ through the root $\alpha\in\Phi$ to the reflection of $\dot X^*$ through
the root $f^*(\alpha)$.
\end{enumerate}
\end{rem}

Suppose $f$ is a morphism from $\Psi$ to $\dot\Psi$ as above.
If $f^*(X^*)^\perp$ denotes the annihilator of $f^*(X^*)$ in $\dot X_*$, then
$f_*$ descends to a map $\abmap{\dot X_*/f^*(X^*)^\perp}{X_*}$, which 
in turn extends to a map
$\map{\bar f_*}{(\dot X_*/f^*(X^*)^\perp)\otimes\Q}{X_*\otimes\Q}$.
The pairing between $\dot X^*$ and $\dot X_*$ determines a pairing between
$f^*(X^*)$ and $\bar f_*\inv(X_*)$ by which the latter two modules are in duality.
Letting $f^*(\Phi)^\vee = \{ (f^*\alpha)^\vee\mid \alpha\in\Phi \}$,
it is easily seen that since $f^*$ is a morphism,
the image $\overline{f^*(\Phi)^\vee}$ of $f^*(\Phi)^\vee$
in $\dot X^* / f^*(X^*)^\perp$ lies in $\bar f_*\inv(X_*)$.
Furthermore, it is straightforward to check that
$(f^*(X^*),f^*(\Phi),\bar f_*\inv(X_*), \overline{f^*(\Phi)^\vee})$
is a root datum. (A key observation toward verifying this statement is that
the reflection of $f^*(X^*)$ corresponding to a root in $f^*(\Phi)$
is the restriction to $f^*(X^*)$ of the reflection of $\dot X^*$ corresponding to this root.)
We refer to this root datum as the \emph{image $f^*(\Psi)$ of $f^*$}.

\begin{defn}
\label{defn:isotypy}
A homomorphism $\map{f}{\dot G}{G}$ between connected reductive groups
is an \emph{isotypy} provided that
the kernel of $f$ contains the center of $\dot G$ and the image of $f$ contains
the derived group of $G$.
Thus, $f$ restricts to give an isogeny
between the derived groups of $\dot G$ and $G$.
\end{defn}

\begin{rem}
\label{rem:morphism-example}
It is easy to see that every morphism
from $\Psi$ to $\dot\Psi$
can be written as a composition of morphisms
of the following forms. 
\begin{enumerate}[(a)]
\item
\label{item:morphism-example-isogeny-derived}
The function $f^*$ restricts to give a bijection from $\Phi$ to $\dot\Phi$.
In this case, there exist an isotypy $\map{f}{\dot G}{G}$ of connected reductive groups
and maximal tori $\dot T\subseteq \dot G$ and $T\subseteq G$
such that $f(\dot T)\subseteq T$ and $\Psi$ (resp.~$\dot\Psi$) is isomorphic to $\Psi(G,T)$ (resp. $\Psi(\dot G,\dot T)$)
in such a way that $f$ induces the morphism $f^*$ on root data.
If both $\Psi$ and $\dot\Psi$ are equipped with an action of $\Gal(k)$,
and $f^*$ is $\Gal(k)$-equivariant, then it follows from~\cite{adler-lansky:data-actions}*{Theorem 1}
that there exist $k$-structures on
$G$, $\dot G$, $T$, $\dot T$ such that $f$ is $k$-rational and the above isomorphisms of root data
are $\Gal(k)$-equivariant.
(Note that if $\Psi$ and $\dot\Psi$ have the same rank and $f^*$ is injective,
then $f^*$ is an isogeny of root data
in the usual sense~\cite{springer:corvallis}*{\S1.7},
and any associated map $f$ is in fact a central isogeny.)
\item
\label{item:morphism-example-embedding}
The function $f^*$ is an isomorphism of lattices,
and
we call $f^*$ an \emph{embedding} of root data.
In this case, $f^*(\Psi)$ is a subdatum of $\dot\Psi$.
\end{enumerate}
\end{rem}

\begin{rem}
\label{rem:morphism-pullback}
Root data along with the morphisms of Definition~\ref{defn:datum-morphism} form a category
for which images exist by the preceding discussion. This category also possesses pullbacks.
To see this, suppose $\dot\Psi=(\dot X^*,\dot\Phi, \dot X_*, \dot\Phi^\vee)$
is a root datum and for $i=1,2$, $f_i^*$ is a morphism from a 
root datum $\Psi_i=(X_i^*,\Phi_i, {X_i}_*, \Phi_i^\vee)$ to $\dot\Psi$.
Define the lattice $X^*$ (resp.~$X_*$) together with the maps
$\map{{g_i}^*}{X^*}{X_i^*}$ (resp.~$\map{{g_i}_*}{{X_i}_*}{X_*}$) for $i=1,2$
to be the pullback (resp.~pushout) of the $f_i^*$ (resp~${f_i}_*$).
Let $\Phi = {g_1^*}\inv(\Phi_1)\cap {g_2^*}\inv(\Phi_2)$,
and $\Phi^\vee = {g_1}_*(g_1^*(\Phi)^\vee)$ (which is equal to ${g_2}_*(g_2^*(\Phi)^\vee)$).
Then one can check that $\Psi = (X^*,\Phi,X_*,\Phi^\vee)$
is a root datum satisfying the universal property of a pullback for the $f_i$.
That $\Psi$ is a root datum follows from the fact that for $\alpha,\beta\in\Phi$,
$\langle\alpha,\beta^\vee\rangle = \langle f_i^*\alpha, (f_i^*\beta)^\vee\rangle$,
which can be used to show that the map $\abmap{\Phi}{\Phi^\vee}$ given by
$\alpha\mapsto {g_1}_*(g_1^*(\alpha)^\vee)$ is bijective,
and that $W(\Psi)$ is a subgroup of $W(\Psi_i)$ and is hence finite.

We note that if $\dot\Psi$ and the $\Psi_i$ carry actions of $\Gal(k)$
and the maps $f_i^*$ are $\Gal(k)$-equivariant,
then $\Psi$ inherits an action of $\Gal(k)$ and the maps $g_i^*$ are $\Gal(k)$-equivariant.
\end{rem}

For any reductive $k$-group $G$, denote by $G\conn$ the connected component of the
identity in $G$.
Let $\TTst(G,k)$ denote the set of stable conjugacy classes of maximal $k$-tori in $G$.
If $T$ is a maximal
$k$-torus of $G$, let $\Phi(G,T)$ denote the system of roots of $T$ in $G$,
$\Psi(G,T)$ the root datum of $G$ with respect to $T$,
$W(G,T)$ the Weyl group of $T$ in $G$,
$\bX^*(T)$ and $\bX_*(T)$ the character and cocharacter modules of $T$.
All of these objects come equipped with actions
of the absolute Galois group $\Gal(k)$ of $k$.
Any homomorphism $\map f{T}{T'}$ of
tori determines  maps $\map {f^*}{\bX^*(T')}{\bX^*(T)}$
and $\map {f_*}{\bX_*(T)}{\bX_*(T')}$.

\section{Basic results on parascopic data and conorms}

\begin{defn}
\label{defn:parascopy}
Let
$\Psi = (\bX^*, \Phi,\bX_*,\Phi^\vee)$
and
$\tilde\Psi = (\tilde\bX^*, \tilde\Phi,\tilde\bX_*,\tilde\Phi^\vee)$
be 
abstract root data with $\Gal(k)$-actions, and $\Gamma$ a finite group.
A \emph{parascopic datum} for the triple $(\tilde\Psi,\Gamma,\Psi)$ is a pair
$(\phi,q^*)$, where
\begin{itemize}
\item
$\phi$ is a homomorphism from $\Gamma$ to $\Aut (\tilde\Psi)$,
such that $\phi(\Gamma)$ commutes with the action of $\Gal(k)$ and preserves some system of
positive roots in $\tilde\Phi$; and
\item
$q^*$ is a $\Gal(k)$-invariant morphism from $\Psi$ to $\lsub{\phi}\tilde\Psi$.
\end{itemize}
If $q^*$ is an embedding of root data, we say that $(\phi,q^*)$ is \emph{torus-inclusive}.
If $q^*$ is actually a morphism from $\Psi$ to $\lsub{\phi}\tilde\Psi\red$, we say that our datum is \emph{root-inclusive}.
\end{defn}

\begin{rem}
\label{rem:new-parascopy-defn}
As noted in the Introduction, Definition~\ref{defn:parascopy} is somewhat different from 
the definition of a parascopic datum given in \cite{adler-lansky:lifting}*{Definition 4.1},
which consisted of the action $\phi$ on $\tilde\Psi$ together with an embedding
$\map{j_*}{X_*\otimes\Q}{\lsub\phi \tilde X_*\otimes\Q}$ satisfying 
the property (\textbf{P1}) that $j\inv_*(\lsub\phi \tilde X_*)\subseteq X_*$
and a property (\textbf{P2}) that amounts to the condition that ${j^*}\inv(\Phi)\subseteq \lsub{\phi}\tilde\Phi$.

To explain the relationship between these definitions,
suppose $(\phi,q^*)$ is a parascopic datum in the sense of Definition~\ref{defn:parascopy}.
If $q^*$ is a full-rank embedding of lattices, then it follows from the fact that $q^*$ is a morphism that
$j_* := q_*\inv$ satisfies \textbf{P1} and \textbf{P2}.
The definition of morphism implies, however, that $j_*$ satisfies the \textit{additional} condition on \emph{coroots} that
$j_*(\alpha^\vee) = {j^*}\inv(\alpha)^\vee$ for $\alpha\in\Phi$,
which does not follow from \textbf{P1} and \textbf{P2}.
It happens though that much of the theory developed in \cite{adler-lansky:lifting} depends 
upon Lemma~\xref{FGL-lem:multiple} \loccit, which implicitly requires
the preceding condition to hold, and so it should have appeared as an additional property in Definition 4.1 \loccit
This does not present any real difficulty, however,
as all of the examples and constructions of parascopic data discussed in that paper 
(e.g., Examples 4.2, and Propositions 7.4 and 7.5\ \loccit)
are easily seen to satisfy this additional condition.
In Example~\ref{ex:parascopy-fixed-point},
we verify this for one important class of parascopic data (those from Example 4.2(a)\ \loccit)
by showing that the map $q^*$ arising in this setting is a morphism of root data.

Thus when $q^*$ is a full-rank embedding $X\rightarrow \lsub\phi\tilde X$,
the condition that it is a morphism from $\Psi$ to $\lsub{\phi}\tilde\Psi$ is equivalent to 
its inverse transpose $j_*$ possessing precisely the properties necessary
for the results of \cite{adler-lansky:lifting} to hold.
Even further, it is straightforward to check that the constructions in
\cite{adler-lansky:lifting} do not make use of $j_*$ directly, but rather 
of its inverse transpose, namely $q^*$. Moreover, they remain valid whether or not $q^*$ is assumed to be
full-rank or even injective, and hence there is no need to require this in Definition~\ref{defn:parascopy}.
\end{rem}

\begin{example}
\label{ex:parascopy-gamma-trivial}
Suppose $\Gamma$ is the trivial group.
Use ``$1$'' to denote both $\Gamma$ and the map
from $\Gamma$ to any automorphism group.
Let $\Psi$ and $\tilde\Psi$
denote root data with $\Gal(k)$-action,
and let $q^*$ be a $\Gal(k)$-equivariant morphism from $\Psi$ to $\tilde\Psi$.
Then $(1,q^*)$ is a parascopic datum for $(\tilde\Psi,1,\Psi)$ that is torus-inclusive when
$q^*$ is an embedding.
\end{example}

\begin{example}
\label{ex:parascopy-group-action}
Let $\tilde\Psi$ be a root datum with $\Gal(k)$-action,
and let $\Psi$ be any $\Gal(k)$-invariant subdatum of $\lsub{\phi}\tilde\Psi$.
Let $\id$ denote the identity map on $X^* = \lsub{\phi}\tilde X^*$.
Then $(\phi,\id)$ is a torus-inclusive parascopic datum for 
$(\tilde\Psi,\Gamma,\Psi)$. If $\Psi$ is a subdatum of $\lsub{\phi}\tilde\Psi\red$, then $(\phi,\id)$ is also root-inclusive.
\end{example}

\begin{example}
\label{ex:bad}
If $\tilde\Phi$ is reduced and has no factor of type $A_{2n}$ on which $\Gamma$
acts nontrivially, then $\lsub{\phi}\tilde\Phi$ is reduced by
\cite{kottwitz-shelstad:twisted-endoscopy}*{\S1.3},
so any datum $(\phi,q^*)$ is root-inclusive. See \S\ref{para:wrong-subsystem} for a basic example
of a parascopic datum that is not root-inclusive.
\end{example}

\begin{defn}
\label{defn:parascopy-groups}
Let $G$ and $\tilde{G}$ be connected reductive $k$-groups, and $\Gamma$ a finite group.
Let $T$ (resp.~$\tilde T$) be a maximal $k$-torus of $G$ (resp.~$\tilde G$).
The root data $\Psi(G,T)$ and $\Psi(\tilde{G},\tilde{T})$
come equipped with an action of $\Gal(k)$.
We will refer to a parascopic datum $(\phi,q^*)$ for the triple $(\Psi(\tilde G,\tilde T),\Gamma,\Psi(G,T))$
as a \emph{parascopic datum for $(\tilde{G},\Gamma,G)$
relative to the tori $\tilde T\subseteq \tilde{G}$ and $T\subseteq G$}.
We will say that $G$ is
a \emph{weakly parascopic group} for the pair $(\tilde{G},\Gamma)$
if such a parascopic datum exists, and we will feel free not to specify a particular datum if
it is clear from the context.
\end{defn}

\begin{example}
\label{ex:parascopy-fixed-point}
The whole point of parascopy was to axiomatize the essential features of the situation where
$\Gamma$ acts on $\tilde{G}$ via $k$-automorphisms, all of which preserve
some positive root system for some maximal torus of $\tilde{G}$,
and $G = (\tilde{G}^\Gamma)\conn$, as in~\cite{adler-lansky:lifting}*{\S3}.
In this situation, we can choose maximal $k$-tori $T\subseteq G$ and $\tilde{T}\subseteq \tilde{G}$
such that $\Gamma$ preserves $\tilde{T}$, and $T = (\tilde{T}^\Gamma)\conn$.
The induced action $\phi$ of $\Gamma$ on $\Psi(\tilde G,\tilde T)$ gives rise to the restricted root datum
$\lsub{\phi}\Psi(\tilde G,\tilde T) =
(\lsub{\phi}X^*(\tilde T),\lsub{\phi}\Phi(\tilde G,\tilde T), \lsub{\phi}X_*(\tilde T), \lsub{\phi}\Phi^\vee(\tilde G,\tilde T)).$
There is a natural isomorphism $\map{q^*}{X^*(T)}{\lsub{\phi}X^*(\tilde T)}$;
namely, for $\chi\in X^*(T)$, define $q^*(\chi)$ to be the image in $\lsub{\phi}X^*(\tilde T)$
of any extension of $\chi$ to $\tilde T$.
We claim that $q^*$ is a $\Gal(k)$-invariant embedding of $\Psi(G,T)$ in $\lsub{\phi}\Psi(\tilde G,\tilde T)$
and thus (as indicated in Example~\ref{ex:parascopy-group-action})
that $(\phi,q^*)$ is a parascopic datum for $(\tilde G,\Gamma ,G)$ relative to $\tilde T $ and $T$
(cf.~\cite{adler-lansky:lifting}*{Example \xref{FGL-ex:fixed-point}}).

First note that $q^*(\Phi(G,T))\subseteq \lsub\phi\Phi(\tilde G,\tilde T)$ since every root $\alpha\in\Phi(G,T)$ is the restriction
of some root $\tilde\alpha\in\Phi(\tilde G,\tilde T)$ \cite{adler-lansky:lifting}*{Proposition 3.5(iv)}.
To conclude that $q^*$ is a morphism of root data, it remains to show that
$q_*(q^*(\alpha))^\vee = \alpha^\vee$ for all $\alpha\in\Phi(G,T)^\vee$.
As defined in~\cite{adler-lansky:data-actions}*{\S2},
$q^*(\alpha)^\vee$ is a multiple of $\sum_{\gamma\in\Gamma}\phi(\gamma)\cdot\tilde\alpha^\vee$.
We may identify $\lsub{\phi}X_*(\tilde T)$ with $X_*(T)$ via $q_*$ (which simply restricts codomains of cocharacters
from $\tilde T$ to $T$), and under this identification,
it suffices to show that the coroot $\alpha^\vee\in\Phi(G,T)^\vee\subset X_*(T)$ is a multiple of the preceding sum.
(That the particular multiples coincide must then follow from the fact that
$\langle\alpha,\alpha^\vee\rangle = 2 = \langle q^*\alpha,q_*(\alpha^\vee)\rangle$.)

To verify this, note that the one-dimensional root group $U_\alpha\subset G$ associated to $\alpha$
is contained in the group generated by the root groups
$U_{\phi(\gamma)\cdot\tilde\alpha}\subset\tilde G$ for $\gamma\in\Gamma$.
(This is a result of~\cite{steinberg:endomorphisms}*{\S8}, which pertains to the case
in which $\Gamma$ is cyclic, together with an argument similar to that in the proof
of~\cite{adler-lansky:lifting}*{Proposition 3.5} that reduces to the cyclic case.)
Thus the coroot $\alpha^\vee$ must lie in the span of $\phi(\Gamma)\cdot \tilde\alpha^\vee$
in $V_*(T) = V_*(\tilde T)^{\phi(\Gamma)}$.
It is well known that $\phi(\Gamma)\cdot \tilde\alpha$, and hence
$\phi(\Gamma)\cdot\tilde\alpha^\vee$, are linearly independent
(see e.g.,~\cite{kottwitz-shelstad:twisted-endoscopy}*{\S1.3}). Since $\alpha^\vee$ is $\phi(\Gamma)$-fixed,
it follows that $\alpha^\vee$ must be a scalar multiple of
$\sum_{\gamma\in\Gamma}\phi(\gamma)\cdot\tilde\alpha^\vee$.

We note moreover, that $(\phi,q^*)$ is torus- and root-inclusive.
Torus inclusivity follows from the fact that $q^*$ is an isomorphism of lattices.
Root inclusivity, i.e., that $q^*(\Phi(G,T))\subseteq\lsub{\phi}\Phi(\tilde G,\tilde T)\red$,
follows from \cite{steinberg:endomorphisms}*{\S8.2(2$''''$)} when $\Gamma$ is cyclic;
one can reduce to the cyclic case using the argument in the proof of \cite{adler-lansky:data-actions}*{Lemma 28}.
Furthermore, it follows from \loccit that the image of this embedding is equal to
$\lsub{\phi}\Phi(\tilde G,\tilde T)\red$
when $\Gamma$ fixes a pinning with respect to some set of simple roots in $\Phi(\tilde G,\tilde T)$.
In this case, we will sometimes identify the root data $\Phi(G,T)$ and $\lsub{\phi}\Phi(\tilde G,\tilde T)\red$.
\end{example}

\begin{para}[Embeddings of Weyl groups]
\label{para:weyl}
Suppose that $\phi$ is a $\Gal(k)$-equivariant action of the finite group $\Gamma$
on the root datum $\tilde\Psi$.  Then $\Gamma$ acts on the Weyl group $W(\tilde\Psi)$, and
by~\cite{adler-lansky:lifting}*{Proposition \xref{FGL-prop:weyl-embedding}},
there is a natural embedding
$\map{i_\phi}{W(\lsub{\phi}\tilde\Psi)}{W(\tilde\Psi)}$ with
$\im(i_\phi) = W(\tilde\Psi)^{\phi(\Gamma)}$
defined by the property that for any
$\alpha\in\lsub{\phi}\tilde\Phi$, the image of the reflection
$w_\alpha$ under $i_\phi$ is $\prod_{\beta\in\Xi_\alpha}w_\beta$,
where $\Xi_\alpha$ is a certain $\phi(\Gamma)$-orbit of orthogonal roots in $\tilde\Phi$ attached to $\alpha$.
Namely, $\Xi_\alpha$ is the set of roots in $\tilde\Phi$ with image in $\lsub{\phi}\tilde\Phi$ equal to the
largest positive multiple $c_\alpha\alpha$ of $\alpha$ in $\lsub{\phi}\tilde\Phi$. The
scalar $c_\alpha$ must be either $1$ or $2$, and by~\cite{kottwitz-shelstad:twisted-endoscopy}*{\S1.3},
$c_\alpha = 2$ if and only if the $\phi(\Gamma)$-orbit of roots in $\tilde\Phi$
that map to $\alpha$ is not an orthogonal set.
This can occur only if some (hence every) $\beta\in\Xi_\alpha$ lies inside an irreducible component of
$\tilde\Phi$ of type $A_{2n}$ on which some element of $\Gamma$ acts nontrivially, fixing $\beta$.
In this case, $\alpha$ must lie inside an irreducible component of $\lsub{\phi}\tilde\Phi$ of type
$BC_n$.

If $(\phi,q^*)$ is a parascopic datum for $(\tilde\Psi,\Gamma,\Psi)$, then we obtain a
$\Gal(k)$-equivariant embedding $\map{i_{\phi,q^*}}{W(\Psi)}{W(\tilde\Psi)}$
by composing $i_\phi$ with the embedding
$\map{i_{q^*}}{W(\Psi)}{W(\lsub{\phi}\tilde\Psi)}$
of Remark~\ref{rem:morphism-properties}(c).
\end{para}

To define the notion of a parascopic group, we need to consider the cohomological parametrization of stable conjugacy classes of tori in reductive groups. The following result is~\cite{adler-lansky:lifting}*{Proposition 6.1}.
\begin{prop}
\label{prop:stable-tori}
Let $G$ denote a connected reductive group over a field $k$.
Let $S$ be a maximal $k$-torus of $G$.
Then there is a natural injection $\abmap{\TTst(G,k)}{H^1(k,W(G,S))}$.
This map is a surjection when $G$ is $k$-quasisplit.
\end{prop}

\begin{defn}
\label{defn:parascopic-group}
Suppose that $(\phi,q^*)$ is a parascopic datum for $(\tilde{G},\Gamma,G)$ relative
to $\tilde{T}$ and $T$.
Thus, $G$ is weakly parascopic for $(\tilde{G},\Gamma)$.
Identify $\TTst(G,k)$ and $\TTst(\tilde{G},k)$ with subsets
of $H^1(k,W(G,T))$ and $H^1(k,W(\tilde{G},\tilde{T}))$, respectively, via the injection
of Proposition \ref{prop:stable-tori}. The embedding of $W(G,T)$ in $W(\tilde G,\tilde T)$
from \S\ref{para:weyl} determines a map
\begin{equation}
\label{eq:H1-lifting}
\abmap{H^1(k,W(G,T))}{H^1(k,W(\tilde{G},\tilde{T}))}.
\end{equation}
We will say that the parascopic datum $(\phi,q^*)$ is \emph{strong} for $(\tilde{G},\Gamma,G)$,
if the map \eqref{eq:H1-lifting} restricts to a map $\abmap{\TTst(G,k)}{\TTst(\tilde{G},k)}$.
In this situation, we say that $G$ is a \emph{parascopic group} for $(\tilde{G},\Gamma)$.
\end{defn}

\begin{defn}
\label{defn:equivalent-data}
Let $(\phi',{q'}^*)$ denote another parascopic datum for
$(\tilde G,\Gamma,G)$,
this time relative to the maximal $k$-tori $T'\subseteq G$ and $\tilde{T}'\subseteq \tilde{G}'$.
Adapting~\cite{adler-lansky:lifting}*{Definition 7.3} to our new definition of parascopic datum,
we say that the data $(\phi,q^*)$ and $(\phi',{q'}^*)$ are
\emph{equivalent}
if there exist elements $g\in G(k\sep)$ and $\tilde{g}\in \tilde{G}(k\sep)$ satisfying:
\begin{enumerate}[(a)]
\item
$\lsup g T  = T'$ and $\lsup {\tilde{g}} \tilde{T} = \tilde{T}'$.
\item 
\label{item:phi-compatible}
For all $\gamma\in \Gamma$,
$\phi'(\gamma) = {\Int(\tilde g)^*}\inv \circ \phi(\gamma) \circ \Int(\tilde{g})^*$,
i.e., $\Int(\tilde g)^*$ is $\Gamma$-equivariant.
\item
\label{item:j-compatible}
${q'}^* = \lsub{\Gamma}{\Int(\tilde g)^*}\inv \circ q^* \circ \Int(g)^*$, where $\lsub{\Gamma}\Int(\tilde g)^*$
is the isomorphism from $\lsub{\phi'}\Psi(\tilde G,\tilde T')$ to $\lsub{\phi}\Psi(\tilde G,\tilde T)$ induced by
$\Int(\tilde g)^*$ as in Remark~\ref{rem:morphism-properties}(b).
\end{enumerate}
The equivalence class of $(\phi,q^*)$ will be denoted $[\phi,q^*]$.
\end{defn}

\begin{rem}
\label{rem:equivalence}
Suppose we are given $(\phi,q^*)$ and elements $g\in G(k\sep)$ and $\tilde{g}\in\tilde G(k\sep)$ such that
the maximal tori $\lsup g T\subseteq G$ and $\lsup{\tilde g}\tilde T\subseteq\tilde G$ are defined over $k$.
Then for any $\sigma\in\Gal(k)$, $g\inv\sigma(g)$ must normalize $T$,
and $\tilde g\inv\sigma(\tilde g)$ must normalize $\tilde T$.
If the pair $(\phi',{q'^*})$ is \emph{defined} as in (a), (b), and (c) above,
we will say that \emph{$(\phi',{q'}^*)$ is the conjugate of $(\phi,q^*)$ via the elements $g$ and $\tilde{g}$.}
If so, then by~\cite{adler-lansky:lifting}*{Lemma 7.4}, $(\phi',{q'^*})$ is a parascopic datum
for $(\tilde G,\Gamma,G)$ relative to $\lsup{\tilde g}\tilde T$ and $\lsup g T$ (necessarily equivalent to $(\phi,q^*)$)
if and only if the images of $g\inv\sigma(g)$ in $W(G,T)$ and of $\tilde g\inv\sigma(\tilde g)$
in $W(\tilde G,\tilde T)$ correspond under the embedding of \S\ref{para:weyl} for all $\sigma\in\Gal(k)$.
\end{rem}

\begin{example}
In the situation of Example~\ref{ex:parascopy-fixed-point}, if $T_i\subseteq G$ and $\tilde{T}_i\subseteq \tilde{G}$
are maximal $k$-tori for $i=1,2$ such that $\tilde{T}_i$ is $\Gamma$-invariant and $T_i = (\tilde{T}_i^\Gamma)\conn$,
then the parascopic data that arise from the actions $\phi_i$ of $\Gamma$ on $\Psi(\tilde G,\tilde T_i)$ are
easily seen to be equivalent via elements $g$ and $\tilde g$ that can be taken to be equal.
\end{example}

\begin{para}[Dual groups, tori, and isogenies]
\label{para:duality}
Suppose that $G$ is a connected, reductive, quasisplit $k$-group,
$B\subseteq G$ is a Borel $k$-subgroup, and $T_0\subseteq B$
is a maximal $k$-torus. Let $G^\wedge$ be a connected, reductive $k$-group
with a maximal $k$-torus $T_0^\wedge$ such that there exists
a $\Gal(k)$-isomorphism $\delta_0$ from $\Psi(G,T_0)$ to
$\Psi(G^\wedge,T_0^\wedge)^\wedge: =
(\Phi^\vee(G^\wedge,T_0^\wedge),X_*(T_0^\wedge),\Phi(G^\wedge,T_0^\wedge),X^*(T_0^\wedge))$.
Then $G^\wedge$ is quasisplit, and we will say that $G^\wedge$ is in \emph{$k$-duality} with $G$
via $\delta_0$.
(See \cite{adler-lansky:lifting}*{\S\xref{FGL-sec:duality}}.)
Such a pair $(G^\wedge, T_0^\wedge)$ must exist by~\cite{springer:lag-article}*{Theorem 6.2.7}.
In the following, we will often suppress the objects $B$, $T_0$, $\delta_0$, etc., saying only that the group
$G^\wedge$ is in $k$-duality with $G$.

Suppose that $G$ is as above and $T\subseteq G$ is an arbitrary maximal $k$-torus.
It is straightforward to show that if $G^\wedge$ is \textit{any} connected, reductive, quasisplit $k$-group,
$T^\wedge\subseteq G^\wedge$ is a maximal $k$-torus, and
there exists a $\Gal(k)$-isomorphism $\delta$ from $\Psi(G,T)$ to
$\Psi(G^\wedge,T^\wedge)^\wedge$, then $G^\wedge$ is in $k$-duality with $G$.
We will refer to such a map $\map{\delta}{\bX^*(T)}{\bX_*(T^\wedge)}$ as a \textit{duality map}.

From \cite{adler-lansky:lifting}*{Proposition \xref{FGL-prop:dual-tori}},
there is a canonical bijection between $\TTst(G,k)$
and $\TTst(G^\wedge,k)$.
Moreover, if $T\subseteq G$ and $T^\wedge\subseteq G^\wedge$ correspond under this bijection,
then there is a duality map 
$\map{\delta_T}{\bX^*(T)}{\bX_*(T^\wedge)}$
that is uniquely determined up to the action of $W(G,T)^{\Gal(k)}$
on $T$.

Suppose $\map{f}{G'}{G}$ is a $k$-isotypy,
and let $T' = f\inv (T)$.
Then one obtains a $\Gal(k)$-morphism $f^*$ from $\Psi(G,T)$ to $\Psi(G',T')$.
Let $G'^\wedge$ be a quasi-split connected reductive $k$-group,
and $T'^\wedge$ a maximal $k$-torus in $G'^\wedge$,
such that $(G'^\wedge,T'^\wedge)$ is in $k$-duality with $(G',T')$
via a duality map $\map{\delta'}{X^*(T')}{X_*(T'^\wedge)}$.
Then the transpose ${f^*}^\wedge$ of
$\delta'\circ f^* \circ \delta\inv$
is a morphism from 
$\Psi(G'^\wedge,T'^\wedge)$
to
$\Psi(G^\wedge,T^\wedge)$.
By Remark \ref{rem:morphism-example}(\ref{item:morphism-example-isogeny-derived}),
we thus obtain a $k$-isotypy $\map{f^\wedge}{G'^\wedge}{G^\wedge}$,
uniquely determined only up to toral inner $k$-automorphisms
of $G^\wedge$,
which we will refer to as an \emph{isotypy dual to $f$} (with respect to
$\delta$ and $\delta'$).
\end{para}

\begin{para}[Tori in parascopic groups and their duals]
\label{para:square-of-tori}
Let $\tilde G$ and $G$ be connected reductive quasisplit $k$-groups, $\Gamma$
a finite group, and $\tilde T$ and $T$ maximal $k$-tori in $\tilde G$ and $G$. Let
$(\phi,q^*)$ be a strong parascopic datum for $(\tilde{G},\Gamma,G)$ relative to $\tilde{T}$ and $T$.

Let $G^\wedge$ and $\tilde G^\wedge$ be $k$-groups in $k$-duality respectively with $G$ and $\tilde G$
as in \S\ref{para:duality}.
We then have canonical bijections $\abmap{\TTst(G,k)}{\TTst(G^\wedge,k)}$
and $\abmap{\TTst(\tilde{G},k)}{\TTst(\tilde{G}^\wedge,k)}$. Since $(\phi,q^*)$ is strong,
it determines a map $\abmap{\TTst(G,k)}{\TTst(\tilde{G},k)}$.
Thus, given a maximal $k$-torus $S$ in $G$, we obtain maximal $k$-tori
$S^\wedge\subseteq G^\wedge$, $\tilde{S}\subseteq \tilde{G}$, and
$\tilde{S}^\wedge\subseteq \tilde{G}^\wedge$, all uniquely determined up to stable conjugacy.
Moreover, by
\cite{adler-lansky:lifting}*{Proposition \xref{FGL-prop:parascopy-equivalence}(\xref{FGL-item:all-tori-equally-good})},
we can find a parascopic datum $(\phi',{q'}^*)$ equivalent to $(\phi,q^*)$
(as in Definition~\ref{defn:equivalent-data}), relative to $S$ and $\tilde S$.

Conversely, given $S^\wedge$, we obtain $S$, $\tilde{S}$, and $\tilde{S}^\wedge$ and a
parascopic datum relative to $S$ and $\tilde S$ that is equivalent to $(\phi,q^*)$.
\end{para}

\begin{para}[Norm and conorm maps on tori]
\label{para:conorm-defn}
Retain the assumptions and notation of \S\ref{para:square-of-tori}.  
Since the endomorphism of $X^*(\tilde T)$ given by $x\mapsto\sum_{\gamma\in\Gamma}\, \phi(\gamma)\cdot x$ 
is constant on $\Gamma$-orbits, it factors through a $\Gal(k)$-equivariant map
$\map{\normchar[\phi]}{\lsub{\phi}X^*(\tilde T)}{X^*(\tilde T)}$. We define the norm map
$\map{\normchar[\phi,q^*]}{X^*(T)}{X^*(\tilde T)}$ to be
$\normchar[\phi]\circ q^*$. 
Thus we have a commutative diagram
\begin{equation}
\label{eqn:norm}
\begin{xy}
\xymatrix{
X^*(\tilde T) \ar[r]^{\sum\phi(\gamma)} \ar@{->>}[d] & X^*(\tilde T)  \\
\lsub{\phi}X^*(\tilde T) \ar[ur]^{\normchar[\phi]} \\
X^*(T) \ar[u]_{q^*} \ar[uur]_{\normchar[\phi,q^*]} 
}
\end{xy}
\end{equation}

Let $T^\wedge\subseteq G^\wedge$ and $\tilde T^\wedge\subseteq \tilde G^\wedge$ be maximal tori, and let $\map{\delta}{\bX^*(T)}{\bX_*(T^\wedge)}$ and $\map{\tilde\delta}{\bX^*(\tilde{T})}{\bX_*(\tilde{T}^\wedge)}$
be duality maps as in \S\ref{para:duality}.
These maps are uniquely determined up to conjugacy by
$W(G,T)^{\Gal(k)}$ and $W(\tilde{G},\tilde{T})^{\Gal(k)}$, respectively.
We obtain a $\Gal(k)$-equivariant homomorphism
\begin{equation}
\label{eq:norm-char}
(\map{\dnorm[\phi,q^*,\tilde\delta,\delta])_*
=
\tilde\delta\circ \norm[\phi,q^*]^* \circ \delta\inv }{\bX_*(T^\wedge)}{\bX_*(\tilde{T}^\wedge)},
\end{equation}
which corresponds to a $k$-homomorphism
\begin{equation}
\label{eq:conorm-points}
\map{\dnorm[\phi,q^*,\tilde\delta,\delta]}{T^\wedge}{\tilde{T}^\wedge}
\end{equation}
that we call a \emph{conorm map}.
When the duality maps $\tilde\delta$ and $\delta$ are understood from 
the context, we will feel free to omit them
from the notation, instead writing
$(\dnorm[\phi,q^*])_*$
and
$\dnorm[\phi,q^*]$.
Furthermore, 
when the parascopic datum $(\phi,q^*)$
is understood from the context,
we will denote 
$\norm[\phi,q^*]^*$ by $\norm[T]^*$;
$(\dnorm[\phi,q^*])_*$ by $\dnorm[T^\wedge,*]$;
and
$\dnorm[\phi,q^*]$ by $\dnorm[T^\wedge]$.
\end{para}

\begin{notation}
\label{notation:classes}
For $H$ a connected reductive $k$-group
and $s\in H(k)$ a semisimple element,
let $\Cl(H)$ denote the $k$-variety of geometric semisimple
conjugacy classes of $H$, and let $[s]$ denote the geometric conjugacy class of $s$ in $\Cl(H)$.
Similarly, let $\Cl\st(H)$ denote the set of stable 
(in the sense of Kottwitz~~\cite{kottwitz:rational-conj})
semisimple conjugacy classes in $H(k)$,
and let $[s]\st$ denote the stable conjugacy class of $s$ in $\Cl\st(k)$.
\end{notation}

The main result of~\cite{adler-lansky:lifting} is the following.

\begin{thm}
\label{thm:lift-geometric-general}
Let $(\phi,q^*)$ be a strong parascopic datum for the triple $(\tilde{G},\Gamma,G)$.
The equivalence class $[\phi,q^*]$ of $(\phi,q^*)$ determines a $k$-morphism
$\map{\dnorm[{[\phi,q^*]}]}{\Cl(G^\wedge)}{\Cl(\tilde G^\wedge)}$
and a function $\map{\dnormst[{[\phi,q^*]}]}{\Cl\st(G^\wedge)}{\Cl\st(\tilde G^\wedge)}$
such that for any
\begin{itemize}
\item maximal $k$-torus $S^\wedge\subseteq G^\wedge$,
\item maximal $k$-tori $S\subseteq G$, $\tilde S\subseteq \tilde G$, and $\tilde S^\wedge\subseteq \tilde G^\wedge$
compatible with $S^\wedge$ as in \S\ref{para:square-of-tori},
\item duality maps $\map{\delta}{X^*(S)}{X_*(S^\wedge)}$ and
$\map{\tilde\delta}{X^*(\tilde S)}{X_*(\tilde S^\wedge)}$, and
\item parascopic datum $(\phi',{q'}^*)$ relative to $\tilde S$ and $S$ equivalent to $(\phi,q^*),$
\end{itemize}
we have
$\dnorm[\phi',q^{\prime *},\tilde\delta,\delta](s)\in \dnorm[{[\phi,q^*]}]([s])$
and $\dnorm[\phi',q^{\prime *},\tilde\delta,\delta](s)\in \dnormst[{[\phi,q^*]}]([s]\st)$
for all $s\in S^\wedge(\bar k)$.
\end{thm}

\begin{para}[Change of group action]
\label{para:group-action}
Keeping the assumptions of Theorem \ref{thm:lift-geometric-general},
suppose that $\phi^\sharp$ is a $\Gal(k)$-equivariant action of a finite group $\Gamma^\sharp$
on the root datum $\tilde\Psi = \Psi(\tilde G,\tilde T)$.
Suppose that the endomorphisms $\psi = \sum_{\gamma\in\Gamma}\phi(\gamma)$ and
$\psi^\sharp = \sum_{\gamma\in\Gamma^\sharp}\phi^\sharp(\gamma)$
of $X^*(\tilde T)$ agree. It is straightforward to check that the kernel
of the quotient map from $X^*(\tilde T)\otimes\Q$ to the coinvariant space
$(X^*(\tilde T)\otimes\Q)_{\phi(\Gamma)}$
(resp.~to $(X^*(\tilde T)\otimes\Q)_{\phi^\sharp(\Gamma)}$)
is equal to the kernel of $\psi$ (resp.~$\psi^\sharp$) on $X^*(\tilde T)\otimes\Q$.
Then by Definition~\ref{defn:restricted},
$\lsub{\phi} X^*(\tilde T) = \lsub{\phi^\sharp} X^*(\tilde T)$
and $\lsub{\phi} \Phi(\tilde G, \tilde T) = \lsub{\phi^\sharp} \Phi(\tilde G, \tilde T)$.
It follows that the restricted root data
$\lsub\phi\tilde\Psi$ and $\lsub{\phi^\sharp}\tilde\Psi$
are the same, and hence that 
the maps $\normchar[\phi]$ and $\normchar[\phi^\sharp]$ from 
\S\ref{para:conorm-defn} agree.
Moreover, if $(\phi,q^*)$ is a (strong) parascopic datum for $(\tilde G, \Gamma,G)$
relative to the tori $\tilde T \subseteq \tilde G$ and $T\subseteq G$,
then $(\phi^\sharp, q^*)$ is a (strong) parascopic datum for $(\tilde G, \Gamma^\sharp, G)$
relative to the same tori.
\end{para}

\begin{prop}
\label{prop:change-group-action}
Under the above assumptions,
$
\dnorm[{[\phi,q^*]}]
=
\dnorm[{[\phi^\sharp,q^*]}]
$
and
$
\dnormst[{[\phi,q^*]}]
=
\dnormst[{[\phi^\sharp,q^*]}]
$.
\end{prop}

\begin{proof}
Let $T'\subseteq G$ be a maximal $k$-torus, and
let $\tilde T' \subseteq\tilde G$ be a maximal $k$-torus
associated to $T'$ as in Definition~\ref{defn:parascopic-group}.
Let $(\phi', {q'}^*)$
be a parascopic datum for $(\tilde G,\Gamma,G)$
relative to $\tilde T'$ and $T'$
that is equivalent to $(\phi,q^*)$ via elements $g\in G(k\sep)$ and $\tilde g\in \tilde G(k\sep)$.
Since $\lsub\phi\tilde\Psi = \lsub{\phi^\sharp}\tilde\Psi$,
the orbits of $\Gamma$ and $\Gamma^\sharp$ are the same,
and hence so are the embeddings $\map{i_{\phi,q^*}, i_{\phi^\sharp,q^*}}{W(G,T)}{W(\tilde G,\tilde T)}$
from \S\ref{para:weyl}. By Remark~\ref{rem:equivalence}, it follows that the
conjugate $({\phi'}^\sharp, {q'}^*)$ of $(\phi', {q'}^*)$ via $g$ and $\tilde g$ is
a parascopic datum for $(\tilde G,\Gamma,G)$
relative to $\tilde T'$ and $T'$ that is equivalent to $(\phi^\sharp, q^*)$.
From \eqref{eq:norm-char} and
Theorem~\ref{thm:lift-geometric-general}, it will be enough to
show that 
$\normchar[\phi',{q'}^*] = \normchar[{\phi'}^\sharp,{q'}^*]$, which will follow from
\eqref{eqn:norm} if we can show that $\normchar[\phi'] = \normchar[{\phi'}^\sharp] $.

Since $\Int(\tilde g)^*$ is a $\Gamma$-equivariant isomorphism from
$\Psi(\tilde G,\tilde T')$ to $\Psi(\tilde G,\tilde T)$,
it induces an isomorphism $\lsub{\Gamma}\Int(\tilde g)^*$ from
$\lsub{\phi'}\Psi(\tilde G,\tilde T')$
to $\lsub{\phi}\Psi(\tilde G,\tilde T)$.
Since $\phi'(\gamma) = {\Int(\tilde g)^*}\inv \circ \phi(\gamma) \circ \Int(\tilde{g})^*$ for all $\gamma\in\Gamma$,
it follows from the definition of the norm (\S\ref{para:conorm-defn}) that 
$\normchar[\phi'] = {\Int(\tilde g)^*}\inv \circ \normchar[\phi] \circ \lsub{\Gamma}\Int(\tilde{g})^*$.
Similarly,
$ \normchar[{\phi'}^\sharp]  = {\Int(\tilde g)^*}\inv \circ \normchar[\phi^\sharp] \circ \lsub{\Gamma^\sharp}\Int(\tilde{g})^*$.
Since $\normchar[\phi] = \normchar[\phi^\sharp]$,
it follows that $ \lsub{\Gamma}\Int(\tilde{g})^* = \lsub{\Gamma^\sharp}\Int(\tilde{g})^*$
and hence that $\normchar[\phi'] = \normchar[{\phi'}^\sharp] $.
\end{proof}

\section{Closed and dual-closed subdata}

If $\tilde\Psi = (X^*, \tilde\Phi,X_*,\tilde\Phi^\vee)$
is a root datum,
and 
$\Psi = (X^*, \Phi,X_*,\Phi^\vee)$ is a subdatum of $\Psi$,
then we will say that $\Psi$ is \emph{closed} in $\tilde\Psi$
if $\Phi$ is closed in $\tilde\Phi$,
and \emph{dual-closed} 
if $\Phi^\vee$ is closed in $\tilde\Phi^\vee$.

\begin{para}
\label{para:conormfunc-subsystem}
Suppose that $\Psi$ and $\tilde\Psi$ possess $\Gal(k)$-actions.
Let $q^*$ be a $\Gal(k)$-equivariant embedding from $\Psi$ to $\tilde\Psi$.
Consider the parascopic datum $(1,q^*)$
for $(\tilde\Psi,1,\Psi)$
as in Example~\ref{ex:parascopy-gamma-trivial}.
From~\cite{adler-lansky:data-actions}*{Remark 25}
there exists a connected reductive $k$-group $G$
and a maximal $k$-torus $T\subseteq G$ such that
$\Psi$ is $\Gal(k)$-isomorphic to $\Psi(G,T)$.
Similarly, we obtain $\tilde T\subseteq\tilde G$
such that $\tilde\Psi$ is $\Gal(k)$-isomorphic to $\Psi(\tilde G,\tilde T)$.
Identifying $\tilde\Psi$
with $\Psi(\tilde G,\tilde T)$
and $\Psi$ with $\Psi(G,T)$, we have that
$(1, q^*)$ is a parascopic datum
for $(\tilde G,1,G)$.
If the image of $\Psi$ under $q^*$ is closed in $\tilde\Psi$,
then it is easy to manufacture the group $G$ in such a way
that there is an $k$-embedding $\abmap{G}{\tilde G}$
that restricts to give an isomorphism of $T$ onto $\tilde T$
which induces the map $\map{{q^*}\inv}{X^*(\tilde T)}{X^*(T)}$.

From~\cite{adler-lansky:data-actions}*{Remark 25},
we can construct $\tilde{G}$ or $G$ to be quasi-split.
However, even if $\tilde{G}$ is quasi-split,
we cannot simultaneously ensure that $G$ is quasi-split
and that the embedding from $G$ to $\tilde{G}$
takes $T$ to $\tilde{T}$.
We might have to replace $\tilde{T}$ by another maximal torus
in $\tilde{G}$ that is stably conjugate to it:
\end{para}

\begin{prop}
\label{prop:embedding}
(cf.~\cite{kottwitz:rational-conj}*{Lemma 3.3})
Let $\tilde G$ be a connected reductive quasisplit $k$-group and
$\tilde T\subseteq \tilde G$ a maximal $k$-torus.
Let $\Psi$
be a closed, $\Gal(k)$-invariant subdatum of $\Psi(\tilde G,\tilde T)$.
Then there exists a connected reductive quasisplit $k$-group $G$,
a maximal $k$-torus $T$ of $G$, and a $\Gal(k)$-embedding
$q^*$ from $\Psi(G,T)$ to $\Psi(\tilde G,\tilde T)$
such that ${q^*} (\Psi(G,T)) = \Psi$.
Moreover,
there exists a $k$-embedding $\map{\nu}{G}{\tilde G}$,
and an element $\tilde x\in \tilde G(k\sep)$ such that
$\lsup{\tilde x} (\nu(T)) = \tilde T$,
and the map $\abmap{X^*(\tilde T)}{X^*(T)}$ induced
by $\Int(\tilde x)\circ\nu$ coincides with ${q^*}\inv$.
In particular, $\nu(T)$ is stably conjugate to $\tilde T$.
\end{prop}
 
\begin{proof}
By~\cite{adler-lansky:data-actions}*{Remark 25}, there exists a
connected reductive $k$-group $G$, a maximal $k$-torus $T$ of $G$,
and a $\Gal(k)$-isomorphism from $\Psi$ to $\Psi(G,T)$.
Thus there is a $\Gal(k)$-embedding $q^*$ from $\Psi(G,T)$
to $\Psi(\tilde G,\tilde T)$ such that $q^* (\Psi(G,T)) = \Psi$.

Let $T_0$ be a maximal $k$-torus contained in a Borel $k$-subgroup of $G$.
Choose $g\in G(k\sep)$ such that $T_0 = (\Int g)(T)$.
Then the map $z$ taking $\sigma\in\Gal(k)$ to
the image of ${g}\inv\sigma(g)$ in $W(G,T)$ is a cocycle in $Z^1(k,W(G,T))$,
and the stable conjugacy class of
$T_0$ in $G$ corresponds to the class
of $z$ in $H^1(k,W(G,T))$.
The $\Gal(k)$-equivariant embedding
$\map{i_{q^*}}{W(G,T)}{W(\tilde G,\tilde T)}$ determines a map
$\abmap{Z^1(k,W(G,T))}{Z^1(k,W(\tilde G,\tilde T))}$.
Let $\tilde z$ be the image of $z$ under this map.
By~\cite{raghunathan:tori}*{Theorem 1.1}, there exists
$\tilde g\in \tilde G(k\sep)$ such that
$\tilde z$ coincides with the map taking $\sigma\in\Gal(k)$ to
the image of $\tilde g\inv \sigma(\tilde g)$ in $W(\tilde G,\tilde T)$.
Then $\tilde T_0 := (\Int \tilde g)(\tilde T)$ is a maximal $k$-torus of $\tilde G$
in the stable conjugacy class of tori corresponding to the class of $\tilde z$ in $H^1(k,W(\tilde G,\tilde T))$.

The subdatum $\Psi$ can be transported to a closed subdatum
$\Psi_0$ of $\Psi(\tilde G,\tilde T_0)$ via $\Int (\tilde g\inv)^*$.
Moreover, $\Psi_0$ is
stabilized by $\Gal(k)$ since $\tilde z$ takes values in the image of
$W(G,T)$ in $W(\tilde G,\tilde T)$, and elements of this image
stabilize $\Psi$.

Thus there is a reductive
$k$-subgroup $\tilde H$ of $\tilde G$ that contains $\tilde T_0$ with root datum
$\Psi_0$.
By construction, the composition of the chain of isomorphisms
\begin{equation}
\label{eq:data-maps}
\Psi(G,T_0)
\xrightarrow{\Int(g)^*} \Psi(G,T)
\xrightarrow{\phantom{xx}q^*\phantom{xx}} \Psi
\xrightarrow{\Int(\tilde g\inv)^*} \Psi (\tilde H,\tilde T_0)
\end{equation}
is readily seen to be $\Gal(k)$-equivariant.
Thus, since $\Gal(k)$ stabilizes a base in $\Phi(G,T_0)$
it must stabilize one in $\Phi(\tilde H,\tilde T_0)$.
It follows (see~\cite{springer:lag-article}*{Theorems 4.4.1(ii) and 6.1.2}) that $\tilde H$ is $k$-quasisplit and hence
that there exists a $k$-isomorphism $\map{\nu}{G}{\tilde H}$ taking $T_0$ onto $\tilde T_0$
and such that the induced isomorphism $\nu^*_{T_0}$
of $\Psi(\tilde H,\tilde T_0)$ with $\Psi(G,T_0)$
is the inverse of that given by \eqref{eq:data-maps}.

The map $\nu$ also induces a $\Gal(k)$-isomorphism
$\nu^*_{T}$ of $\Psi(\tilde H,\nu(T))$ with $\Psi(G,T)$, where
$$
\nu^*_{T} = \Int(g)^*\circ \nu^*_{T_0}\circ \Int(\nu(g)\inv)^*.
$$
Thus, by the definition of $\nu_{T_0}$,
\begin{equation*}
\begin{aligned}
\nu^*_{T}
&=  \Int(g)^*\circ
\bigl(\Int(g\inv)^*\circ{q^*}\inv\circ\Int(\tilde g)^*\bigr)
\circ \Int(\nu(g)\inv)^*\\
&=  {q^*}\inv\circ\Int(\nu(g)\inv \tilde g)^* .
\end{aligned}
\end{equation*}
Letting $\tilde x = \tilde g\inv\nu(g)$, we therefore have that
$$
{q^*}\inv = \nu^*_{T}\circ\Int(\tilde x)^* = (\Int(\tilde x)\circ\nu)^*,
$$
as desired.

Note that $\Int (\tilde x)(\nu(T)) = \tilde T$. Since $\nu^*$ and $q^*$
are both $\Gal(k)$-equivariant, it follows that
the isomorphism $\Int(\tilde x)^*$ of $X^*(\tilde T)$ with $X^*(\nu(T))$ is as well.
Thus $\nu(T)$ is stably conjugate to $\tilde T$ via $\tilde x$.
\end{proof}

\begin{prop}
\label{prop:max-subdatum}
Let $\tilde \Psi$ be an irreducible, reduced root datum carrying an action of $\Gal(k)$.
Let $\Psi$ be a maximal $\Gal(k)$-invariant subdatum of $\tilde\Psi$.
Then $\Psi$ is either closed or dual-closed in $\tilde\Psi$.
\end{prop}

\begin{proof}
Let $\overline\Psi$ be a maximal subdatum of $\tilde\Psi$
containing the maximal $\Gal(k)$-invariant subdatum $\Psi$ of $\tilde\Psi$.
Then $\Psi$ is equal to the intersection of all
$\Gal(k)$-conjugates of $\overline\Psi$.
If $\overline\Psi$ is closed (resp.\ dual-closed), then 
so is $\Psi$.
Therefore, replacing $\Psi$ by $\overline\Psi$, it is enough to show
that every maximal subdatum of $\tilde\Psi$ is either
closed or dual-closed.

Let $\Phi$ and $\tilde\Phi$ denote the root systems of $\Psi$ and
$\tilde\Psi$.
If $\tilde\Phi$ is simply laced, then every subsystem is both
closed and dual-closed, so there is nothing to prove.
If $\Phi$ contains only long roots, then it is closed in $\tilde\Phi$;
if it contains only short roots, then its dual $\Phi^\wedge$
contains only long roots, making it closed in $\tilde\Phi^\wedge$,
thus making $\Phi$ dual-closed in $\tilde\Phi$.
Therefore, assume from now on that $\tilde\Phi$ and $\Phi$
each contain roots of two different lengths.

If $\tilde\Phi$ has type $G_2$, then it is easy to check
our result directly. Also, our result is true for $C_n$ if and only if it is true for $B_n$.
Therefore, we only need to check the cases of $C_n$ and $F_4$.

If $\Phi$ has smaller rank than $\tilde\Phi$, then so do its
closure and dual-closure.  By maximality, $\Phi$ must be
both closed and dual-closed.  Therefore, we may assume that
$\Phi$ has the same rank as $\tilde\Phi$.

Suppose $\Phi$ is a product of two nonempty subsystems,
not assumed irreducible:
$\Phi = \Weil_1 \times \Weil_2$.
Let $\overline \Weil_i$ denote the smallest subsystem of $\tilde\Phi$ that
contains $\Weil_i$ and is both closed and dual-closed in $\tilde\Phi$.
Then
$\Phi\subseteq \overline \Weil_1 \times \overline \Weil_2$.
We have that $\overline \Weil_1 \times \overline \Weil_2 \neq \tilde\Phi$
by the irreducibility of $\tilde\Phi$,
so the maximality of $\Phi$
implies that each factor $\Weil_i$ must be both closed and dual-closed
in $\tilde\Phi$.

Now suppose that
$\Phi$ is a product of more than two irreducible factors:
$\Phi = \prod_{i=1}^n \Phi_i$ ($n\geq 3$).
Suppose $\alpha, \beta\in \Phi$, and $\alpha+\beta\in \tilde\Phi$.
Grouping together the one or two factors $\Phi_i$ that contain
$\alpha$ and $\beta$, we can write
$\Phi = \Phi_0 \times \Phi'$, where each factor is nonempty,
and $\alpha,\beta \in \Phi_0$.
From the previous paragraph, $\Phi_0$ is both closed and dual-closed
in $\tilde\Phi$.
Therefore,
$\alpha+\beta\in\Phi_0\subset\Phi$, so $\Phi$ is closed in $\tilde\Phi$.
Similarly, $\Phi$ is also dual-closed in $\tilde\Phi$.

Therefore, we only need to consider the case where
$\Phi$ is either irreducible or is a product of two
irreducible factors.

Borel and de Siebenthal~\cite{borel-desiebenthal}*{\S7}
provide an algorithm to list all
maximal, full-rank, closed subsystems of $\tilde\Phi$ up to isomorphism:
Remove one vertex from
the extended
Dynkin diagram of $\tilde\Phi$.
Applying the same algorithm to $\tilde\Phi^\vee$ and taking
the duals of the results, we obtain all maximal, full-rank,
dual-closed subsystems of $\tilde\Phi$.
Applying both of the above to $F_4$, we obtain
the closed subsystems
$B_4$,
$A_1\Short \times A_3\Long$,
$A_2\Short \times A_2\Long$,
and
$C_3\times A_1\Long$;
and the dual-closed subsystems
$C_4$,
$A_1\Long \times A_3\Short$,
$A_2\Long \times A_2\Short$,
and
$B_3\times A_1\Short$.
In $C_n$,
we obtain the closed subsystems
$C_i \times C_j$ ($i+j=n$);
and the dual-closed subsystems
$C_i \times D_j$ ($i+j=n$).
We need to show that $F_4$ and $C_n$ have no other
maximal, full-rank, proper subsystems.

Suppose that $\Phi$ is irreducible,
and recall that we may assume that $\Phi$ is not simply laced.
If $\tilde\Phi$ has type $F_4$, then observe that
$B_4$ and $C_4$ are closed and dual-closed (respectively) in $\tilde\Phi$.
If $\tilde\Phi$ has type $C_n$, then it is straightforward
to check that
$\tilde\Phi$ contains no proper irreducible, non-simply-laced subsystem
of rank $n$.

It remains only to consider the case when $\Phi$ is a product
of two irreducible factors: $\Phi = \Weil_1 \times \Weil_2$.
Suppose for a contradiction
that $\Phi$ is neither closed nor dual-closed in $\tilde\Phi$.
Then since we have already seen that each $\Weil_i$ must be closed
and dual-closed,
we must have short roots
$\alpha_i \in \Weil_i$ for $i=1,2$
such that $\alpha_1+\alpha_2\in \tilde\Phi$,
and
we must have long roots
$\beta_i\in \Weil_i$ 
such that $(\beta_1+\beta_2)/2\in \tilde\Phi$.
In particular, each $\Weil_i$ must contain both short and long roots,
and so cannot be simply laced.

Suppose that $\tilde\Phi$ has type $F_4$.
It is straightforward to check that $G_2 \times B_2$ is not
contained in $\tilde\Phi$, and $B_2 \times B_2$ is contained
(as a dual-closed subsystem) in $B_4$, and therefore is not
maximal in $\tilde\Phi$.

Suppose $\tilde\Phi$ has type $C_n$.
Note that an irreducible, non-simply-laced root system of rank
$i$ has at least $i$ long roots, with equality if and only if
our root system is $C_i$.
Therefore, the only product of two such root systems
that can be a full-rank subsystem of $C_n$ is $C_i\times C_j$ for $i+j=n$,
and this is a closed subsystem.
\end{proof}

\section{Compatibility of conorms with isogenies}
\label{sec:isogenies}
Suppose $(\phi,q^*)$ is a strong parascopic datum for $(\tilde{G},\Gamma,G)$.
Then one has associated conorm functions
$\dnorm[{[\phi,q^*]}]$ and $\dnormst[{[\phi,q^*]}]$
from the sets of geometric and stable conjugacy classes in $G^\wedge$ to those in 
$\tilde{G}^\wedge$.
Suppose we have a separable surjective $k$-homomorphism $\tilde f$ from $\tilde{G}$ to another $k$-group
$\tilde{G}'$ with central kernel, and that $\tilde f$ has good invariance properties with
respect to $\Gamma$.
In this section, we show the existence of a group $G'$,
a separable surjective $k$-homomorphism $\map{f}{G}{G'}$ with central kernel,
and a parascopic datum $(\phi',q^{*\prime})$ for $(\tilde{G}',\Gamma,G')$
such that $\dnorm[{[\phi,q^*]}]$ (resp.\ $\dnormst[{[\phi,q^*]}]$)
can be described in terms of
$\dnorm[{[\phi',q^{*\prime}]}]$ (resp.\ $\dnormst[{[\phi',q^{*\prime}]}]$),
together with the maps $f$ and $\tilde f$.
(In fact, we show that such a result holds in a somewhat more general situation described below.)
In particular,
for the purpose of understanding conorm maps, this will allow one
to pass to the situation in which $\tilde{G}$ is adjoint,
or, equivalently, $\tilde{G}^\wedge$ is simply connected.

In this section, let
$\tilde\Psi=(\tilde X^*,\tilde\Phi, \tilde X_*, \tilde\Phi^\vee)$
and $\Psi=(X^*,\Phi, X_*, \Phi^\vee)$
denote root data with actions of $\Gal(k)$; and $\Gamma$ a finite group.
Let $(\phi,q^*)$ denote a parascopic datum for $(\tilde\Psi,\Gamma,\Psi)$.
Let $\tilde\Psi'=(\tilde X'^*,\tilde\Phi',\tilde X_*',\tilde\Phi'^\vee)$
denote another root datum with $\Gal(k)$-action,
and $\tilde f^*$ a $\Gal(k)$-equivariant morphism
from $\tilde\Psi'$ to $\tilde\Psi$.
We will assume that there exists a $\Gal(k)$-equivariant action $\phi'$ of $\Gamma$ on $\tilde\Psi'$
such that $\tilde f^*\circ \phi'(\gamma) = \phi(\gamma)\circ \tilde f^*$
for all $\gamma\in \Gamma$. We observe that a unique such $\phi'$ must exist
when $\tilde f^*$ is injective and $\im \tilde f^* \subseteq \tilde X^*$
is invariant under the action of $\phi(\Gamma)$.

Since $\tilde f^*$ is $\Gamma$-equivariant, it induces a morphism
$\lsub{\Gamma}\tilde f^*$ from $\lsub{\phi'}\tilde\Psi'$ to $\lsub{\phi}\tilde\Psi$.
Taking the pullback of $q^*$ and $\lsub{\Gamma}\tilde f^*$
(see Remark~\ref{rem:morphism-pullback}), we obtain
a root datum $\Psi' = ({X'}^*,\Phi', X'_*, {\Phi'}^\vee)$,
and $\Gal(k)$-equivariant morphisms ${q'}^*$ from $\Psi'$ to $\lsub{\phi'}\tilde \Psi'$
and $f^*$ from $\Psi'$ to $\Psi$ satisfying $q^*\circ f^* = \lsub{\Gamma}\tilde f^*\circ {q'}^*$.

\begin{prop}
\label{prop:associated-morphism}
The pair $(\phi',{q'}^*)$ is a parascopic datum for $(\tilde \Psi', \Gamma, \Psi')$.
This datum is torus-inclusive if $(\phi,q^*)$ is,
and is root-inclusive if and only if $(\phi,q^*)$ is.
\end{prop}
We will call $(\phi', q^{\prime *})$
the \emph{pullback of $(\phi, q^*)$ along $\tilde f^*$}.

\begin{proof}
To see that $(\phi', {q'}^*)$ is a parascopic datum for $(\tilde \Psi', \Gamma, \Psi')$,
it suffices to observe that if $\tilde\Phi^+$ is a
$\phi(\Gamma)$-invariant
set of positive roots in $\tilde\Phi$, then
$\phi'(\Gamma)$ preserves the set $\tilde f^{*\, -1} (\tilde\Phi^+)$ of positive roots in $\tilde\Phi'$.

To verify the second statement,
note that if $q^*$ is an isomorphism of lattices, then so is $q^{\prime *}$,
and that $q^*(\Phi) \subseteq \lsub{\phi}\tilde\Phi\red$ if and only if
$q^{\prime*}(\Phi') \subseteq \lsub{\phi'}\tilde\Phi'\red$.
\end{proof}

Suppose we have connected reductive $k$-groups $\tilde{G}$, $\tilde{G}'$,
and $G$;
maximal
$k$-tori $\tilde T \subseteq \tilde{G}$ and $T\subseteq G$;
and a  parascopic datum $(\phi,q^*)$ for $(\tilde{G},\Gamma,G)$
relative to the tori $\tilde{T}$ and $T$.
Suppose $\map{\tilde{f}}{\tilde{G}}{\tilde{G}'}$ is a surjective $k$-homomorphism
with central kernel invariant under the action of $\Gamma$ on $\tilde T$ induced by $\phi$.

Let $\tilde T' = \tilde{f}(\tilde{T})\subseteq \tilde{G}'$.
Let
$\tilde\Psi = \Psi(\tilde G,\tilde T)$,
$\Psi = \Psi(G,T)$,
and
$\tilde\Psi' = \Psi(\tilde G',\tilde T')$.
Then $\tilde f^*$ is an injective morphism from $\tilde\Psi'$
to $\tilde\Psi$ with $\phi(\Gamma)$-invariant image.
The pullback $(\phi',{q'}^*)$ of $(\phi,q^*)$ along $\tilde f^*$
as in Proposition~\ref{prop:associated-morphism} is a
parascopic datum for $(\tilde\Psi', \Gamma,\Psi')$,
where the root datum $\Psi'$, as defined above, comes
equipped with a morphism $f^*$ from $\Psi'$ to $\Psi$.

It is a straightforward consequence of~\cite{springer:corvallis}*{Lemma 2.8}
that there exists a $k$-group $G'$
with maximal $k$-torus $T'$
such that $\Psi'$ is $k$-isomorphic to $\Psi(G',T')$,
and a surjective $k$-homomorphism
$\map{f}{G}{G'}$ with central kernel that (identifying $\Psi'$ and $\Psi(G',T')$) induces the morphism $f^*$.
Then $(\phi',q^{\prime *})$ is a parascopic datum 
for $(\tilde G', \Gamma, G')$
relative to the tori $\tilde T'$ and $T'$.

\begin{prop}
\label{prop:isogeny-pushout}
With the preceding notation, the parascopic datum $(\phi,q^*)$ is strong for $(\tilde G,\Gamma, G)$
if and only if its pullback
$(\phi',q^{\prime *})$ is strong for $(\tilde G',\Gamma, G')$.
\end{prop}

\begin{proof}
This follows from the fact that
$\tilde f$ induces a bijection $\abmap{\TTst(\tilde G, k)}{\TTst(\tilde G', k)}$
and
$f$ induces a bijection $\abmap{\TTst(G, k)}{\TTst(G', k)}$.
\end{proof}

\begin{prop}
\label{prop:conorm-isogeny}
In the situation of Proposition~\ref{prop:isogeny-pushout},
suppose that
the parascopic datum $(\phi,q^*)$ is strong
and that the connected reductive $k$-groups
$\tilde{G}^\wedge$, $G^\wedge$, $\tilde G'^\wedge$, and $G'^\wedge$ are in
$k$-duality respectively with $\tilde{G}$, $G$, $\tilde G'$, and $G'$.
Let $\map{\tilde f\dual}{\tilde G'^\wedge}{\tilde G^\wedge}$
and
$\map{f\dual}{G'^\wedge}{G^\wedge}$ be
$k$-isotypies dual (see \S\ref{para:duality})
respectively to $\tilde f$ and $f$.
Let $\tilde f\dual$ and $f\dual$ also denote the corresponding
maps on the sets of geometric and stable conjugacy classes.
Then
\begin{align*}
\dnormst[{[\phi,q^*]}]
&= \tilde f\dual\circ \dnormst[{[\phi',q^{\prime *}]}] \circ {f\dual}\inv,
\quad\text{and}\\
\dnorm[{[\phi,q^*]}]
&= \tilde f\dual\circ \dnorm[{[\phi',q^{\prime *}]}] \circ {f\dual}\inv.
\end{align*}
In particular, the right-hand sides of the equations above are well defined.
\end{prop}

\begin{proof}
We only prove the statement about the conorm maps
on stable
conjugacy classes, as the statement
for geometric conjugacy classes follows from considering stable classes over $\overline k$.

For a semisimple element $s \in G^\wedge(k)$,
choose a maximal $k$-torus $T^\wedge\subseteq G^\wedge$ containing $s$.
As in \S\ref{para:square-of-tori},
we obtain
maximal $k$-tori
$T\subseteq G$,
$\tilde{T}\subseteq \tilde{G}$, and
$\tilde{T}^\wedge \subseteq \tilde{G}^\wedge$,
all determined up to stable conjugacy.
These tori in turn determine maximal $k$-tori $T'\subseteq G'$, ${T'}^\wedge\subseteq {G'}^\wedge$,
$\tilde T'\subseteq\tilde G'$, and $\tilde T^{\prime\wedge}\subseteq \tilde G^{\prime\wedge}$.
Choose duality maps $\map{\delta}{X^*(T)}{X_*(T^\wedge)}$ and
$\map{\tilde\delta}{X^*(\tilde T)}{X_*(\tilde T^\wedge)}$. Then as in
\S\ref{para:duality}, these maps determine duality maps
$\map{\delta'}{X^*(T')}{X_*(T'^\wedge)}$ and
$\map{\tilde\delta'}{X^*(\tilde T')}{X_*(\tilde T'^\wedge)}$.
Corresponding to these tori and duality maps, we obtain
conorm maps $\map{\dnorm[T^\wedge]}{T^\wedge}{\tilde{T}^\wedge}$
and $\map{\dnorm[{T'}^\wedge]}{T'^\wedge}{\tilde{T'}^\wedge}$.
From 
\cite{adler-lansky:lifting}*{Theorem \xref{FGL-thm:main}
	and Remark \xref{FGL-rem:main}},
$\dnorm[T^\wedge](s) \in \dnormst[{[\phi,q^*]}]([s]\st)$.
Therefore, it will be enough to prove that
\begin{equation}
\label{eq:isogeny-conorm}
\dnorm[T^\wedge](s) = (\tilde f\dual\circ\dnorm[T'^\wedge] \circ {f\dual}\inv )(s),
\end{equation}
and that the expression on the right is well defined.

We have norm maps $\map{\normchar[T]}{X^*(T)}{X^*(\tilde{T})}$ and $\map{\normchar[T']}{X^*(T')}{X^*(\tilde{T}')}$
given by \eqref{eqn:norm} that yield the diagram:
$$
\begin{xy}
\xymatrix{ 
& \bX^*(\tilde{T}')
	\ar[rr]^{\tilde\delta'}
	\ar@{<-}'[d][dd]_(.3){\normchar[T']}
& & \bX_*(\tilde T^{\prime\wedge})
	\ar@{<-}[dd]^{\dnormcochar{T'^\wedge}}
\\ 
\bX^*(\tilde T)
	\ar@{<-}[ur]^{\tilde f^*}
	\ar[rr]^(.65){\tilde\delta}
	\ar@{<-}[dd]_{\normchar[T]}
& & \bX_*(\tilde{T}^\wedge)
	\ar@{<-}[ur]^{\tilde f^\wedge_*}
	\ar@{<-}[dd]^(.35){\dnormcochar{T^\wedge}}
\\ 
& \bX^*(T')
	\ar'[r]_(.6){\delta'}[rr]
& & \bX_*(T^{\prime\wedge})
\\ 
\bX^*(T)
	\ar[rr]_{\delta}
	\ar@{<-}[ur]^{f^*} 
& & \bX_*(T^\wedge)
	\ar@{<-}[ur]^{f^\wedge_*} 
} 
\end{xy}
$$
The top and bottom faces commute by the definitions of the duality maps
$\delta'$ and $\tilde\delta'$.
The front and back faces commute by the definitions of the conorms
$\dnormcochar{T^\wedge}$ and $\dnormcochar{T'^\wedge}$.
The left face commutes by the definitions of $\normchar[T]$, $\normchar[T']$,
$\phi'$ and ${q'}^*$.
Since all of the duality maps are isomorphisms, it follows that the right face must also commute.
Thus $\dnormcochar{T^\wedge}\circ {f_*\dual} = \tilde f_*\dual\circ \dnormcochar{T'^\wedge}$
so $\dnorm[T^\wedge]\circ {f\dual} = \tilde f\dual\circ \dnorm[T'^\wedge]$.
To prove \eqref{eq:isogeny-conorm}, it remains to show that 
$\dnorm[T'^\wedge] (\ker f \dual) \subseteq \ker \tilde f\dual$, or equivalently that
$\dnormcochar{T'^\wedge} (\ker f_* \dual) \subseteq \ker \tilde f_*\dual$.
But this follows from the commutativity of the right face of the cube.
\end{proof}

\section{Composition of parascopic data and conorms: Definition and basic results}
\label{sec:decomposition-defn}

\begin{defn}
\label{defn:composition}
Suppose that
\begin{itemize}
\item
for $j\in \{0,1,2\}$,
$\Psi_j = (X^*_j, \Phi_j, {X_j}_*, \Phi^\vee_j)$,
is a root datum equipped with an action of $\Gal(k)$;
\item
$\Gamma$, $\Gamma_0$, and $\Gamma_1$ are finite groups;
\item
for $j=0$ or $1$,
$(\phi_j, q_j^*)$
is a parascopic datum for
$(\Psi_j, \Gamma_j, \Psi_{j+1})$;
and
\item
$(\phi, q^*)$
is a parascopic datum for
$(\Psi_0, \Gamma, \Psi_2)$.
\end{itemize}
We say that
\emph{\IsCompOf{(\phi,q^*)}{(\phi_0, q_0^*)}{(\phi_1, q_1^*)}}
if the following are satisfied.
\begin{itemize}
\item
$\Gamma$ is an extension of $\Gamma_1$ by $\Gamma_0$.
\item
$\phi_0$ is the restriction of $\phi$ to $\Gamma_0$.
\item
The morphism $q_0^*$ from $\Psi_1$ to $\lsub{\phi_0}\Psi_0$
is $\Gamma_1$-equivariant with respect to the action $\phi_1$ on $\Psi_1$ and
the action $\bar\phi$
of $\Gamma_1 = \Gamma / \Gamma_0$
on $\lsub{\phi_0}\Psi_0$ arising from the action $\phi$ of $\Gamma$ on $\Psi_0$.
Hence it induces a morphism $\llsub{\Gamma_1}q_0^*$ from $\lsub{\phi_1}\Psi_1$ to
$\lsub{\bar\phi}(\lsub{\phi_0}\Psi_0) = \lsub{\phi}\Psi_0$.
\item
As morphisms from $\Psi_2$ to $\lsub{\phi}\Psi_0$,
we have $q^* = \llsub{\Gamma_1}q_0^* \circ q_1^*$.
\end{itemize}
\end{defn}

\begin{rem}
In general, a composition, like a group extension,
is not determined by its factors alone.
However, the composition is indeed uniquely determined if either $\Gamma_0$ or $\Gamma_1$ is trivial.
\end{rem}

\begin{example}
Suppose $\Psi_0$ is a root datum with $\Gal(k)$-action,
$\Psi_1$ is a $\Gal(k)$-invariant subdatum of $\Psi_0$,
and $\Psi_2$ is a $\Gal(k)$-invariant subdatum of $\Psi_1$.
Then the parascopic datum for $(\Psi_0,1,\Psi_2)$
induced by the inclusion morphism from $\Psi_2$ to $\Psi_0$
is the composition of the analogous parascopic data for
$(\Psi_0,1,\Psi_1)$ and $(\Psi_1,1,\Psi_2)$.
\end{example}

\begin{rem}
\label{rem:decomp-canonical}
Suppose $\tilde\Psi$ and $\Psi = (X^*,\Phi,X_*,\Phi^\vee)$ are root data, each with an action of $\Gal(k)$,
$\Gamma$ is a finite group,
and $(\phi,q^*)$ is a parascopic datum for $(\tilde\Psi,\Gamma,\Psi)$.
Then $(\phi,q^*)$ is a composition of the following three parascopic data:
\begin{itemize}
\item
Letting $\id$ denote the morphism from $\lsub{\phi}\tilde\Psi$ to itself,
we obtain a torus-inclusive parascopic datum $(\phi,\id)$
for $(\tilde\Psi, \Gamma,\lsub{\phi}\tilde\Psi)$
(see Example \ref{ex:parascopy-group-action}).
\item Letting $1$ denote the trivial action of the trivial group on $\lsub\phi\tilde\Psi$,
and letting $\id$ denote the identity map on the character lattice in $\lsub{\phi}\tilde\Psi$,
we obtain the torus-inclusive parascopic datum $(1,\id)$
for $(\lsub\phi\tilde\Psi, 1, q^*(\Psi))$
(see Example~\ref{ex:parascopy-gamma-trivial}).
\item
The map $q^*$ restricts to give a bijection of $\Phi$ onto $q^*(\Phi)$,
so the pair $(1,q^*)$ is a root-inclusive parascopic datum for $(q^*(\Psi), 1,\Psi)$
of the kind considered in Example~\ref{ex:parascopy-gamma-trivial} for
$q^*$ as in Remark~\ref{rem:morphism-example}(\ref{item:morphism-example-isogeny-derived})
on isotypies.
\end{itemize}
Note that if $(\phi,q^*)$ is torus-inclusive,
then $q^*$ actually gives an isomorphism from $\Psi$ to $q^*(\Psi)$,
so the third factor above is essentially unnecessary.
Also, $(\phi,q^*)$ is root-inclusive if and only if we can replace
$\lsub{\phi}\tilde\Phi$ by $\lsub{\phi}\tilde\Phi\red$ in the first two factors above.
\end{rem}

\begin{lem}
\label{lem:weyl-embedding-factorization}
Retain the notation of Definition \ref{defn:composition}.
Let
$\map{i_{\phi,q^*}}{W(\Psi_2)}{W(\Psi_0)}$, 
$\map{i_{\phi_0,q_0^*}}{W(\Psi_1)}{W(\Psi_0)}$, and
$\map{i_{\phi_1,q_1^*}}{W(\Psi_2)}{W(\Psi_1)}$
be the Weyl group embeddings described in~\S\ref{para:weyl}.
Then $i_{\phi,q^*}= i_{\phi_0,q_0^*} \circ i_{\phi_1,q_1^*}$.
\end{lem}

\begin{proof}
As above, let $\bar\phi$ denote the action of $\Gamma_1 = \Gamma / \Gamma_0$
on $\lsub{\phi_0}\Psi_0$ arising from the action $\phi$ of $\Gamma$ on $\Psi_0$.
As in \S\ref{para:weyl}, we have embeddings
\begin{align*}
i_{\phi_0}\colon \abmap{W(\lsub{\phi_0}\Psi_0)}{W(\Psi_0)}, & \quad
i_{\bar\phi}\colon \abmap{W(\lsub{\phi}\Psi_0) = W(\lsub{\bar\phi}(\lsub{\phi_0}\Psi_0))}{W(\lsub{\phi_0}\Psi_0)},\\ 
\map{i_{\phi_1}}{W(\lsub{\phi_1}\Psi_1)}{W(\Psi_1)}, & \quad \text{and}\quad  i_\phi\colon  \abmap{W(\lsub{\phi}\Psi_0)}{W(\Psi_0)}.
\end{align*}
Also, from Remark \ref{rem:morphism-properties},
the morphisms $q^*$, $q_0^*$, $q_1^*$, and $\lsub{\Gamma_1}q_0^*$ induce embeddings
\begin{align*}
\map{i_{q_0*}}{W(\Psi_1)}{W(\lsub{\phi_0}\Psi_0)}, & \quad
\map{i_{\lsub{\Gamma_1}q_0^*}}{W(\lsub{\phi_1}\Psi_1)}{W(\lsub{\bar\phi}(\lsub{\phi_0}\Psi_0)) = W(\lsub{\phi}\Psi_0)},\\
\map{i_{q_1^*}}{W(\Psi_2)}{W(\lsub{\phi_1}\Psi_1)},\ & \text{and}\quad \map{i_{q^*}}{W(\Psi_2)}{W(\lsub{\phi}\Psi_0)}.
\end{align*}
These maps fit into the diagram:
\begin{equation}
\label{eq:weyl}
\begin{xy}
\xymatrix{
& &
W(\Psi_0)\\
&
W(\lsub{\phi}\Psi_0) \ar[r]^{i_{\bar\phi}} \ar[ur]^{i_{\phi}} &
W(\lsub{\phi_0}\negthinspace\Psi_0) \ar[u]^{i_{\phi_0}}\\
W(\Psi_2) \ar[r]^{i_{q_1^*}} \ar[ur]^{i_{q^*}} \ar@/^2pc/[uurr]^{i_{\phi,q^*}} 
	\ar@/_2pc/[rr]_{i_{\phi_1,q_1^*}} &
W(\lsub{\phi_1}\Psi_1) \ar[r]^{i_{\phi_1}} \ar[u]_{i_{\llsub{\Gamma_1}q_0^*}} &
W(\Psi_1) \ar[u]^{i_{q_0^*}} \ar@/_2pc/[uu]_{i_{\phi_0,q_0^*}}
}
\end{xy}
\end{equation}
The lower left triangle commutes since
$q^* = \llsub{\Gamma_1}q_0^* \circ q_1^*$ by Definition~\ref{defn:composition}.
The commutativity of the lower square follows from the definitions of
$\llsub{\Gamma_1}q_0^*$, $i_{\phi_1}$, and $i_{\bar\phi}$. Finally, the commutativity of diagrams involving each
curved arrow and its two neighbors follows from the definitions of
$i_{\phi,q^*}$, $i_{\phi_0,q_0^*}$, and $i_{\phi_1,q_1^*}$ in \S\ref{para:weyl}.
To conclude that $i_{\phi,q^*}= i_{\phi_0,q_0^*} \circ i_{\phi_1,q_1^*}$,
it suffices to show that the upper right triangle commutes, i.e., that $i_{\phi} = i_{\phi_0} \circ i_{\bar\phi}$.

Let $\alpha\in \lsub{\phi}\Phi_0$. As described in \S\ref{para:weyl},
$i_\phi(w_\alpha)$ (resp.~$i_{\bar\phi}(w_\alpha)$)
is defined in terms of an orthogonal $\phi(\Gamma)$-orbit $\Xi_\alpha$ in $\Phi_0$
(resp.~$\bar\phi(\Gamma_1)$-orbit $\Xi_\alpha^1\subseteq\lsub{\phi_0}\Phi_0$)
attached to $\alpha$.
Analogously, for $\beta\in\lsub{\phi_0}\Phi_0$, $i_{\phi_0}(w_\beta)$
is defined in terms of an
orthogonal $\phi_0(\Gamma_0)$-orbit
$\Xi_\beta^0\subseteq\Phi_0$.

By the definitions of these embeddings, to verify that $i_{\phi} = i_{\phi_0} \circ i_{\bar\phi}$,
it suffices to show that for any $\alpha\in\lsub{\phi}\Phi_0$,
\begin{equation}
\label{eq:orbits}
\Xi_\alpha = \bigcup_{\beta\in\Xi^1_\alpha} \Xi^0_\beta .
\end{equation}

To establish this equality, suppose $\xi\in\Xi^0_\beta$ for some $\beta\in\Xi^1_\alpha$. Then the image of
$\xi$ in $\lsub{\phi_0}\Phi_0$ (resp.~$\beta$ in $\lsub{\phi}\Phi_0$) is
$c_\beta \beta$ for $c_\beta = 1$ or $2$ (resp.\ $c_\alpha\alpha$ for $c_\alpha= 1$ or $2$).
It follows that the image of $\xi$ in $\lsub{\phi}\Phi_0$ is
$c_\alpha c_\beta \, \alpha$. To show that
$\xi\in\Xi_\alpha$, it remains to verify that $c_\alpha c_\beta\alpha$ is the largest positive multiple
of $\alpha$ in $\lsub{\phi}\Phi_0$.

If $c_\alpha = 2$, it follows from \S\ref{para:weyl} that $\beta$
must be contained in an irreducible component of $\lsub{\phi_0}\Phi_0$ of type
$A_{2n}$ on which an element of $\bar\phi(\Gamma_1)$ acts nontrivially.
Since this component is not of type $BC$,
$c_\beta$ must equal $1$, hence $c_\alpha c_\beta \alpha = 2\alpha$
is the largest positive multiple of $\alpha$ in $\lsub{\phi}\Phi_0$.

On the other hand, if $c_\beta = 2$, then from \S\ref{para:weyl},
$\beta$ must be contained in an irreducible component of $\lsub{\phi_0}\Phi_0$ of type $BC_n$.
Since this component is not of type $A$, it follows from \S\ref{para:weyl} that $c_\alpha = 1$. Thus again
$c_\alpha c_\beta\alpha = 2\alpha$
is the largest positive multiple of $\alpha$ in $\lsub{\phi}\Phi_0$.

If $c_\alpha=c_\beta=1$, we must show that $\alpha$ is nonmultipliable in $\lsub{\phi}\Phi_0$.
If it is multipliable, then as in \S\ref{para:weyl}, the $\phi(\Gamma)$-orbit $\Theta$ of roots
in $\Phi_0$ with image $\alpha$ in $\lsub{\phi}\Phi_0$ is not orthogonal.
However, since $c_\beta = 1$, the $\phi(\Gamma_0)$-orbit $\Theta_0$ of roots
in $\Phi_0$ with image $\beta$ in $\lsub{\phi_0}\Phi_0$ is orthogonal.
Let $\Phi'$ be an irreducible component of $\Phi_0$ meeting $\Theta_0\subseteq\Theta$
(which must necessarily be of type $A_{2n}$). Then
$\Phi'\cap\Theta_0$ is orthogonal, but $\Phi'\cap\Theta$ is not. In particular,
$\Phi'\cap\Theta_0\neq\Phi'\cap\Theta$, so
the stabilizer of $\Phi'$ in $\phi(\Gamma_0)$ must act trivially on $\Phi'$, while
the stabilizer of $\Phi'$ in $\phi(\Gamma)$ acts nontrivially on $\Phi'$.
Thus the image $\Phi''$ of $\Phi'$ in $\lsub{\phi_0}\Phi_0$ is isomorphic to $\Phi'$.
The image of $\Theta$ in $\lsub{\phi_0}\Phi_0$ is the $\bar\phi(\Gamma_1)$-orbit $\Theta_1$
of roots in $\lsub{\phi_0}\Phi_0$ that map to $\alpha\in\lsub{\phi}\Phi_0$. 
The image of $\Phi'\cap\Theta$ in $\Phi''$ is $\Phi''\cap\Theta_1$,
and since $\Phi'\cap\Theta$ is not orthogonal, neither is $\Phi''\cap\Theta_1\subseteq\Theta_1$.
But this contradicts $c_\alpha=1$ by \S\ref{para:weyl}. Thus $\alpha$ must be nonmultipliable.

We have therefore shown that the right side of \eqref{eq:orbits} is contained in the left.
To prove the opposite inclusion, suppose that $\xi\in\Xi_\alpha$, and let
$\beta'$ be the image of $\xi$ in $\lsub{\phi_0}\Phi_0$. Let 
$\beta$ be either $\beta'/2$ or $\beta'$ accordingly as $\beta'$ is either divisible or not.
Then it is readily seen that $\beta\in\Xi^1_\alpha$ and $\xi\in\Xi^0_\beta$.
This completes the proof of \eqref{eq:orbits}, and it follows that $i_{\phi} = i_{\phi_0} \circ i_{\bar\phi}$.
\end{proof}

\begin{prop}
\label{prop:composition-equivalence}
Let $G_0$, $G_1$, and $G_2$ be connected reductive $k$-groups,
and let $\Gamma$,
$\Gamma_0$, and $\Gamma_1$ be finite groups. Let
$T_j$ be a maximal $k$-torus of $G_j$ for $j=0,1,2$.
Suppose that $(\phi,q^*)$ is a parascopic datum for
$(G_0,\Gamma,G_2)$ relative to the tori $T_0$ and $T_2$,
and for $j=0,1$, $(\phi_j,q^*_j)$ is a parascopic
datum for $(G_j,\Gamma_j,G_{j+1})$
relative to the tori $T_j$ and $T_{j+1}$.
Suppose that
\IsCompOf{(\phi,q^*)}{(\phi_0,q^*_0)}{(\phi_1,q^*_1)}.
If $(\phi',q^{\prime *})$ is equivalent to $(\phi,q^*)$, then for $j=0,1$, there exist parascopic data
$(\phi'_j,q^{\prime *}_j)$ for $(G_j,\Gamma_j,G_{j+1})$ such that
\begin{itemize}
\item $(\phi'_j,q^{\prime*}_j)$ is equivalent to $(\phi_j,q^*_j)$ for $j = 0,1$,
\item
\IsCompOf{(\phi',q^{\prime *})}{(\phi'_0,q^{\prime*}_0)}{(\phi'_1,q^{\prime *}_1)}.
\end{itemize}
\end{prop}

\begin{proof}
For $j = 0,1,2$, let $W_j$ denote $W(G_j,T_j)$.
Suppose that $(\phi',q'^*)$ is equivalent to $(\phi,q^*)$ via the elements
$g_2\in G_2(k\sep)$ and $ g_0\in G_0(k\sep)$.
For $j = 0,2$, $g_j\inv\sigma(g_j)\in N_{G_j}(T_j)(k\sep)$ for all $\sigma\in\Gal(k)$.
Moreover, the map that takes $\sigma$ to the image of
$g_j\inv\sigma(g_j)$ in $W_j$ is a cocycle $c_j$ in $Z^1(k,W_j)$.
It follows from~\cite{adler-lansky:lifting}*{Lemma 7.4} that the map
$\abmap{Z^1(k,W_2)}{Z^1(k,W_0)}$
induced by $\map{i_{\phi,q^*}}{W_2}{W_0}$ takes $c_2$ to $c_0$.

Let $c_1\in Z^1(k,W_1)$ be the image of $c_2$ under the map
$\abmap{Z^1(k,W_2)}{Z^1(k,W_1)}$ induced by
$\map{i_{\phi_1,q_1^*}}{W_2}{W_1}$.
From \cite{raghunathan:tori}*{Theorem 1.1},
there exists $g_1\in G_1(k\sep)$ such that
the image of $g_1\inv\sigma(g_1)$ in $W(G_1,T_1)$ coincides with
$c_1(\sigma)$ for each $\sigma\in\Gal(k)$.
Also, it follows from Lemma~\ref{lem:weyl-embedding-factorization} that $c_0$ is the
image of $c_1$ under the map $\abmap{Z^1(k,W_1)}{Z^1(k,W_0)}$
induced by $\map{i_{\phi_0,q_0^*}}{W_1}{W_0}$.

For $j = 0,1$, define $\phi'_j$ and $q'^*_j$ as in 
Definition \ref{defn:equivalent-data}.
That is, for each
$\gamma\in\Gamma_j$, let $\phi'_j(\gamma) = {\Int(g_j)^*}\inv \circ \phi_j(\gamma) \circ \Int(g_j)^*$,
and define $q'^*_j$ to be the map
${\lsub{\Gamma_j}\Int(g_j)^*}\inv \circ q^*_j \circ \Int(g_{j+1})^*$, where 
$\lsub{\Gamma_j}\Int(g_j)^*$ is the isomorphism from
$\lsub{\phi'_j}\Psi(G_j, \lsup{g_j}T_j)$ to $\lsub{\phi_j}\Psi(G_j,T_j)$ induced by
$\Int(g_j)^*$. For $j = 0,1$, it follows from~\cite{adler-lansky:lifting}*{Lemma 7.4}
that the pair $(\phi'_j,{q_j^{\prime *}})$ is a parascopic
datum for $(G_j,\Gamma,G_{j+1})$ that is evidently equivalent to
$(\phi_j,q^*_j)$ via $g_{j+1}$ and $g_j$.

It is now easily verified that
\IsCompOf{(\phi',q'^*)}{(\phi'_0,q'^*_0)}{(\phi'_1,q'^*_1)}.
\end{proof}

\begin{prop}
\label{prop:conorm-factorization}
Maintaining the notation from Proposition~\ref{prop:composition-equivalence},
for $j=0,1,2$, suppose that $G_j$ is $k$-quasisplit, 
and let $G^\wedge_j$ be a connected reductive quasisplit
$k$-group in $k$-duality with $G_j$.
For $j=0,1$,
consider the maps
$
\map
{\dnorm[{[\phi_j,q^*_j]}]}
{\Cl(G^\wedge_{j+1})}
{\Cl(G^\wedge_j)}
$
and
$
\map
{\dnormst[{[\phi_j,q^*_j]}]}
{\Cl\st(G^\wedge_{j+1})}
{\Cl\st(G^\wedge_j)}
$,
and also the maps 
$\dnorm[{[\phi,q^*]}]$
and
$\dnormst[{[\phi,q^*]}]$.
Then
$$
\dnorm[{[\phi,q^*]}]
=
\dnorm[{[\phi_0,q^*_0]}]
\circ
\dnorm[{[\phi_1,q^*_1]}]
\qquad\text{and}\qquad
\dnormst[{[\phi,q^*]}]
=
\dnormst[{[\phi_0,q^*_0]}]
\circ
\dnormst[{[\phi_1,q^*_1]}].
$$
\end{prop}

\begin{proof}
Our parascopic data
$(\phi_0,q^*_0)$,
$(\phi_1,q^*_1)$,
and
$(\phi,q^*)$,
induce embeddings
of the sets of stable conjugacy classes of maximal $k$-tori
(from Proposition~\ref{prop:stable-tori}).
From Lemma \ref{lem:weyl-embedding-factorization},
these embeddings are transitive:
$$
\begin{xy}
\xymatrix{
\TTst(G_2,k) \ar[r] \ar@/^1pc/[rr]
&\TTst(G_1,k) \ar[r]
&\TTst(G_0,k).
}
\end{xy}
$$

Let $s\in G_2^\wedge(\bar k)$, and let $S_2^\wedge\subseteq G_2^\wedge$ be a maximal $k$-torus containing $s$.
Then as in \S\ref{para:square-of-tori},
our parascopic datum $(\phi_1,q^*_1)$ gives us maximal
$k$-tori $S_2\subseteq G_2$, $S_1\subseteq G_1$, and $S_1^\wedge\subseteq G_1^\wedge$,
unique up to stable conjugacy,
and $(\phi_0,q^*_0)$ gives
$S_0\subseteq G_0$ and $S_0^\wedge\subseteq G_0^\wedge$.
From the previous paragraph, we may assume that $(\phi,q^*)$ gives
$S_2\subseteq G_2$, $S_0\subseteq G_0$, and $S_0^\wedge\subseteq G_0^\wedge$.
As in \S\ref{para:square-of-tori},
for $j=0,1,2$, let $\delta_j$ be a map implementing the $k$-duality
between $S_j$ and $S_j^\wedge$. The parascopic data together with the
duality maps
$\delta_j$ thus give maps
$$
\map{\dnorm[\phi_1,q^*_1]}{S_2^\wedge}{S_1^\wedge},
\quad
\map{\dnorm[\phi_0,q^*_0]}{S_1^\wedge}{S_0^\wedge},
\quad
\map{\dnorm[\phi,q^*]}{S_2^\wedge}{S_0^\wedge},
$$
and it will be enough to show that
\begin{equation}
\label{eq:conorm-factorization}
\dnorm[\phi,q^*] (s) = 
\dnorm[\phi_0,q^*_0]
\bigl(
\dnorm[\phi_1,q^*_1] (s)
\bigr) .
\end{equation}
This will follow from the corresponding equality
$$
\big(\dnorm[\phi,q^*]\big)_* = 
\big(\dnorm[\phi_0,q^*_0]\big)_*
\circ
\big(\dnorm[\phi_1,q^*_1]\big)_*
$$
of maps of cocharacters. By the definition of the conorm \eqref{eq:norm-char},
this
is equivalent to the equation
\begin{equation}
\label{eq:normchar-composition}
\norm[\phi,q^*]^*
= 
\norm[\phi_0,q^*_0]^*
\circ
\norm[\phi_1,q^*_1]^* .
\end{equation}
To see that this equation holds, consider the diagram
\begin{equation}
\label{eqn:weyl}
\begin{xy}
\xymatrix{
& &
X^*(T_0)\\
&
\lsub{\phi}X^*(T_0) \ar[r]^{\normchar[\bar\phi]} \ar[ur]^{\normchar[\phi]} &
\lsub{\phi_0}X^*(T_0) \ar[u]^{\normchar[\phi_0]}\\
X^*(T_2) \ar[r]^{q_1^*} \ar[ur]^{q^*} \ar@/^2pc/[uurr]^{\normchar[\phi,q^*]}
	\ar@/_2pc/[rr]_{\normchar[\phi_1,q^*_1]} &
\lsub{\phi_1}X^*(T_1) \ar[r]^{\normchar[\phi_1]} \ar[u]_{\llsub{\Gamma_1}q_0^*} &
X^*(T_1) \ar[u]^{q_0^*} \ar@/_2pc/[uu]_{\normchar[\phi_0,q^*_0]}
}
\end{xy}
\end{equation}
The upper right triangle commutes because
$$\sum_{\bar\gamma\in\Gamma_1}\bar\phi(\bar\gamma) \sum_{\gamma_0\in\Gamma_0} \phi_0(\gamma_0) = \sum_{\gamma\in\Gamma }\phi(\gamma).$$
The commutativity of the lower left triangle, the lower square, and the diagrams involving each
curved arrow and its two neighbors all hold by reasoning analogous to that in \eqref{eq:weyl},
and \eqref{eq:normchar-composition} follows.
\end{proof}

\section{Decompositions coming from subnormal series for $\Gamma$}
\label{sec:decomposition-gamma}

In this section and the next, we provide specific decompositions of an arbitrary
parascopic datum into simpler data.  
From Proposition \ref{prop:conorm-factorization},
if one wants to understand
conorms in an explicit way, it will then be sufficient to understand
them in the cases of these simpler data.

\begin{prop}
\label{prop:decomposition-gamma}
Suppose
\begin{itemize}
\item
$\Psi_0$ and $\Psi_2$ are root data, each with an action of $\Gal(k)$,
\item
$(\phi,q^*)$ is a parascopic datum for $(\Psi_0,\Gamma,\Psi_2)$,
and
\item
$1 \longrightarrow \Gamma_0 \longrightarrow \Gamma \longrightarrow \Gamma_1 \longrightarrow 1$
is an exact sequence of groups.
\end{itemize}
Then there exist
\begin{itemize}
\item
a root datum $\Psi_1$ with an action of $\Gal(k)$,
and
\item
for $i=0,1$,
parascopic data 
$(\phi_i, q^*_i)$ for $(\Psi_i,\Gamma_i,\Psi_{i+1})$,
\end{itemize}
such that \IsCompOf{(\phi,q^*)}{(\phi_0,q^*_0)}{(\phi_1,q^*_1)}.
Moreover, if $(\phi,q^*)$ is torus-inclusive
(resp.\ root-inclusive),
then we can take 
$(\phi_0,q^*_0)$
and
$(\phi_1,q^*_1)$
to be as well.
In all cases, we can take $(\phi_0,q_0^*)$ to be torus-inclusive.
\end{prop}

\begin{proof}
For $j = 0,2$, write
$\Psi_j = (X^*_j, \Phi_j, X_{j*}, \Phi^\vee_j)$.
Let $\phi_0$ denote the restriction of $\phi$ to $\Gamma_0$.
Let $X_1^* = \lsub{\phi_0}X_0^*$
and
$X_{1*} = \lsub{\phi_0}X_{0*}$.

The action $\map{\phi}{\Gamma}{\Aut(X_0^*)}$
induces an action $\map{\phi_1}{\Gamma_1}{\Aut(X_1^*)}$.
Define
$\map{q^*_1}{X_2^*}{\lsub{\phi_1}X_1^*=\lsub{\phi}X_0^*}$ by $q^*_1 = q^*$,
and define 
$\map{q^*_0}{X_1^*}{\lsub{\phi_0}X_0^* = X_1^*}$ to be the identity map.

Let $\Phi_1$ denote any $\Gal(k)$-invariant subsystem of $\lsub{\phi_0}\Phi_0$
such that 
$\lsub{\phi_1}\Phi_1 \supseteq q_1^*(\Phi_2)$.
(Such a subsystem always exists.
For example, we could take $\Phi_1 = \lsub{\phi_0}\Phi_0$.)
Let $\Phi_1^\vee = \set{\alpha^\vee \in \lsub{\phi_0}\Phi_0^\vee}{\alpha \in \Phi_1}$.
Let $\Psi_1 = (X_1^*, \Phi_1, X_{1*}, \Phi_1^\vee)$.
Then $\Psi_1$ is a $\Gal(k)$-invariant root datum,
and we claim that
$(\phi_0,q_0^*)$
and
$(\phi_1,q_1^*)$
are the desired parascopic data.

Note that by its definition, $q_j$ is a morphism from $\Psi_{j+1}$ to $\lsub{\phi_j}\Psi_j$ for
$j = 0,1$.
Moreover, since $q_0^*$ is an isomorphism of lattices, we have that
$(\phi_0,q_0^*)$ is torus-inclusive.

We now show that 
\IsCompOf{(\phi,q^*)}{(\phi_0,q^*_0)}{(\phi_1,q^*_1)}.
Since $q_0^*$ is the identity 
map on $X_1^* = \lsub{\phi_0}X_0^*$,
it is obviously equivariant with respect to our action $\phi_1$
of $\Gamma_1$ on $X_1^*$.
Since $q^* = q_1^*$ and $\lsub{\Gamma_1}q_0^*$ is the identity map on
$\lsub{\phi_1}X_1^*=\lsub{\phi}X_0^*$,
we have that
$q^* = \lsub{\Gamma_1}q_0^* \circ q_1^*$.
Thus,
\IsCompOf{(\phi,q^*)}{(\phi_0,q^*_0)}{(\phi_1,q^*_1)}.

Suppose that the parascopic datum $(\phi, q^*)$ is torus-inclusive.
That is, $q^*$ is an isomorphism of lattices, and therefore, so is $q_1^*$,
so $(\phi_1,q_1^*)$ is torus-inclusive.
We have already seen that $(\phi_0,q_0^*)$ is torus-inclusive.

Suppose that the parascopic datum $(\phi, q^*)$ is root-inclusive.
That is, $q^*(\Phi_2) \subseteq \lsub{\phi}\Phi_0\,\red$.
Recall that we allowed $\Phi_1$ to be any root subsystem 
of $\lsub{\phi_0}\Phi_0$ with $\lsub{\phi_1}\Phi_1 \supseteq q_1^*(\Phi_2)$.
For $j = 0,1$, we must show that
$q_j^*(\Phi_{j+1}) \subseteq \lsub{\phi_j}\Phi_j\,\red$.
Now $q_1^*(\Phi_2)
=
q^*(\Phi_2)
\subseteq
\lsub{\phi}\Phi_0\,\red
=
\lsub{\phi_1}(\lsub{\phi_0}\Phi_0)\red$
and also $q_1^*(\Phi_2)\subseteq \lsub{\phi_1}\Phi_1$, so
$q_1^*(\Phi_2)\subseteq \lsub{\phi_1}(\lsub{\phi_0}\Phi_0)\red\cap \lsub{\phi_1}\Phi_1
= \lsub{\phi_1}\Phi_1\,\red$, as desired.

Meanwhile, since $q_0=\id$, the condition $q_0^*(\Phi_1) \subseteq \lsub{\phi_0}\Phi_0\,\red$
becomes $\Phi_1 \subseteq \lsub{\phi_0}\Phi_0\,\red$.
Thus we may take any $\Phi_1$ satisfying both $\lsub{\phi_1}\Phi_1 \supseteq q_1^*(\Phi_2)$
and this latter condition. 
There do indeed exist such $\Phi_1 $; e.g.,
$\Phi_1$ can be taken to be $ \lsub{\phi_0}\Phi_0\,\red$ itself since
$$q_1^*(\Phi_2) = (\lsub{\Gamma_1}q_0 \circ q_1^*)(\Phi_2) = q^*(\Phi_2)\subseteq
\lsub{\phi}\Phi_0\,\red\ = \lsub{\phi_1}(\lsub{\phi_0}\Phi_0)\red \subseteq
\lsub{\phi_1}(\lsub{\phi_0}\Phi_0\, \red).$$
\end{proof}

\begin{rem}
\label{rem:group-construction}
\label{rem:group-construction-solution}
If $G_0$ and $G_2$ are connected reductive $k$-groups
and the pair $(\phi,q^*)$ is a parascopic datum for $(G_0,\Gamma,G_2)$,
then Proposition~\ref{prop:decomposition-gamma}
gives us a root datum $\Psi_1$ with an action of $\Gal(k)$.
Thus, we have a connected reductive $k$-group $G_1$, determined up to its inner form.
For any choice of $G_1$, Proposition~\ref{prop:decomposition-gamma},
then gives us \emph{weak} parascopic data
$(\phi_j,q^*_j)$ ($j = 0,1$) for $(G_j,\Gamma_j,G_{j+1})$.
In order for these pairs to be parascopic data for the triples in question,
they must satisfy an additional condition that guarantees
transfers of tori
from $G_2$ to $G_1$ and from $G_1$ to $G_0$
(see Definition~\ref{defn:parascopic-group}).
In general,
we do not know if one can always choose an inner form of $G_1$ so that this
condition is satisfied by both pairs.
In certain cases,
there is a natural choice for an inner form of $G_1$.
For example:
\begin{enumerate}[(a)]
\item
\label{item:quasi-split}
If $G_0$ is quasisplit, then we can take $G_1$ to be quasisplit.
If, further, 
$\Gamma$ preserves a $\Gal(k)$-invariant positive system in $\Phi_0$,
then $G_1$ must be quasisplit.
\item
\label{item:k-good}
If $k$ is perfect and of cohomological dimension $\leq 1$, then all reductive $k$-groups are 
quasisplit, and
so $G_1$ is uniquely determined up to isomorphism.
\item
\label{item:fixed-point}
If $\phi$ extends to give an action of $\Gamma$ on $G_0$,
and $G_2 = (G_0^\Gamma)\conn$,
then we may take
$G_1 = (G_0^{\Gamma_0})\conn$.
From Example \ref{ex:parascopy-fixed-point},
the action $\phi$, together with appropriate maps
of character lattices of tori,
gives parascopic data for $(G_0,\Gamma_0,G_1)$ and $(G_1, \Gamma_1, G_2)$.
\end{enumerate}
\end{rem}

\section{Decompositions into parascopies coming from actions on reductive groups}
\label{sec:fixed-pinning}
In this section,
let $\tilde{G}$ denote a connected reductive quasisplit $k$-group and 
$\Gamma$ a finite group.
\begin{defn}
We will say that a homomorphism
$\map{\varphi}{\Gamma}{\Aut_k(\tilde G)}$ is \emph{quasicentral} if
$\varphi(\Gamma)$ stabilizes some Borel subgroup $\tilde{B}_0 \subseteq \tilde{G}$,
some maximal torus $\tilde{T}_0\subseteq \tilde{B}_0$,
and some pinning of $(\tilde{G},\tilde{B}_0,\tilde{T}_0)$
(not necessarily defined over $k$). If moreover we can take
$\tilde B_0$ and $\tilde T_0$ to be defined over $k$,
and the pinning to be invariant under the action of $\Gal(k)$,
then we will say that $\varphi$ is \emph{$k$-quasicentral}.
\end{defn}

\begin{rem}
\label{rem:quasicentral}
We note that if $\varphi$ is quasicentral, then $\varphi(\Gamma)$
must stabilize a pinning with respect to \emph{any} $\varphi(\Gamma)$-stable pair
$\tilde{T}_0\subseteq\tilde{B}_0$,
since any two such pairs are conjugate by some element of
$(G^{\varphi(\Gamma)})\conn(k)$
(see~\cite{adler-lansky:lifting}*{Proposition 3.5(ii)}).
\end{rem}

\begin{prop}
\label{prop:decomp-fixed-pinning}
Let $G$ be a connected reductive quasisplit $k$-group, and let
$(\phi,q^*)$ be a root-inclusive and torus-inclusive
parascopic datum for $(\tilde{G},\Gamma,G)$ relative to the maximal $k$-tori
$\tilde{T}\subseteq\tilde G$ and $T\subseteq G$.
Then there exists a maximal $k$-torus $\tilde T'$ of $\tilde G$ stably conjugate to $\tilde T$,
and a $k$-quasicentral homomorphism
$\map{\varphi}{\Gamma}{\Aut_k(\tilde G)}$
with image stabilizing $\tilde T'$ that possesses the following property.
Let $\bar G = (\tilde{G}^{\varphi(\Gamma)})\conn$
and $\bar T = (\tilde{T}'\,^{\varphi(\Gamma)})\conn$.
Then $\bar G$ is $k$-quasisplit,
and $(\phi,q^*)$ is a composition of
\begin{itemize}
\item a parascopic datum for $(\tilde G,\Gamma,\bar G)$
relative to $\tilde T$ and $\bar T$ equivalent to a parascopic
datum arising from $\varphi$ as in Example~\ref{ex:parascopy-fixed-point}, and
\item a parascopic datum for $(\bar G,1,G)$ relative to $\bar T$ and $T$
arising from an embedding from $\Psi(G,T)$ into $\Psi(\bar G,\bar T)$.
\end{itemize}
Both of these factors are root-inclusive and torus-inclusive. Moreover,
with respect to the above parascopic data,
$G$ is a parascopic group for $(\tilde G,\Gamma)$ and
for $(\bar G, 1)$, and $\bar G$ is a parascopic group for $(\tilde G,\Gamma)$.
\end{prop}

\begin{proof}
By~\cite{adler-lansky:data-actions}*{Remarks 25 and 27},
there exists a maximal $k$-torus $\tilde T'$ of $\tilde G$
stably conjugate to $\tilde T$,
an element $\tilde g\in\tilde G$, and a
$k$-quasicentral homomorphism
$\map{\varphi}{\Gamma}{\Aut_k(\tilde G)}$
such that $\varphi(\Gamma)$ stabilizes $\tilde T'$ and 
$\Int(\tilde g)$ induces a $\Gal(k)\times\Gamma$-isomorphism from
the root datum $\Psi(\tilde G,\tilde T')$
(with the action $\phi'$ of $\Gamma$ induced by $\varphi$)
to the root datum $\Psi(\tilde G,\tilde T)$ (with 
the action $\phi$).
Let $\tilde B_0$ be a $\varphi(\Gamma)$-stable Borel $k$-subgroup of $\tilde G$.
Then $(\tilde{B}_0^{\varphi(\Gamma)})\conn$ is a Borel $k$-subgroup of $\bar G$
by~\cite{adler-lansky:data-actions}*{Proposition 3.5(ii)}. Hence $\bar G$ is $k$-quasisplit.

By Example~\ref{ex:parascopy-fixed-point},
$\Psi(\bar G,\bar T)$ can be identified with
$\lsub{\phi'}\Psi(\tilde G,\tilde T')\red$,
and
the pair $(\phi',\id)$ is thus
a parascopic datum for $(\tilde G,\Gamma,\bar G)$ relative to
$\tilde T'$ and $\bar T$.
Since $\map{\Int(\tilde g)^*}{X^*(\tilde T')}{X^*(\tilde T)}$ is a
$\Gal(k)\times\Gamma$-isomorphism,
we obtain a $\Gal(k)$-isomorphism
$\lsub\Gamma\Int(\tilde g)^*$
from $\Psi(\bar G,\bar T) = \lsub{\phi'}\Psi(\tilde G, \tilde T')\red$
to $\lsub\phi \Psi(\tilde G, \tilde T)\red$.
Then $(\phi,\lsub\Gamma\Int(\tilde g)^*)$ is a
parascopic datum for $(\tilde G,\Gamma,\bar G)$ relative to
$\tilde T$ and $\bar T$ that is equivalent to $(\phi',\id)$.
This relationship is represented in the upper square of the following diagram,
where the label on each dotted, non-horizontal line represents
a parascopic datum for the root data at the upper and lower ends of the
dotted line.

$$
\xymatrix@C+2em{
&
\Psi(\tilde G, \tilde T')
\ar[r]^{\Int(\tilde g)^*}_{\sim}
\ar@{.}[d]_{(\phi',\id)}
&
\Psi(\tilde G, \tilde T)
\ar@{.}[d]_{(\phi,\id)} 
\ar@/^3pc/@{.}[dd]^{(\phi,q^*)}  
\\
\Psi(\bar G, \bar T) \ar@{=}[r]
&
\lsub{\phi'}\Psi(\tilde G, \tilde T')\red
\ar@{.}[ur]^{(\phi,\lsub\Gamma\Int(\tilde g)^*)}
\ar[r]^{\lsub\Gamma\Int(\tilde g)^*}_{\sim}
\ar@{.}[dr]_{(1, (\lsub\Gamma \Int(\tilde g)^*)\inv \circ q^*)\phantom{xxx}}
&
\lsub{\phi}\Psi(\tilde G, \tilde T)\red
\ar@{.}[d]_{(1, q^*)}
\\
&& \Psi(G,T)
}
$$

By Remark \ref{rem:decomp-canonical},
the torus-inclusive datum $(\phi,q^*)$
is a composition of the parascopic datum 
$(\phi, \id)$ for
$(\Psi(\tilde G,\tilde T),\Gamma,\lsub\phi\Psi(\tilde G,\tilde T)\red)$
and the parascopic datum $(1, q^*)$ for
$(\lsub\phi\Psi(\tilde G,\tilde T)\red,1,\Psi(G,T))$
provided in Example~\ref{ex:parascopy-gamma-trivial}.
This composition is illustrated in the right-hand column of the diagram
above.
It follows that
$(\phi, q^*)$ is also a composition of
the parascopic datum
$(\phi, \lsub\Gamma\Int(\tilde g)^*)$
for $(\tilde G,\Gamma, \bar G)$ with respect to $\tilde T$ and $\bar T$,
and the datum
$(1, (\lsub\Gamma\Int(\tilde g)^*)\inv \circ q^*)$
for $(\bar G, 1, G)$ with respect to $\tilde T$ and $T$.

The statement concerning root- and torus-inclusivity
follows from their definitions,
while the final statement of the proposition follows from
Proposition \ref{prop:stable-tori}
and the fact that $\tilde{G}$ and $\bar{G}$ are quasisplit.
\end{proof}

\begin{prop}
\label{prop:quasi-central-dual-closed}
Let $\map{\varphi,\bar\varphi}{\Gamma}{\Aut_k(\tilde G, \tilde T)}$
be homomorphisms with $\bar\varphi$ $k$-quasicentral.
Suppose that both $\varphi$ and $\bar\varphi$ 
induce the action $\map{\phi}{\Gamma}{\Aut_k\Psi(\tilde G, \tilde T)}$
and that
\begin{itemize}
\item for every $\alpha\in\Phi(\tilde G,\tilde T)$,
$\stab_\Gamma(\alpha)$ is either $\Gamma$ or $1$;
\item for every irreducible component $\tilde\Phi$ of $\Phi(\tilde G,\tilde T)$, 
$\stab_\Gamma(\tilde\Phi)$ acts faithfully on $\tilde\Phi$.
\end{itemize}
Let $T = (\tilde{T}^{\varphi(\Gamma)})\conn
= (\tilde{T}^{\bar\varphi(\Gamma)})\conn$,
$G = (\tilde{G}^{\varphi(\Gamma)})\conn$,
and
$\bar G = (\tilde{G}^{\bar\varphi(\Gamma)})\conn$.
Then $\Psi(G,T)$ is a dual-closed subdatum of $\Psi(\bar G, T)$.
\end{prop}

\begin{proof}
Since $\bar\varphi$ is quasicentral, $\phi$ must preserve a positive system in $\Psi(\tilde G, \tilde T)$.
By Example~\ref{ex:parascopy-fixed-point},
$\Psi(\bar G, T) = \lsub{\phi}\Psi(\tilde G, \tilde T)\red$,
so the identity map on $X^*(T)$ gives an embedding
of $\Psi(G,T)$ into $\Psi(\bar G, T)$.
Clearly then, to show that $\Psi(G,T)$ is dual-closed in $\Psi(\bar G, T)$,
it suffices to show that
$\Phi^\vee(G,T)$ contains all of the long coroots in $\Phi^\vee(\bar G,T)$, or equivalently,
that $\Phi(G,T)$ contains all of the short roots in $\Phi(\bar G,T)$.

Let $\map{i_*}{X_*(T)}{X_*(\tilde T)}$
be the inclusion map. The adjoint of $i_*$
is the natural quotient
$\map{i^*}{X^*(\tilde T)}{X^*(T)}$.
Since $\Phi(G,T)\subseteq\Phi(\bar G,T)\subseteq i^*(\Phi(\tilde G,\tilde T))$,
to prove this proposition, it is enough to prove that for any irreducible component
$\tilde\Phi$ of $\Phi(\tilde G,\tilde T)$,
the system $\Phi = i^*(\tilde\Phi)\cap\Phi(G,T)$ contains all of the short roots in the system $\bar\Phi =
i^*(\tilde\Phi)\cap\Phi(\bar G,T) = i^*(\tilde\Phi)\cap\lsub{\phi}\Phi(\tilde G,\tilde T)\red$.

If the action of $\stab_\Gamma(\tilde\Phi)$ on such a component $\tilde\Phi$ is trivial,
then $\stab_\Gamma(\tilde\Phi) = 1$ since
the action of $\stab_\Gamma(\tilde\Phi)$ is faithful. Thus every $\beta\in\tilde\Phi$ has trivial stabilizer,
and it follows from an argument similar to that in~\cite{steinberg:endomorphisms}*{Theorem 8.2($2''''$)} that
$i^*\beta\in\Phi$. Thus, $\Phi = \bar\Phi$, completing the proof in this case.

Now suppose that the action of $\stab_\Gamma(\tilde\Phi)$ on $\tilde\Phi$ is nontrivial.
Consider first the case in which $\tilde\Phi$ is not of type $A_{2n}$.
As above,
for all $\beta\in\tilde\Phi$ having trivial stabilizer in $\Gamma$,
$i^* \beta$
lies in $\Phi$.
Thus in this case,
it suffices to show that all short roots of
$\bar\Phi$
have the form $i^*\beta$ for such 
$\beta\in \tilde\Phi$.

Since the action of $\stab_\Gamma(\tilde\Phi)$ on $\tilde\Phi$ is nontrivial,
$\tilde\Phi$ must be simply laced
and must contain a root $\alpha$ that is fixed by $\stab_\Gamma(\tilde\Phi)$ and
a nonorthogonal root $\beta$ that is not fixed by $\stab_\Gamma(\tilde\Phi)$.
By assumption, $\stab_\Gamma(\beta)$ must be trivial.
Using the definition of the coroot system $\lsub\phi\Phi(\tilde G,\tilde T)^\vee$
from \cite{adler-lansky:data-actions}*{Theorem 7},
we have
\[
\langle i^*\alpha, (i^*\beta)^\vee\rangle
=
\langle \alpha, i_*((i^*\beta)^\vee)\rangle
=
\Bigl\langle
	\alpha,  \sum_{\theta\in\Gamma\cdot\beta}\theta^\vee
\Bigr\rangle
=
|\stab_\Gamma(\tilde\Phi)\cdot\beta|\langle\alpha ,\beta^\vee\rangle,
\]
while
\[
\langle i^*\beta, (i^*\alpha)^\vee\rangle
=
\langle \beta, i_*((i^*\alpha)^\vee)\rangle
=
\langle \beta, \alpha^\vee\rangle
=
\langle \alpha, \beta^\vee\rangle ,
\]
where the last equality holds because $\Phi(\tilde G,\tilde T)$ is simply laced.
Since $\langle \alpha, \beta^\vee\rangle\neq 0$, we therefore have
$\langle i^*\alpha, (i^*\beta)^\vee\rangle > \langle i^*\beta, (i^*\alpha)^\vee\rangle$,
so $i^*\alpha$ is a long root and $i^*\beta$ is a short root.

Now consider an arbitrary short root in $\bar\Phi$.
This root is of the form $i^*\beta'$ for some
$\beta'\in \Phi(\tilde G, \tilde T)$.
Since $i^*\beta$ and $i^*\beta'$ have the same length, there
must be an element $w \in W(\bar G, T)$
such that $w(i^*\beta)= i^*\beta'$.
The embedding $\map{i_\phi}{W(\bar G, T)}{W(\tilde G, \tilde T)}$
(see \S\ref{para:weyl}) has the property that
$i^*\beta' = w (i^*\beta) = i^* (i_\phi(w)\beta)$.
Since $i_\phi(w)$ is fixed by $\Gamma$,
the stabilizer of $i_\phi(w)\beta$ in $\Gamma$, like that of $\beta$,
is trivial.
(In fact, we note that a similar argument shows that
a root in $\Phi(\bar G,T)$
is short if and only if
it is the image under $i^*$ of a root with trivial stabilizer.)
This completes the proof in the case where $\Psi(\tilde{G}, \tilde{T})$
is not of type $A_{2n}$.

Now suppose that $\tilde\Phi$ is of type $A_{2n}$.
Let $\tilde H$ be the connected reductive subgroup of $\tilde G$ of full rank with root system
$\bigcup_{\gamma\in\Gamma} \gamma\cdot\tilde\Phi$.
Note that $\Gamma$ acts on $\tilde H$ via
both $\varphi$ and $\bar\varphi$.
Since the action of $\stab_\Gamma(\tilde\Phi)$ on $\tilde\Phi$ is faithful,
no simple root in $\tilde\Phi$ (with respect to a chosen $\Gamma$-stable Borel subgroup of $\tilde G$
containing $\tilde T$)
is fixed by $\stab_\Gamma(\tilde\Phi)$. Thus every simple root of $\tilde\Phi$ has trivial stabilizer.
The same is true for each conjugate $\gamma\cdot\tilde\Phi$ for $\gamma\in\Gamma$.
The existence of a $\varphi(\Gamma)$-stable pinning of $\tilde H$ follows immediately from this.
Thus the action $\varphi$ on $\tilde H$ is quasicentral.
Letting $H = G\cap\tilde H$, it follows from Example \ref{ex:parascopy-fixed-point} that
$\Phi(H,T) = \lsub{\phi}\Phi(\tilde H,\tilde T)\red$ and hence that $\Phi = \bar\Phi$, completing the proof.
\end{proof}

\section{The conorm as a function on points}
\label{sec:conormfunc}

Let $\tilde{G}$ and $G$ denote connected $k$-quasisplit reductive $k$-groups,
and $\Gamma$ a finite group.
Let $\tilde{G}^\wedge$ and $G^\wedge$ denote $k$-groups that are in $k$-duality
with $\tilde{G}$ and $G$

\begin{defn}
\label{defn:conorm-function}
Suppose $(\phi,q^*)$ is a parascopic datum for $(\tilde{G},\Gamma,G)$.
A \emph{pointwise conorm}
for the class $[\phi,q^*]$ is an algebraic $k$-morphism (not necessarily a group homomorphism)
$\map{\dnormfunc}{G^\wedge}{\tilde{G}^\wedge}$ possessing the following properties.
\begin{enumerate}
\item $\dnormfunc$ descends to a map on geometric conjugacy classes.
\item For each maximal $k$-torus $S^\wedge\subseteq G^\wedge$,
there exist maximal $k$-tori $S\subseteq G$, $\tilde S\subseteq\tilde G$,
and $\tilde S^\wedge\subseteq\tilde G^\wedge$ (as in \S\ref{para:conorm-defn}), a parascopic datum $(\phi',{q'}^*)$
relative to $\tilde S$ and $S$ in the class $[\phi,q^*]$, and duality maps 
$\map{\delta}{\bX^*(S)}{\bX_*(S^\wedge)}$ and $\map{\tilde\delta}{\bX^*(\tilde{S})}{\bX_*(\tilde{S}^\wedge)}$
such that the restriction of $\dnormfunc$ to $S^\wedge$
coincides with the conorm $\map{\dnorm[\phi',{q'}^*,\tilde\delta,\delta]}{S^\wedge}{\tilde{S}^\wedge}$.
\end{enumerate}
 \end{defn}

\begin{example}
\label{ex:conormfunc-isogeny}
Suppose that
$\map{f}{\tilde G}{G}$ is a $k$-isotypy.
Let $\tilde T$ be a maximal
$k$-torus of $\tilde G$ and let $T$ be the (unique) maximal torus in $G$
containing $f(\tilde{T})$.
Let $\tilde\Psi = \Psi(\tilde{G},\tilde{T})$ and $\Psi=\Psi(G,T)$.
Then one obtains a morphism
$f^*$ from $\Psi$ to $\tilde\Psi$,
and thus a parascopic datum
$(1,f_*)$
for $(\tilde\Psi,1,\Psi)$, hence for $(\tilde G, 1,G)$,
as in 
Example~\ref{ex:parascopy-gamma-trivial}.
The isotypy
$\map{f^\wedge}{G^\wedge}{\tilde{G}^\wedge}$
dual to $f$
(\S\ref{para:duality})
is clearly compatible with conorm maps
$\dnorm[S^\wedge]$ on tori,
hence is a pointwise conorm for the class $[1,f_*]$.
\end{example}

\begin{example}
\label{ex:trivial-action}
Suppose $\Gamma$ acts trivially on $\tilde G$,
and $G := \tilde{G}^\Gamma = \tilde{G}$.
Then we have a pointwise conorm for the class $[1, \id]$
given by $s \mapsto s^{|\Gamma|}$.
\end{example}

\begin{example}
\label{ex:permutation-action}
Suppose that $\tilde G = \prod_{i=1}^r G$,
$\Gamma$ is a group of order $r$,
and $\Gamma$ acts on $\tilde G$ by permuting the factors
simply and transitively.
Then $\tilde G^\wedge$ is $k$-isomorphic to $\prod_{i=1}^r G^\wedge$,
and it is easy to check that the diagonal embedding
$\map{\diag}{G^\wedge}{\tilde G^\wedge}$
is a pointwise conorm for this action.
\end{example}

The following result on pointwise conorms for parascopic data obtained via composition
(see \S\ref{sec:decomposition-defn}) follows easily from \eqref{eq:conorm-factorization}.
\begin{lem}
\label{lem:pointwise-conorm-factorization}
In the setting of Proposition~\ref{prop:conorm-factorization},
if $\dnormfunc_j$ is a pointwise conorm for
$[\phi_j,q^*_j]$ for $j = 0,1$, then $\dnormfunc_0\circ\dnormfunc_1$
is a pointwise conorm function for $[\phi,q^*]$.
\end{lem}

\begin{prop}
\label{prop:conorm-func-hom}
Suppose $(\phi,q^*)$ is a parascopic datum for $(\tilde{G},\Gamma,G)$
relative to $k$-tori $\tilde{T}\subseteq \tilde{G}$
and $T\subset G$. Choosing dual maximal $k$-tori $T^\wedge\subseteq G^\wedge$
and $\tilde T^\wedge\subseteq\tilde G^\wedge$ along with appropriate duality maps
$\delta$ and $\tilde\delta$,
we obtain a conorm
$\map{\dnorm[T^\wedge]=\dnorm[\phi,q^*,\delta,\tilde\delta]}{T^\wedge}{\tilde T^\wedge}$.
Let $\map{f}{G^\wedge}{\tilde{G}^\wedge}$
be a $k$-homomorphism
whose restriction to $T^\wedge$
agrees with the map
$\map{\dnorm[T^\wedge]}{T^\wedge}{\tilde T^\wedge}$.
Suppose that $C_{\tilde G^\wedge}(f(T^\wedge)) = \tilde T^\wedge$.
Then
$f$ is a pointwise conorm for $[\phi,q^*]$.
\end{prop}

\begin{proof}
We begin by constructing two maps
from $W(G^\wedge,T^\wedge)$ to $W(\tilde G^\wedge, \tilde T^\wedge)$.
The first arises from the fact that the homomorphism $f$ maps
$N_{G^\wedge}(T^\wedge)$ into $N_{\tilde G^\wedge}(\tilde T^\wedge)$,
and we will also denote this new map by $f$.
To obtain the second map,
recall from 
\S\ref{para:weyl}
that we have a map $\map{i_{\phi,q^*}}{W(G,T)}{W(\tilde G,\tilde T)}$,
and since $\delta$ and $\tilde\delta$ give identifications
$\abmap{W(G,T)}{W(G^\wedge,T^\wedge)}$
and
$\abmap{W(\tilde G,\tilde T)}{W(\tilde G^\wedge,\tilde T^\wedge)}$,
we obtain a map, also denoted $i_{\phi,q*}$,
from $W(G^\wedge,T^\wedge)$ to $W(\tilde G^\wedge, \tilde T^\wedge)$.

We claim that the maps $f$ and $i_{\phi,q^*}$ on $W(G^\wedge, T^\wedge)$
are equal.
To see this, let $w\in W(G^\wedge, T^\wedge)$.
Since the centralizer of $f(T^\wedge)$ is $\tilde T^\wedge$,
it will be enough to show that $f(w)$ and $i_{\phi,q^*}(w)$ act in the same way
on $f(T^\wedge)$, or
equivalently, that $i_{\phi,q^*}(w)f_*(x) = f(w) f_*(x)$ for all $x\in X_*(T^\wedge)$.
To verify this,
note that
$\normchar[\phi,q^*](w\lambda) = i_{\phi,q^*}(w) \normchar[\phi,q^*](\lambda)$
for all $\lambda \in X^*(T)$, justifying the second of the following
equalities:
$$
i_{\phi,q^*}(w) f_*(x)
=
i_{\phi,q^*}(w) \dnorm[T^{\wedge,*}](x)
=
\dnorm[T^{\wedge,*}](wx)
=
f_*(wx)
=
f(w) f_*(x),
$$
proving the claim.

Suppose $T'^\wedge$ is a maximal $k$-torus in $G^\wedge$.
Choose $g^\wedge\in G^\wedge(k\sep)$ such that
$\lsup{g^\wedge}T^\wedge = T'^\wedge$.
Let $\tilde g^\wedge = f(g^\wedge)$,
and let $\tilde{T}'^\wedge = \lsup{\tilde g ^\wedge}\tilde T ^\wedge$.
Let $c \in Z^1(k,W(G^\wedge,T^\wedge))$ be the cocycle given by
$\sigma \mapsto {g^\wedge}\inv \sigma(g^\wedge)$.
Then our map $f$ on $W(G^\wedge,T^\wedge)$ takes $c$ to the cocycle
$\tilde c \in Z^1(k, W(\tilde G^\wedge, \tilde T^\wedge))$
given by 
$\sigma \mapsto {\tilde g^\wedge}{}\inv \sigma(\tilde g^\wedge)$.

Pick a maximal $k$-torus $T\subseteq G$
that corresponds to $T^\wedge\subseteq G^\wedge$ as in 
\S\ref{para:duality}.
Pick $g\in G(k\sep)$ (resp.~$\tilde g\in \tilde G(k\sep)$)
so that the cocycle in $Z^1(k, W(G,T))$ (resp.~$Z^1(k, W(\tilde G,\tilde T))$)
given by $\sigma\mapsto g\inv \sigma(g)$
(resp.~$\sigma\mapsto \tilde g\inv \sigma(\tilde g)$)
corresponds to $c$ (resp.~$i_{\phi,q^*}\circ c$) under the identification 
$\delta$ (resp.~$\tilde\delta$).
Let $T' = \lsup g T$
and $\tilde T' = \lsup{\tilde g} \tilde T$.
Then $T'\subseteq G$ and $\tilde T' \subseteq \tilde G$ are both $k$-tori.

Let $(\phi', q'^*)$ be the conjugate of $(\phi, q^*)$
via the elements $g$ and $\tilde g$
(see Remark \ref{rem:equivalence}).
Then by the Remark, $(\phi,' q'^*)$ is a parascopic datum
for $(\tilde G,\Gamma, G)$ with respect to $\tilde T'$ and $T'$,
equivalent to $(\phi, q^*)$.

Define duality maps
$\map{\delta'}{X^*(T')}{X_*(T'^\wedge)}$
and
$\map{\tilde\delta'}{X^*(\tilde T')}{X_*(\tilde T'^\wedge)}$
by
$$
\delta' = \Int (g^\wedge)_* \circ \delta \circ \Int (g)^*,
\qquad
\tilde\delta' = \Int (\tilde g^\wedge)_* \circ \tilde \delta \circ \Int (\tilde g)^*.
$$

In order to complete the proof, it is enough to show that, as maps
$\abmap{X_*(T'^\wedge)}{X_*(\tilde T'^\wedge)}$,
we have
\begin{equation}
\label{eqn:conorm-func-hom-want}
(f|_{T'^\wedge})_* = \dnormcochar{\phi',q'^*,\delta',\tilde\delta'}.
\end{equation}
The left-hand side of \eqref{eqn:conorm-func-hom-want} equals
$
\Int (\tilde g^\wedge)_* \circ (f|_{T^\wedge})_* \circ \Int (g^\wedge)_*\inv
$,
which by \eqref{eq:norm-char} and the fact that
$f|_{T^\wedge} = \dnorm[T^\wedge]$,
we may rewrite as
\begin{equation}
\label{eqn:conorm-func-hom-want-lhs}
\Int (\tilde g^\wedge)_* \circ \tilde\delta \circ
\normchar[\phi]  \circ q^* \circ \delta\inv\circ \Int (g^\wedge)_*\inv.
\end{equation}
Meanwhile, the right-hand side of \eqref{eqn:conorm-func-hom-want}
is equal to
$\tilde \delta ' \circ \normchar[\phi'] \circ q'^* \circ \delta'{} \inv$,
which we may rewrite as
$$
\Int (\tilde g^\wedge)_* \circ \tilde \delta \circ \Int (\tilde g)^*
\circ
\normchar[\phi'] \circ q'^*
\circ
\Int (g)^*{}\inv \circ \delta \inv \circ \Int (g^\wedge)_*\inv.
$$
From the equivalence of $(\phi,q^*)$ and $(\phi',q'^*)$ via $g$ and $\tilde g$,
we have that $q'^* = \lsub\Gamma \Int(\tilde g)^*{}\inv \circ q^* \circ (\Int g)^*$,
so the right-hand side of \eqref{eqn:conorm-func-hom-want} equals
\begin{equation}
\label{eqn:conorm-func-hom-want-rhs}
\Int (\tilde g^\wedge)_* \circ \tilde \delta \circ \Int (\tilde g)^*
\circ
\normchar[\phi'] \circ \lsub\Gamma \Int (\tilde g)^*{}\inv \circ  q^* 
\circ
\delta \inv \circ \Int (g^\wedge)_*\inv.
\end{equation}
Comparing \eqref{eqn:conorm-func-hom-want-lhs}
and \eqref{eqn:conorm-func-hom-want-rhs},
we see that it will be sufficient to show
$$
\normchar[\phi']
= \Int(\tilde g)^*{}\inv \circ \normchar[\phi] \circ \lsub\Gamma\Int (\tilde g)^*
$$
or, equivalently, that the right-hand face of the following prism commutes.
\begin{equation*}
\begin{xy}
\xymatrix{
&
&
&
X^*(\tilde T')
	\ar[dlll]_{\Int (\tilde g)^*}
	\ar@{->>}[dd]|!{[dl];[rr]} \hole
	\ar[rr]^{\sum\phi'(\gamma)} &
&
X^*(\tilde T')
	\ar[dlll]^(0.4){\Int(\tilde g)^*}
\\
X^*(\tilde T)
	\ar[rr]_{\sum\phi(\gamma)}
	\ar@{->>}[dd] &
&
X^*(\tilde T)
\\
&
&
&
\lsub{\phi'}X^*(\tilde T')
	\ar[uurr]_{\normchar[\phi']}
	\ar[dlll]^{\lsub\Gamma \Int (\tilde g) ^*}
\\
\lsub{\phi}X^*(\tilde T)
	\ar[uurr]_{\normchar[\phi]}
}
\end{xy}
\end{equation*}
This follows because the front and back faces commute by
\eqref{eqn:norm};
the top face commutes by the definition of $\phi'$;
and the left-hand face commutes by the definition of
$\lsub\Gamma\Int \tilde g^*$.
\end{proof}

\begin{para}[Pointwise conorms for data obtained from dual-closed subsystems]
\label{para:conorm-dual-closed}
Suppose that
$\tilde T\subseteq \tilde G$ and $T \subseteq G$ are maximal $k$-tori.
Let $q^*$
be an embedding of
$\Psi(G, T)$ in $\Psi(\tilde G, \tilde T)$ with dual-closed image.
In particular, $q^*$ and its transpose are isomorphisms of lattices.
Then, as in Example~\ref{ex:parascopy-gamma-trivial},
$(1,q^*)$ is a parascopic datum for $(\tilde G, 1, G)$
relative to the tori $\tilde T$ and $T$.

Let $\tilde G^\wedge$ and $G^\wedge$ be quasi-split $k$-groups
with maximal $k$-tori $\tilde T^\wedge\subseteq \tilde G^\wedge$ and
$T^\wedge\subseteq G^\wedge$ such that
$\tilde G^\wedge$ and $G^\wedge$ are respectively in
$k$-duality with $\tilde G$ and $G$ via duality maps
$\map{\tilde \delta}{\bX^*(\tilde T)}{\bX_*(\tilde T^\wedge)}$ and
$\map{\delta}{\bX^*(T)}{\bX_*(T^\wedge)}$.
Then $\Psi(G^\wedge, T^\wedge)$ is naturally isomorphic to a closed
subdatum of $\Psi(\tilde G^\wedge, \tilde T^\wedge)$ via the map
\begin{equation}
\label{eq:dual-embedding}
\map%
{\lsup{t}\tilde \delta\inv\circ q _*\inv\circ\lsup{t}\delta}%
{X^*(T^\wedge)}%
{X^*(\tilde T^\wedge)},
\end{equation}
where the $t$ superscript denotes transpose.

Proposition \ref{prop:embedding} implies the existence of a $k$-embedding
$\map{\nu}{G^\wedge}{\tilde G^\wedge}$.
Moreover, since $\tilde T^\wedge$ is determined only up to stable conjugacy, 
it follows from this proposition that we may choose $\tilde T^\wedge$ to be $\nu(T^\wedge)$
and that we may assume that the map
$\map{\nu^*}{X^*(\tilde T^\wedge)}{X^*(T^\wedge)}$ induced by $\nu$ coincides with the inverse of
\eqref{eq:dual-embedding}, i.e.,
\begin{equation}
\label{eq:character-map}
\nu^* = 
\lsup{t}\delta\inv\circ q_*\circ\lsup{t}\tilde \delta .
\end{equation}
Thus the transpose $\map{\nu_*}{X_*(T^\wedge)}{X_*(\tilde T^\wedge)}$ is
given by $\nu_* = \tilde \delta \circ q^* \circ \delta\inv$.
Since $q^*$ coincides with
the norm function $\map{\normchar[1,q^*]}{X^*(T)}{X^*(\tilde{T})}$,
it follows from \eqref{eq:norm-char} that $\nu_*$ agrees with $(\dnorm[1,q^*])_*$
on $X_*(T^\wedge)$. Thus $\nu$ agrees with $\dnorm[1,q^*]$
on $T^\wedge$. It follows from Proposition \ref{prop:conorm-func-hom}
that $\nu$ is a pointwise conorm for $[1,q^*]$.
\end{para}

\begin{para}[Pointwise conorms for non-root-inclusive data]
\label{para:wrong-subsystem}
Suppose $\tilde{G}$ is a connected semisimple $k$-quasisplit $k$-group
of type $A_{2n}$ and
$\gamma$ is a non-inner quasisemisimple $k$-involution on $\tilde{G}$.
Then $\gamma$ preserves some Borel-torus pair $(\tilde B, \tilde T)$ for $\tilde{G}$,
where $\tilde T$ is defined over $k$.
Let $\phi$ denote the action of $\langle\gamma\rangle$ on $\Psi(\tilde G,\tilde T)$. Then
the root datum $\lsub\phi\Psi(\tilde G,\tilde T)$ has type $BC_n$.
It is straightforward to show that $\gamma$ must preserve a pinning 
of $(\tilde G,\tilde B,\tilde T)$.

Let $\Psi = \lsub\phi \Psi(\tilde G,\tilde T)\redwrong$
(see \S\ref{sec:basic}).
That is,
$\Psi$ has type $B_n$ or $C_n$ according as $p$ is $2$ or not,
so $\Psi$ is not the root datum of
the fixed-point group $\tilde{G}^\gamma$.
Let $G$ be a $k$-group with a maximal $k$-torus
$T$ and a $\Gal(k)$-equivariant isomorphism $q^*$ from $\Psi(G,T)$ to $\Psi$.
We can and will take $G$ to be quasisplit over $k$ (but note that this is automatic
if $\tilde B$ is defined over $k$,
since $\tilde B$ then determines a $\Gal(k)$-invariant positive system in $\Psi$
and thus in $\Psi(G,T)$).
Then 
$(\phi,q^*)$ is a parascopic
datum for the triple $(\tilde{G},\Gamma, G)$, the basic
example of a datum that is not root-inclusive.

Assume for now that $\tilde B$ is defined over $k$.

Let $\tilde G^\wedge$ be the $k$-dual of $\tilde G$,
and let $\tilde T^\wedge$ be a maximal $k$-torus of $\tilde G^\wedge$ corresponding to $\tilde T$ as in \S\ref{para:duality}.
The action of $\langle\gamma\rangle$ on $\Psi(\tilde G,\tilde T)$ together with the choice of a duality map
$\abmap{X^*(\tilde T)}{X_*(\tilde T^\wedge)}$ determines an action $\hat\phi$ of
$\langle\gamma\rangle$ on $\Psi(\tilde G^\wedge, \tilde T^\wedge)$.
Moreover, $\tilde B$ determines a $\Gal(k)$- and $\gamma$-stable positive system in
$\Phi^\vee(\tilde G, \tilde T) = \Phi(\tilde G^\wedge, \tilde T^\wedge)$, hence a $\gamma$-stable
$k$-subgroup $\tilde B^\wedge$ of $\tilde G^\wedge$ containing $\tilde T^\wedge$. Choosing
a $\Gal(k)$-stable pinning of $(\tilde G^\wedge,\tilde B^\wedge,\tilde T^\wedge)$ (which always exists),
we obtain a homomorphism $\psi$ from the automorphism group of the based root datum
$\Psi(\tilde G^\wedge,\tilde B^\wedge,\tilde T^\wedge)$ to $\Aut(\tilde G^\wedge)$
(see~\cite{adler-lansky:data-actions}*{Remark 19}).
It follows that $\psi\circ\hat\phi$ is a $k$-quasicentral action of $\langle\gamma\rangle$ on $\tilde G^\wedge$.

Let $G^\wedge$ denote the $k$-group $(\tilde G^{\wedge\,\gamma})\conn$
and $T^\wedge\subseteq G^\wedge$ the maximal $k$-torus $(\tilde T^{\wedge\,\gamma})\conn$.
By Example~\ref{ex:parascopy-fixed-point},
$\Psi(G^\wedge ,T^\wedge ) = \lsub{\hat\phi}\Psi(\tilde G^\wedge,\tilde T^\wedge)\red$,
and the latter is dual to $\Psi$, which is $\Gal(k)$-isomorphic to $\Psi(G,T)$.
Hence $G^\wedge$ is $k$-dual to $G$. It is straightforward to verify
that the restriction of the embedding
from $G^\wedge$ to $\tilde G^\wedge$
agrees with the map $\map{\dnorm[\phi,q^*]}{T^\wedge}{\tilde{T}^\wedge}$.
By Proposition \ref{prop:conorm-func-hom},
the embedding $\abmap{G^\wedge}{\tilde G^\wedge}$
is a pointwise conorm for $[\phi, q^*]$.

Now drop the assumption that $\tilde B$ is defined over $k$.
Let $B_0$ be a Borel $k$-subgroup of $G$, and $T_0\subseteq B_0$
a maximal $k$-torus.
Since $\langle\gamma\rangle$ preserves a pinning of $(\tilde G,\tilde B,\tilde T)$,
by \cite{adler-lansky:data-actions}*{Lemma 14}
these choices determine a $\gamma$-invariant
Borel-torus pair $(\tilde B_0, \tilde T_0)$ in $\tilde{G}$.
Since this pair is defined over the splitting field of $\tilde T_0$
and is unique, it must be invariant under $\Gal(k)$
and thus defined over $k$.
The action of $\langle\gamma\rangle$ on $(\tilde G,\tilde B_0,\tilde T_0)$
determines a parascopic datum $(\phi_0,q_0^*)$
relative to $\tilde T_0$ and $T_0$ that is equivalent to $(\phi,q^*)$.
Therefore, as above, we obtain an embedding $\abmap{G^\wedge}{\tilde{G}^\wedge}$
that is a pointwise conorm for $[\phi_0, q_0^*] = [\phi, q^*]$.
\end{para}

\begin{para}[Action on a product]
\label{para:product}
Let $\tilde G$ be a direct product of $k$-almost simple $k$-groups
$\prod_{i=1}^r \tilde{G}_i$,
and let $\tilde T = \prod_{i=1}^r \tilde T_i$,
where $\tilde T_i\subseteq \tilde G_i$ is a maximal $k$-torus.
Let $\tilde\Psi := \Psi(\tilde G,\tilde T)$
and $\tilde\Psi_i := \Psi(\tilde G_i, \tilde T_i)$
denote the corresponding root data.
Let $\Gamma$ be a finite group acting on the
root datum $\tilde\Psi$,
commuting with the action of $\Gal(k)$, and preserving a positive
system of roots.
Then $\Gamma$ permutes the factors $\tilde\Psi_i$,
and we assume this permutation is transitive.
Let $G$ be a connected reductive $k$-group,
$T\subseteq G$ a maximal $k$-torus,
and $(\phi,q^*)$ a parascopic datum for $(\tilde G,\Gamma,G)$
relative to $\tilde T$ and $T$.
Assume that $\tilde G$ and $G$ are $k$-quasisplit.

Let $\Gamma_1 = \stab_\Gamma(\tilde\Psi_1)$,
and let $\phi_1$ be the action of $\Gamma_1$ on $\tilde\Psi_1$
determined by $\phi$.
For each $i$, choose $\gamma_i\in \Gamma$ such that
$\gamma_i(\tilde \Psi_1) = \tilde \Psi_i$.
Define an isomorphism
$\psi^*$ from $\prod_{i=1}^r \tilde \Psi_1$ to $\tilde\Psi$
by
$\psi^*(x_1, \ldots , x_r) = \sum_i \gamma_i(x_i)$
for $x_i\in X^*(\tilde T_1)$.
Since $\Gamma_1$ acts on $\tilde\Psi_1$ and the cyclic
group $\Z/r$ acts on $\prod_{i=1}^r \tilde\Psi_1$ via
cyclic permutation of coordinates, we obtain an action
of $\Gamma^\sharp := \Gamma_1 \times \Z/r$
on
$\prod_{i=1}^r \tilde\Psi_1$.
Via the isomorphism $\psi^*$ we thus obtain an action $\phi^\sharp$
of $\Gamma^\sharp$ on $\tilde\Psi$.

It is straightforward to check that the endomorphisms
$\sum_{\gamma\in\Gamma}\phi(\gamma)$ and
$\sum_{\gamma\in\Gamma^\sharp}\phi^\sharp(\gamma)$
of $X^*(\tilde T)$
agree.
Thus $(\phi^\sharp,q^*)$ is a parascopic datum for $(\tilde G,\Gamma^\sharp,G)$ by \S\ref{para:group-action}.
By Proposition \ref{prop:change-group-action},
in order to describe
$\dnorm[{[\phi,q^*]}]$ and 
$\dnormst[{[\phi,q^*]}]$
we may and will replace
$\phi$ and $\Gamma$ by $\phi^\sharp$ and $\Gamma^\sharp$.
Since $\lsub\phi\tilde\Psi= \lsub{\phi_1}\tilde\Psi_1$,
$(\phi_1,q^*)$ is a parascopic datum for $(\tilde G_1, \Gamma_1, G)$.
Note that if $\tilde G_1^\wedge$ is in $k$-duality with $\tilde G_1$
via a map $\map{\tilde\delta_1}{X^*(\tilde T_1)}{X_*(\tilde T_1^\wedge)}$,
then $\tilde G^\wedge := \prod_{i=1}^r \tilde G_1^\wedge$ is in duality
with $\tilde G$
via $\tilde\delta := (\prod \tilde\delta_1)\circ \psi^*{}\inv$.
Let $\map\diag{\tilde G_1^\wedge}{\tilde G^\wedge}$ denote the diagonal map.
\end{para}

\begin{prop}
\label{prop:product-conorm}
In the situation of \S\ref{para:product},
suppose that $\map{\dnormfunc_1}{G^\wedge}{\tilde G_1^\wedge}$
is a pointwise conorm
for 
$[\phi_1, q_1^*]$.
Then $\map{\widehat N := \diag\circ\widehat N_1}{G^\wedge}{\tilde G^\wedge}$
is a pointwise conorm for $[\phi, q^*]$.
\end{prop}
\begin{proof}
The parascopic datum $(\phi,q^*)$ is a composition of the datum $(\phi |_{\Z/r},\id)$
for $(\tilde G,\Z/r,\tilde G_1)$ and the datum $(\phi_1,q^*)$
for $(\tilde G_1, \Gamma_1, G)$.
The result now follows from
Proposition \ref{lem:pointwise-conorm-factorization},
and Example \ref{ex:permutation-action}.
\end{proof}

\section{Weil restriction and parascopic data}
\label{sec:weil-parascopic}
In this section, we let $E/k$ be a finite separable extension.
For a connected reductive $E$-group $H$,
recall that the Weil restriction ${\Weil_{E/k}H}$ of $H$ is 
a connected reductive $k$-group equipped with a surjective $E$-homomorphism
$\map{\pr}{\Weil_{E/k}H}{H}$ satisfying a certain universal property
(see~\cite{springer:lag}*{Theorem 11.4.16}),
and thus it is well defined up to a canonical isomorphism.
Whenever we say that a $k$-group $G$ is equal to $\Weil_{E/k}H$,
we will mean that $G$ is equipped with a map $\pr$ as above.
If $H$ has a $k$-structure,
then the above universal property implies that the identity map
$\abmap{H}{H}$ induces a $k$-morphism $\abmap{H}{\Weil_{E/k}H}$,
which we will denote by $\diag_{E/k}$.

\begin{rem}
\label{rem:weil-restriction-maps}
(Weil restriction of lattices.)
Let $X$ and $Y$ be lattices on which 
$\Gal(E)$ acts,
and let $\map{f}{X}{Y}$ be a $\Gal(E)$-homomorphism.
Then by applying the functor $\Ind_{\Gal(E)}^{\Gal(k)}$,
we obtain lattice representations $\Weil_{E/k}X$
and $\Weil_{E/k}Y$ 
of $\Gal(k)$
and a $\Gal(k)$-homomorphism
$\map{\Weil_{E/k} f}{\Weil_{E/k}X}{\Weil_{E/k}Y}$.
The lattice $\Weil_{E/k}X$ is equipped with a
$\Gal(E)$-embedding $\map{\iota^*}{X}{\Weil_{E/k}X}$
satisfying a well-known universal property.
For any $\Gal(E)$-invariant set $\Xi\in X$ and $\tau\in\Gal(k)$,
the image of $\iota^*(\Xi)$ under $\tau$
depends only on the class of $\tau$ in $\Gal(k)/\Gal(E) = \Hom_k(E,k\sep)$, i.e.,
only on the restriction $\bar\tau$ of $\tau$ to $E$. We will denote
this image by $\Xi_{\bar\tau}$.
\end{rem}

\begin{defn}
\label{def:weil-restriction-root-data}
Let $\Psi = (X^*,\Phi,X_*,\Phi^\vee)$ be a root datum equipped with an action 
of $\Gal(E)$. Define the \emph{Weil restriction} $\Weil_{E/k}\Psi$ of $\Psi$ to be the
quadruple $(\Weil_{E/k} X^*,\Weil_{E/k}\Phi,\Weil_{E/k}X_*,\Weil_{E/k}\Phi^\vee)$, where
$\Weil_{E/k}\Phi := \bigsqcup\Phi_\sigma$ and $\Weil_{E/k}\Phi^\vee := \bigsqcup \Phi^\vee_\sigma $,
and $\sigma$ ranges over $\Hom_k(E,k\sep)$ (see Remark~\ref{rem:weil-restriction-maps}).
Identifying $\Weil_{E/k}X_*$ with the dual of $\Weil_{E/k}X^*$, it is easily seen that $\Weil_{E/k}\Psi$ is a root datum
with an action of $\Gal(k)$
and that $\Psi_\sigma := (X^*_\sigma,\Phi_\sigma,{X_*}_\sigma,\Phi_\sigma^\vee)$ is a root datum
with an action of $\Gal(\sigma E)$.
Moreover, the implicit $\Gal(E)$-embedding $\map{\iota^*}{X^*}{\Weil_{E/k}X^*}$
(see Remark~\ref{rem:weil-restriction-maps}) is in fact a morphism
from $\Psi$ to $\Weil_{E/k}\Psi$.
\end{defn}

\begin{rem}
\label{rem:weil-restriction-groups}
(Root data of Weil restrictions of reductive groups.)
Let $H$ be a connected reductive $E$-group.
If $S\subseteq H$ is a maximal $E$-torus, then $\Weil_{E/k}S$ is a maximal $k$-torus of $\Weil_{E/k}H$
(and every maximal $k$-torus of $\Weil_{E/k}H$ arises in this way).
Moreover, $\Psi (\Weil_{E/k}H,\Weil_{E/k}S)$ is naturally isomorphic to $\Weil_{E/k}(\Psi(H,S))$,
as follows from the fact that the map $\map{\pr^*}{X^*(S)}{X^*(\Weil_{E/k}S)}$
satisfies the universal property of the functor $\Ind_{\Gal(E)}^{\Gal(k)}$.
\end{rem}

Let $\Psi$ be a root datum equipped with an action of $\Gal(E)$,
and let $\Aut_E(\Psi)$ denote the group
of $\Gal(E)$-automorphisms of $\Psi$.
Let $\Gamma$ be a finite group and $\map{\phi}{\Gamma}{\Aut_E\Psi}$ a
homomorphism. Then $\Gamma$ acts naturally on
$\Weil_{E/k}\Psi$ via $\Gal(k)$-automorphisms
(apply the functor $\Ind_{\Gal(E)}^{\Gal(k)}$ to $\phi(\gamma)$ for $\gamma\in\Gamma$),
and we denote the corresponding map
$\abmap{\Gamma}{\Aut_k(\Weil_{E/k}\Psi)}$ by $\Weil_{E/k}\phi$.

\begin{defn}
\label{defn:weil-restriction-parascopy}
Let $\tilde\Psi$ and $\Psi$
be root data with actions of $\Gal(E)$, let $\Gamma$ be a finite group, and let $(\phi , q^*)$ be
a parascopic datum for $(\tilde\Psi,\Gamma,\Psi)$. 
The morphism $q^*$ from $\Psi$ to $\lsub\phi\tilde\Psi$ determines a morphism
$\Weil_{E/k}q^*$ from $\Weil_{E/k}\Psi$ to
$\Weil_{E/k} (\lsub\phi\tilde \Psi)=\lsub{\Weil_{E/k}\phi}\Weil_{E/k}\tilde\Psi$.
We define the \emph{Weil restriction}
$\Weil_{E/k}(\phi,q^*)$ to be the pair $(\Weil_{E/k}\phi,\Weil_{E/k}q^*)$
and observe that it
is a parascopic datum for $(\Weil_{E/k}\tilde\Psi,\Gamma,\Weil_{E/k}\Psi)$.
\end{defn}

\begin{rem}
\label{rem:weil-restriction-groups-parascopy}
(Parascopic data of Weil restrictions of reductive groups.)
Let $H$ (resp.~$\tilde H$) be a connected reductive $E$-group and
$S\subseteq H$ (resp.~$\tilde S\subseteq \tilde H$) a maximal $E$-torus.
Let $(\phi,q^*)$ be a parascopic datum for $(\tilde H,\Gamma ,H)$
relative to $\tilde S$ and $S$.
Then via the identifications of Remark~\ref{rem:weil-restriction-groups},
$\Weil_{E/k}(\phi,q^*)$ can be viewed as a parascopic datum for
$(\Weil_{E/k}\tilde H,\Gamma ,\Weil_{E/k}H)$
relative to $\Weil_{E/k}\tilde S$ and $\Weil_{E/k}S$.
\end{rem}

\begin{defn}
\label{defn:induction-action}
Let $\Psi$ be a root datum equipped with an action of $\Gal(k)$, 
$\Gamma$ a finite group, and
 $\phi$ a homomorphism $\abmap{\Gamma}{\Aut_k\Psi}$.
We will say that $\phi$ is \emph{$E/k$-induced} if there exists a root datum
$\breve\Psi$ with an action of $\Gal(E)$, 
a homomorphism $\map{\breve\phi}{\Gamma}{\Aut_E\breve\Psi}$,
and a $\Gal(E)$-embedding from $\breve\Psi$ to $\Psi$
that realizes $\Psi$ as $\Weil_{E/k}\breve\Psi$ and that intertwines $\breve\phi$ and $\phi$
(so that $\phi = \Weil_{E/k}\breve\phi$).
\end{defn}

\begin{lem}
\label{lem:parascopy-restriction-over-k}
Let $\Psi$ and $\tilde\Psi$ be a root data equipped with actions of $\Gal(k)$.
Let $\Gamma$ be a finite group, and let $(\phi,q^*)$ be a
torus-inclusive parascopic datum for $(\tilde\Psi , \Gamma ,\Psi)$.
Assume that $\phi$ is $E/k$-induced. Then there exist root data
$\breve\Psi$ and $\bar\Psi$ with actions of $\Gal(E)$,
and a torus-inclusive parascopic datum $(\breve \phi, \breve q^*)$ for
$(\bar\Psi, \Gamma ,\breve\Psi)$,
such that
$\tilde\Psi = \Weil_{E/k} \bar\Psi$,
$\Psi = \Weil_{E/k} \breve\Psi$,
and
$(\phi,q^*) = \Weil_{E/k}(\breve \phi, \breve q^*)$.
\end{lem}

\begin{proof}
By assumption, there is a 
root datum $\bar\Psi$
equipped with an action of $\Gal(E)$,
a homomorphism $\map{\breve\phi}{\Gamma}{\Aut_E \bar\Psi}$
and a $\Gal(E)$-embedding $\tilde\iota^*$  
from $\bar\Psi$ to $\tilde \Psi$ realizing $\tilde\Psi$ as $\Weil_{E/k}\bar\Psi$
and such that $\phi = \Weil_{E/k}\breve\phi$.
Since $\phi$ stabilizes a system of positive roots for $\tilde\Psi$,
the same is true for the action $\breve\phi$ on $\bar\Psi$.

By assumption, $q^*$ is a $\Gal(k)$-isomorphism from
$\Psi$ to $\lsub{\phi}\tilde \Psi$. Moreover, by Remark~\ref{rem:morphism-properties},
$\tilde\iota^*$ descends to a $\Gal(E)$-morphism 
$\lsub{\Gamma}\tilde \iota^*$
from $\lsub{\breve\phi}\bar \Psi$ to $\lsub{\phi}\tilde\Psi$
which satisfies the universal property of induction,
realizing $\lsub{\phi}\tilde \Psi$ as $\Weil_{E/k}\big(\lsub{\breve\phi}\bar \Psi\big)$

Let $\breve \Psi$ be the pullback of $\lsub{\Gamma}\tilde \iota^*$ and $q^*$,
and let $\breve q^*$ (resp.~$\iota^*$) be the
$\Gal(E)$-equivariant morphism from
$\breve\Psi$ to $\lsub{\breve\phi}\bar \Psi$
(resp.~$\breve\Psi$ to $\Psi$)
arising from the universal property of $\breve\Psi$.
By Remark~\ref{rem:morphism-pullback},
$\breve\Psi$ is a root datum with an action of $\Gal(E)$.
Since $q^*$ is an isomorphism, $\breve q^*$ is as well,
so $(\breve \phi, \breve q^*)$ is a torus-inclusive parascopic datum for
$(\bar\Psi, \Gamma ,\breve\Psi)$.
Since $\lsub{\phi}\tilde \Psi = \Weil_{E/k}\big(\lsub{\breve\phi}\bar \Psi\big)$
and $q^*\circ \iota^* = \lsub{\Gamma}\tilde \iota^*\circ\breve q^*$,
it follows that $\Psi = \Weil_{E/k}\breve\Psi$
and
$q^* = \Weil_{E/k}\breve q^*$, as desired.
\end{proof}

\begin{rem}
\label{rem:weil-restriction-automorphism}
(Automorphisms and Weil restriction.)
Let $\breve\Psi = (\breve X^*,\breve\Phi,\breve X_*,\breve\Phi^\vee)$
be an irreducible root datum with $\Gal(E)$-action, $\Psi = \Weil_{E/k}\breve\Psi$,
and $\varphi\in\Aut_k (\Psi)$.
Then, in the notation of Definition~\ref{def:weil-restriction-root-data} and
viewing $\id$ as an element of $\Hom_k(E,k\sep)$,
$\varphi (\breve X^*_{\id})$ must equal $\breve X^*_\sigma$ for some $\sigma\in\Hom_k(E,k\sep)$.
Moreover, it follows from the $\Gal(k)$-equivariance of $\varphi$ that
$\sigma$ must preserve $E$, and thus we may view $\sigma$ as
an element of  $\Aut (E/k)$, the group of $k$-automorphisms of $E$.
The map
$\abmapto{\varphi}{(\varphi |_{X^*_{\id}},\sigma)}$
defines a one-to-one correspondence between
$\Aut_k (\Psi)$ and the collection of pairs
$(f,\sigma)$, where $\sigma\in\Aut (E/k)$ and $f$ is a
$\Gal(E)$-equivariant isomorphism from $\breve\Psi$ to $\breve\Psi_\sigma$.
Moreover,
$\abmapto{\varphi}{\sigma}$ defines a homomorphism $\map{\pi}{\Aut_k (\Psi)}{\Aut(E/k)}$.
\end{rem}

We now consider the situation in which a semisimple root datum $\Psi$ obtained via Weil restriction
carries an action of $\Gamma$ that is not $K/k$-induced for any nontrivial finite separable extension
$K/k$.

\begin{prop}
\label{prop:non-induced}
Suppose that $\breve\Psi$ is a irreducible root datum equipped with an action of $\Gal(E)$ and let
$\phi$ be a $\Gal(k)$-equivariant action of a finite group $\Gamma$ on $\Psi = \Weil_{E/k}\breve\Psi$.
Then $\phi$ is not $K/k$-induced for any nontrivial subextension $K/k$ of $E/k$ 
if and only if $E/k$ is Galois and $\map{\pi\circ\phi}{\Gamma}{\Aut (E/k)}$ is surjective
(see Remark~\ref{rem:weil-restriction-automorphism}). If these conditions hold, $E\neq k$, and $\Gamma$
is simple, then $\pi\circ\phi$ is an isomorphism, $\lsub\phi\Psi$ is $\Gal(E)$-isomorphic to
$\breve\Psi$, and $\phi$ determines an action of $\Gal(k)$ on $\breve\Psi$.
\end{prop}

\begin{proof}
Suppose $\phi$ is not $K/k$-induced for any field $K$ with $k\subsetneqq K\subseteq E$.
Let $\Sigma = \im (\pi\circ\phi)\subseteq \Aut (E/k)$, and let $K\subseteq E$ be the fixed field of $\Sigma$.
Then the direct sum $\bar\Psi$ over $\sigma\in \Sigma$ of the data $\breve\Psi_\sigma$
(as in Definition~\ref{def:weil-restriction-root-data}) is $\Gal(K)\times\Gamma$-invariant
and can be identified with $\Weil_{E/K}\breve\Psi$.
We have $\Psi = \Weil_{K/k}(\Weil_{E/K}\breve\Psi) = \Weil_{K/k}\bar\Psi$, and it is easily seen that 
the action $\phi$ of $\Gamma$ on $\Psi$ is induced by 
the $\Gal(K)$-equivariant action of $\Gamma$ on $\bar\Psi$
given by $\abmapto{\gamma}{\phi(\gamma)|_{\bar\Psi}}$.
Thus $\phi$ is $K/k$-induced, which, by assumption, implies that $K = k$, so $\Sigma = \Aut(E/k)$ and $\pi\circ\phi$ is surjective.
Moreover, since $k$ is the fixed field of a subgroup of $\Aut(E/k)$, $E/k$ is Galois.

Conversely, assume that $E/k$ is Galois and $\phi$ is $K/k$-induced from an action
$\bar\phi$ on a root datum $\bar\Psi$ with $\Gal(K)$-action,
where $k\subsetneqq K\subseteq E$.
Thus there is a natural $\Gal(K)$-equivariant embedding $\bar\iota^*$
of $\bar\Psi$ in $\Psi$ that realizes $\Psi$ as $\Weil_{K/k}\bar\Psi$ and intertwines $\bar\phi$ and $\phi$.
Using the notation of Definition~\ref{def:weil-restriction-root-data},
$\im\bar\iota^*$ is $\phi(\Gamma)$-invariant and is of the form
$\bigoplus_\Sigma \breve\Psi_\sigma$ for some proper subset $\Sigma$ of $\Hom_k(E,k\sep) =\Gal(E/k)$.
For $\gamma\in\Gamma$ and $\tau\in\Gal(k)$ with image $\bar\tau$ in $\Gal(E/k)$,
\begin{equation}
\label{eq:gamma-action}
\phi(\gamma) (\Psi_{\bar\tau}) = \phi(\gamma) (\tau(\Psi_{\id})) = \tau(\phi(\gamma)(\Psi_{\id})) =
\tau(\Psi_{(\pi\circ\phi)(\gamma)}) = \Psi_{\bar\tau \cdot (\pi\circ\phi)(\gamma)} .
\end{equation}
Since $\im\bar\iota^*$ is $\phi(\Gamma)$-invariant,
it follows that $\Sigma$ is a union of left $(\pi\circ\phi)(\Gamma)$-cosets.
Since $\Sigma$ is a proper subset of $\Gal(E/k)$, it follows that $(\pi\circ\phi)(\Gamma)\neq\Gal(E/k)$,
so $\pi\circ\phi$ is not surjective. 

We now prove the last statement. Let $\iota^*$ denote the implicit $\Gal(E)$-embedding
of $\breve\Psi$ in $\Psi = \Weil_{E/k}\breve\Psi$,
and let $X^*$ denote the character lattice component of $\Psi$.
If $\Gamma$ is simple, then $\ker(\pi\circ\phi)$ must be trivial,
since $\ker(\pi\circ\phi) = \Gamma$ implies $E=k$. Thus $\pi\circ\phi$ is an isomorphism,
and \eqref{eq:gamma-action} shows that the action of $\Gamma$ on the conjugates of
$\iota^*(\breve\Psi)$ is simply transitive. It follows readily from this that 
the composition of the quotient map $\abmap{X^*}{\lsub{\phi}X^*}$ with $\iota^*$
gives a $\Gal(E)$-isomorphism of $\breve\Psi$ with $\lsub\phi \Psi$. Since
$\lsub\phi \Psi$ carries an action of $\Gal(k)$, this isomorphism extends the
action of $\Gal(E)$ on $\breve\Psi$ to $\Gal(k)$.
\end{proof}

\begin{notation}
\label{notation:weil-restriction}
For the remainder of this section, we let $\tilde H$ denote a connected, reductive $E$-group,
and let $\tilde{G} = \Weil_{E/k} \tilde H$.
Let $G$ denote a connected reductive $k$-group, and let
$(\phi,q^*)$ be a torus-inclusive parascopic datum for $(\tilde G,\Gamma, G)$ relative to the maximal
$k$-tori $\tilde T \subseteq\tilde G$ and $T\subseteq G$.
Then $\tilde T = \Weil_{E/k}\tilde S$
for some maximal $E$-torus $\tilde S\subseteq\tilde H$.
\end{notation}

\begin{examples}
\label{ex:weil-restriction-parascopy}
Here are two special cases of the above situation.

\begin{enumerate}[(a)]
\item
\label{ex:weil-restriction-datum}
Let $G = \Weil_{E/k}H$ for some connected reductive $E$-group $H$.
Then $T = \Weil_{E/k}S$ for some maximal $E$-torus $S$ of $H$.
Let $(\breve \phi, \breve q^*)$ be a parascopic datum for $(\tilde H, \Gamma, H)$
relative to $\tilde S$ and $S$.
Then $(\phi,q^*) = \Weil_{E/k}(\breve \phi, \breve q^*)$ is a parascopic datum for
$(\tilde G,\Gamma, G)$ relative to $\tilde T$ and $T$, as in Remark~\ref{rem:weil-restriction-groups-parascopy}.
\item
\label{ex:base-change-parascopy}
Suppose that $E/k$ is a Galois extension and that $\tilde H$
is equipped with a $k$-structure.
This structure canonically determines a homomorphism $\abmap{\Gal(E/k)}{\Aut_k(\tilde G)}$.
Letting $\Gamma$ denote its image, we let
$G:=\tilde G^\Gamma = \diag_{E/k}(\tilde H)$.
Let $\tilde T = \Weil_{E/k}\tilde S$ be a
$\Gamma$-invariant maximal $k$-torus of $\tilde G$.
Then $T: = \tilde T^\Gamma$ is a maximal
$k$-torus of $G$ that is $E$-isomorphic to $\tilde S$,
and the action of $\Gamma$ on $\Psi(\tilde G,\tilde T)$ gives rise to a
parascopic datum for $(\tilde G, \Gamma, G)$ relative to $\tilde T$ and $T$
as in Example~\ref{ex:parascopy-fixed-point}.
\end{enumerate}
\end{examples}

The following lemma shows that the parascopic data for $(\tilde G,\Gamma, G)$
for which the action of $\Gamma$ is $E/k$-induced must be of the form described in
Example~\ref{ex:weil-restriction-parascopy}\eqref{ex:weil-restriction-datum}.

\begin{prop}
\label{lem:parascopic-group-restriction-over-k}
Using Notation \ref{notation:weil-restriction},
assume that the action $\phi$ of $\Gamma$ on
$\Psi(\tilde G,\tilde T) = \Weil_{E/k}(\Psi(\tilde H,\tilde S))$
is $E/k$-induced from an action $\breve\phi$ of $\Gamma$ on
$\Psi(\tilde H,\tilde S)$.
Then the parascopic datum $(\phi,q^*)$ determines a connected reductive
$E$-group $H$ and a maximal $E$-torus $S$ of $H$ such that
$G = \Weil_{E/k}H$, $T = \Weil_{E/k}S$, and
$(\phi,q^*) = \Weil_{E/k}(\breve \phi, \breve q^*)$ for some
parascopic datum $(\breve\phi, \breve q^*)$
for $(\tilde H, \Gamma, H)$ relative to $\tilde S$ and $S$.
\end{prop}

\begin{proof}
By Lemma~\ref{lem:parascopy-restriction-over-k},
there exists a root datum $\breve\Psi$ with an action of $\Gal(E)$,
and a torus-inclusive parascopic datum $(\breve \phi, \breve q^*)$ for
$(\Psi(\tilde H,\tilde S), \Gamma ,\breve\Psi)$
such that $(\phi,q^*) = \Weil_{E/k}(\breve \phi, \breve q^*)$.
In particular, there is a
$\Gal(E)$-embedding $\iota^*$ from $\breve\Psi$ to $\Psi(G,T)$
that realizes $\Psi(G,T)$ as $\Weil_{E/k} \breve\Psi$.

It is easily seen that $\im\iota^*$ determines a reductive $E$-subgroup $H$ of $G$ and a maximal
$E$-torus $S$ of $H$ such that $G = \Weil_{E/k}H$ and $T = \Weil_{E/k}S$. Then
$\breve\Psi$ can be identified with $\Psi(H,S)$ and hence $(\breve\phi,\breve q^*)$
is a parascopic datum for $(\tilde H,\Gamma,H)$ relative to $\tilde S$ and $S$.
\end{proof}

The following result shows that Weil restriction of parascopic data descends to an operation on equivalence classes
of data.

\begin{prop}
\label{prop:weil-restriction-classes}
Let $\tilde H$ and $H$ be connected reductive $E$-groups, $\Gamma$ a finite group,
and $(\breve\phi,\breve q^*)$ a parascopic datum for $(\tilde H,\Gamma,H)$.
Let $\tilde G = \Weil_{E/k}\tilde H$ and $G = \Weil_{E/k}H$, and let $(\phi,q^*)$ denote the parascopic datum
$\Weil_{E/k}(\breve\phi,\breve q^*)$ for $(\tilde G,\Gamma,G)$ (see Remark~\ref{rem:weil-restriction-groups-parascopy}).
If $(\breve\phi',\breve q^{\prime *})$ is a parascopic datum for $(\tilde H,\Gamma,H)$
equivalent to $(\breve\phi,\breve q^*)$, then $\Weil_{E/k}(\breve\phi',\breve q^{\prime *})$
is equivalent to $(\phi,q^*)$. Moreover, every parascopic datum for $(\tilde G,\Gamma,G)$ equivalent
to $(\phi,q^*)$ arises in this way.
\end{prop}

\begin{proof}
Suppose $(\breve\phi,\breve q^*)$ is a parascopic datum relative to the maximal $E$-tori
$\tilde S\subseteq\tilde H$ and $S\subseteq H$, and that $(\breve\phi,\breve q^*)$ is equivalent
to $(\breve\phi',\breve q^{\prime *})$ via $h\in H(k\sep)$ and $\tilde{h}\in \tilde{H}(k\sep)$.
Let $\tilde S' = \lsup{\tilde h}\tilde S$ and $S' = \lsup h S$.
Then $T := \Weil_{E/k}S$ and $T' := \Weil_{E/k}S'$ (resp.~$\tilde T := \Weil_{E/k}\tilde S$ and $\tilde T' := \Weil_{E/k}\tilde S'$)
are maximal $k$-tori of $G$ (resp.~$\tilde G$).
Choose $g\in G(k\sep)$ and $\tilde g\in \tilde G(k\sep)$ such that
$T' = \lsup g T$ and $\tilde T' = \lsup{\tilde g}\tilde T$.
It is easily seen that these choices can be made in such a way that
$\pr(g) = h$ and $\pr(\tilde g) = \tilde h$.

We claim that $\tilde g$ and $g$ implement an equivalence between the parascopic data
$(\phi,q^*)$ and $(\phi',{q'}^*) := \Weil_{E/k}(\breve\phi',\breve q^{\prime *})$. To prove this, we first show that
$\phi'(\gamma) = {\Int(\tilde g)^*}\inv \circ \phi(\gamma) \circ \Int(\tilde{g})^*$
for all $\gamma\in \Gamma$ (see Definition~\ref{defn:equivalent-data}(b)).
By the universal property of $\phi' = \Ind_{\Gal(E)}^{\Gal(k)}\breve\phi '$,
this is equivalent to showing that
\begin{equation}
\label{eq:weil-restriction-phi}
{\Int(\tilde g)^*}\inv \circ \phi(\gamma) \circ \Int(\tilde{g})^* \circ \pr^* = \pr^* \circ\, \breve\phi'(\gamma) .
\end{equation}

The equation $\phi = \Weil_{E/k}\breve\phi$ is equivalent to the property that
$\phi(\gamma)\circ\pr^* = \pr^*\circ\,\breve\phi(\gamma)$
for all $\gamma\in\Gamma$. Thus we have 
\begin{equation}
\label{eq:phi-equivalence}
\begin{aligned}
&    {\Int(\tilde g)^*}\inv \circ \phi(\gamma) \circ \Int(\tilde{g})^* \circ \pr^*
&&=  {\Int(\tilde g)^*}\inv \circ \phi(\gamma) \circ (\pr \circ \Int(\tilde{g}))^* \\
&=   {\Int(\tilde g)^*}\inv \circ \phi(\gamma) \circ (\Int(\tilde{h}) \circ \pr)^*
&&=  {\Int(\tilde g)^*}\inv \circ \phi(\gamma) \circ \pr^* \circ \Int(\tilde{h})^* \\
&=   {\Int(\tilde g)^*}\inv \circ \pr^* \circ \,\breve\phi(\gamma) \circ \Int(\tilde{h})^* 
&&=  (\pr \circ \Int(\tilde g)\inv)^* \circ  \breve\phi(\gamma) \circ \Int(\tilde{h})^*\\
&=   (\Int(\tilde h)\inv \circ \pr )^* \circ  \breve\phi(\gamma) \circ \Int(\tilde{h})^* 
&&=  \pr^* \circ \,{\Int(\tilde h)^*}\inv  \circ  \breve\phi(\gamma) \circ \Int(\tilde{h})^* \\
&=   \pr^* \circ \,\breve\phi'(\gamma) , 
\end{aligned}
\end{equation}
proving \eqref{eq:weil-restriction-phi}. The proof that
${q'}^* = \lsub{\Gamma}{\Int(\tilde g)^*}\inv \circ q^* \circ \Int(g)^*$
(see Definition~\ref{defn:equivalent-data}(c)) is analogous.
Thus $(\phi,q^*)$ and $(\phi',{q'}^*)$ are equivalent via the elements $\tilde g$ and $g$.

To prove the last statement of the proposition, suppose now that
$(\phi',{q'}^*)$ is an arbitrary parascopic datum for $(\tilde G,\Gamma , G)$
relative to maximal $k$-tori $\tilde T'\subseteq\tilde G$ and $T'\subseteq G$.
Moreover, suppose that $(\phi',{q'}^*)$ is equivalent to $(\phi,q^*)$ via elements
$g\in G(k\sep)$ and $\tilde g\in \tilde G(k\sep)$. Let 
$\tilde S' =\pr(\tilde T')$, $S' =\pr(T')$, $h = \pr(g)$ and $\tilde h = \pr(\tilde g)$.
Then $S'\subseteq H$ and $\tilde S'\subseteq\tilde H$ are maximal $E$-tori,
$S' = \lsup h S$, and $\tilde S' = \lsup{\tilde h} \tilde S$. 

Since $(\phi',{q'}^*)$ is equivalent to $(\phi,q^*)$,
the images of $g\inv\sigma(g)$ in $W(G,T)$ and of $\tilde g\inv\sigma(\tilde g)$
in $W(\tilde G,\tilde T)$ correspond as in Remark~\ref{rem:equivalence} for all $\sigma\in\Gal(k)$.
Applying $\pr$, it follows that the images of $h\inv\sigma(h)$ in $W(H,S)$ and of $\tilde h\inv\sigma(\tilde h)$
in $W(\tilde H,\tilde S)$ correspond for all $\sigma\in\Gal(E)$.
Letting $(\breve\phi',\breve q^{\prime *})$ be the conjugate of $(\breve\phi,\breve q^*)$ via $h$ and $\tilde h$,
it follows again from Remark~\ref{rem:equivalence} that $(\breve\phi',\breve q^{\prime *})$ is a parascopic datum
for $(\tilde H,\Gamma,H)$ relative to $\tilde S'$ and $S'$ that is equivalent to $(\breve\phi,\breve q^*)$.
The same calculation as that in \eqref{eq:phi-equivalence} now shows that
$(\phi',{q'}^*) = \Weil_{E/k}(\breve\phi',\breve q^{\prime *})$, completing the proof.
\end{proof}

\section{Weil restriction and pointwise conorm functions}
\label{sec:weil-conorm}
In this section, we analyze the pointwise conorm functions that arise from parascopic data for
$(\tilde G,\Gamma, G)$ in the case in which the $k$-group $\tilde G$ is the Weil restriction of an $E$-group $\tilde H$.
We continue to use the terminology in Notation~\ref{notation:weil-restriction}, and
we assume further for the remainder of the section that both $\tilde G$ and $G$ are $k$-quasisplit.
It follows immediately that $\tilde H$ is $E$-quasisplit.

\subsection{$E/k$-induced action of $\Gamma$}
\label{sub:conorm-restricted-datum}
We first analyze the pointwise conorm functions that arise from parascopic data $(\phi,q^*)$ for
$(\tilde G,\Gamma, G)$ for
which the action of $\Gamma$ is $E/k$-induced.
By Lemma \ref{lem:parascopic-group-restriction-over-k}, this is precisely the setting of
Example~\ref{ex:weil-restriction-parascopy}\eqref{ex:weil-restriction-datum}, and so there
exist a connected reductive quasisplit $E$-group $H$ such that $G = \Weil_{E/k}H$,
a maximal $E$-torus $S\subseteq H$ such that $T = \Weil_{E/k}S$,
and a parascopic datum $(\breve \phi, \breve q^*)$
for $(\tilde H, \Gamma, H)$ relative to $\tilde S$ and $S$ such that
$(\phi,q^*) = \Weil_{E/k}(\breve \phi, \breve q^*)$.

In this setting, let $H^\wedge$ and $\tilde H^\wedge$ be connected reductive quasisplit $E$-groups
in $E$-duality respectively with $H$ and $\tilde H$.
Choose maximal $E$-tori $S^\wedge\subseteq H^\wedge$ and
$\tilde S^\wedge\subseteq\tilde H^\wedge$ in respective duality with $S$ and $\tilde S$ via
$E$-duality maps $\map{\delta}{\bX^*(S)}{\bX_*(S^\wedge)}$
and $\map{\tilde\delta}{\bX^*(\tilde S)}{\bX_*(\tilde S^\wedge)}$, as in \S\ref{para:conorm-defn}.
Let $G^\wedge = \Weil_{E/k}H^\wedge$ and $\tilde G^\wedge = \Weil_{E/k}\tilde H^\wedge$.
Then $T^\wedge :=\Weil_{E/k}S^\wedge$ (resp.\ $\tilde T^\wedge :=\Weil_{E/k}\tilde S^\wedge$)
is a maximal $k$-torus of $G^\wedge$ (resp.\ $\tilde G^\wedge$).
Moreover, under the identifications in Remark~\ref{rem:weil-restriction-groups},
$T$ (resp.\ $\tilde T$) is in $k$-duality with
$T^\wedge$ (resp.\ $\tilde T^\wedge$) via the map
$\map{\Weil_{E/k}\delta}{\bX^*(T)}{\bX_*(T^\wedge)}$
(resp.\ $\map{\Weil_{E/k}\tilde\delta}{\bX^*(\tilde T)}{\bX_*(\tilde T^\wedge)}$).
Thus $G^\wedge$ and $\tilde G^\wedge$ are
in $k$-duality respectively with $G$ and $\tilde G$.

\begin{lem}
\label{lem:conorm-restriction-over-k}
Under the preceding assumptions and using the above notation (including~\ref{notation:weil-restriction}) 
the following hold.
\begin{enumerate}[(a)]
\item
\label{item:unrestriction-conorm-torus}
 $\dnorm[\phi, q^*,\Weil_{E/k}\tilde\delta,\Weil_{E/k}\delta]
 = \Weil_{E/k}\dnorm[\breve\phi, \breve q^*,\tilde\delta,\delta] $.
\item
\label{item:unrestriction-conorm}
Under the natural identifications of $G^\wedge(k)$ with $H^\wedge(E)$
and $\tilde G^\wedge(k)$ with $\tilde H^\wedge(E)$,
the conorm maps $\dnormst[{[\phi,q^*]}]$ and $\dnormst[{[\breve\phi,\breve q^*]}]$
coincide.
\end{enumerate}
\end{lem}

\begin{rem}
Shapiro's lemma implies that the natural identification of $G^\wedge(k)$ with $H^\wedge(E)$
(resp.~$\tilde G^\wedge(k)$ with $\tilde H^\wedge(E)$) induces one of 
$\Cl\st(G)$ with $\Cl\st(H)$ (resp.~$\Cl\st(\tilde G)$ with $\Cl\st(\tilde H)$).
Thus the domains and codomains of $\dnormst[{[\phi,q^*]}]$ and $\dnormst[{[\breve\phi,\breve q^*]}]$
can be respectively identified.
\end{rem}

\begin{proof}
Proving \eqref{item:unrestriction-conorm-torus}
is clearly equivalent to showing that
the induced maps $(\dnorm[\phi, q^*,\Weil_{E/k}\tilde\delta,\Weil_{E/k}\delta])_*$ and
$\Weil_{E/k}(\dnorm[\breve\phi, \breve q^*,\tilde\delta,\delta])_* $ on $X_*(T^\wedge)$ agree.
This follows from the computation
\begin{align*}
(\dnorm[\phi, q^*,\Weil_{E/k}\tilde\delta,\Weil_{E/k}\delta])_*
&= \Weil_{E/k}\tilde\delta\circ\ \norm[\phi,q^*]^* \circ \Weil_{E/k}\delta\inv \\
&= \Weil_{E/k}\tilde\delta\circ\  \Weil_{E/k}\norm[\breve\phi, \breve q^*]^* \circ \Weil_{E/k}\delta\inv \\
&= \Weil_{E/k} (\tilde\delta\circ \norm[\breve\phi, \breve q^*]^* \circ \delta\inv) \\
&= \Weil_{E/k}(\dnorm[\breve\phi, \breve q^*,\tilde\delta,\delta])_* \ ,
\end{align*}
where the first and last equalities follow from \eqref{eq:norm-char} and the second holds
since $(\phi,q^*) = \Weil_{E/k}(\breve \phi, \breve q^*)$.

To prove \eqref{item:unrestriction-conorm}, given a semisimple element $t\in G^\wedge(k)$,
choose a maximal $k$-torus ${T'}^\wedge$ of $G^\wedge$ with $t\in {T'}^\wedge(k)$.
Let $T'$ be a maximal $k$-torus of $G$ that is in $k$-duality with ${T'}^\wedge$.
By~\cite{adler-lansky:lifting}*{Proposition 7.5(i)}, there exists a maximal $k$-torus
$\tilde T'$ and a parascopic datum $(\phi',{q^*}')$ relative to $\tilde T'$ and $T'$ that is equivalent to $(\phi,q^*)$.
By Proposition~\ref{prop:weil-restriction-classes}, $(\phi',{q^*}') = \Weil_{E/k}(\breve\phi',\breve q^{\prime *})$
for some parascopic datum $(\breve\phi',\breve q^{\prime *})$ for $(\tilde H,\Gamma, H)$
equivalent to $(\breve \phi, \breve q^*)$. Since $\dnormst[{[\phi,q^*]}] = \dnormst[{[\phi',{q^*}']}]$
and $\dnormst[{[\breve\phi,\breve q^*]}] = \dnormst[{[\breve\phi',\breve q^{\prime *}]}]$,
we may therefore assume that $T'=T$, ${T'}^\wedge = T^\wedge$, and $\tilde T' = \tilde T$.

Since $\dnorm[\phi, q^*,\Weil_{E/k}\tilde\delta,\Weil_{E/k}\delta]
 = \Weil_{E/k}\dnorm[\breve\phi, \breve q^*,\tilde\delta,\delta] $, it follows that these homomorphisms 
induce the same map of points on $T^\wedge(k) = S^\wedge(E)$.
Since  $\dnormst[{[\phi,q^*]}]$ (resp.~$\dnormst[{[\breve\phi,\breve q^*]}]$)
is defined to agree with $\dnorm[\phi,q^*,R_{E/k}\tilde\delta,R_{E/k}\delta]$
(resp.~$\dnorm[\breve\phi,\breve q^*,\tilde\delta,\delta]$) on $T^\wedge(k)$ (resp~$S^\wedge(E)$), it follows that
(under the above identifications)
the images of $[t]\st$ under
$\dnormst[{[\phi,q^*]}]$ and $\dnormst[{[\breve\phi,\breve q^*]}]$
coincide, proving (b).
\end{proof}

\begin{prop}
\label{prop:conorm-weil-restriction-general}
In the notation of Lemma~\ref{lem:conorm-restriction-over-k},
if $\map{\dnormfunc}{H^\wedge}{\tilde{H}^\wedge}$
is a pointwise conorm for the equivalence class $[\breve\phi,\breve j_*]$
of the parascopic datum $(\breve\phi,\breve j_*)$
for $(\tilde H,\Gamma, H)$,
then $\map{\Weil_{E/k}\dnormfunc}{G^\wedge}{\tilde{G}^\wedge}$ is a pointwise conorm
for the class $[\phi, j_*]$
for $(\tilde G,\Gamma, G)$.
\end{prop}
\begin{proof}
Since $\dnormfunc$ is defined over $E$, $\Weil_{E/k}\dnormfunc$ is defined over $k$.
Moreover, over $k\sep$, $G^\wedge = \Weil_{E/k}H^\wedge$
(resp.~$\tilde G^\wedge = \Weil_{E/k}\tilde H^\wedge$) can be identified with
a product of $[E:k]$ $k\sep$-groups, each $k\sep$-isomorphic to $H^\wedge$ (resp.~$\tilde H^\wedge$).
Under the identifications of $G^\wedge$ and $\tilde G^\wedge$ with these product groups, it is easily
seen that $\Weil_{E/k}\dnormfunc$ coincides with the map $\dnormfunc\times\cdots\times\dnormfunc$. Since
$\dnormfunc$ descends to a map $\abmap{\Cl(H^\wedge)}{\Cl(\tilde H^\wedge)}$ on geometric conjugacy classes,
it follows that $\Weil_{E/k}\dnormfunc$ descends to a map $\abmap{\Cl(G^\wedge)}{\Cl(\tilde G^\wedge)}$,
hence satisfies Definition~\ref{defn:conorm-function}(1).

We now show that $\Weil_{E/k}\dnormfunc$ satisfies Definition~\ref{defn:conorm-function}(2).
Suppose that ${T'}^\wedge$ is a maximal $k$-torus of $G^\wedge$.
We need to show that $\Weil_{E/k}\dnormfunc$ restricts to a conorm on ${T'}^\wedge$
corresponding to some parascopic datum for $(\tilde G,\Gamma, G)$ equivalent to $(\phi,q^*)$.

There exists a maximal $E$-torus ${S'}^\wedge\subseteq H^\wedge$ such that ${T'}^\wedge = \Weil_{E/k}{S'}^\wedge$.
Since $\dnormfunc$ is a pointwise conorm function, there exist
maximal $k$-tori $S'\subseteq G$, $\tilde S'\subseteq\tilde G$,
and $\tilde S'{}^\wedge\subseteq\tilde G^\wedge$, a parascopic datum $(\breve\phi',\breve q^{\prime *})$
relative to $\tilde S'$ and $S'$ equivalent to $(\breve\phi,\breve q^*)$, and duality maps 
$\map{\delta'}{\bX^*(S')}{\bX_*({S'}^\wedge)}$
and
$\map{\tilde\delta'}{\bX^*(\tilde{S'})}{\bX_*(\tilde S'{}^\wedge)}$
such that the restriction of $\dnormfunc$ to $S'{}^\wedge$
coincides with the conorm $\map{\dnorm[\breve\phi',\breve q^{\prime *},\tilde\delta',\delta']}{S'{}^\wedge}{\tilde S'{}^\wedge}$.
Moreover, by Proposition~\ref{prop:weil-restriction-classes},
$(\phi',{q'}^*): = \Weil_{E/k}(\breve\phi',\breve q^{\prime *})$ is a parascopic datum for $(\tilde G,\Gamma, G)$ relative
to the maximal $k$-tori $\tilde T':=\Weil_{E/k}\tilde S'$ and $T':=\Weil_{E/k}S'$ that is equivalent to $(\phi,q^*)$.
Thus $S'{}^\wedge$, $T'{}^\wedge$, $(\breve\phi',\breve q^{\prime *})$, $(\phi',{q'}^*)$, etc.~satisfy
precisely the same conditions as ${S}^\wedge$, ${T}^\wedge$, $(\breve\phi,\breve q^*)$, $(\phi,q^*)$, etc.
Without loss of generality, therefore, it suffices to consider the particular maximal $k$-torus
$T^\wedge = \Weil_{E/k}S^\wedge$ of $G^\wedge$.

As in the preceding paragraph, we may assume that
the restriction of $\dnormfunc$ to ${S}^\wedge$
coincides with the conorm
$\map{\dnorm[\breve\phi,\breve q^*,\tilde\delta,\delta]}{S^\wedge}{\tilde{S}^\wedge}$.
Thus the restriction of $\Weil_{E/k}\dnormfunc$ to $T^\wedge$ coincides with
$\Weil_{E/k}\dnorm[\breve\phi, \breve q^*,\tilde\delta,\delta]$, which
equals $\dnorm[\phi,q^*,\Weil_{E/k}\tilde\delta,\Weil_{E/k}\delta]$
by Lemma~\ref{lem:conorm-restriction-over-k}(\ref{item:unrestriction-conorm}).
Thus, Definition \ref{defn:conorm-function}(2)
holds for $T^\wedge$.
\end{proof}

\subsection{Non-induced action of $\Gamma$}
\label{sub:act-via-group-aut}
We now consider a parascopic datum $(\phi,q^*)$ for $(\tilde G,\Gamma, G)$ 
for which the action of $\Gamma$ is not $K/k$-induced for any subextension
$K/k$ of $E/k$.
Since $\Gamma$ does not preserve
any irreducible factor of $\Psi(\tilde G,\tilde T)$,
we have that $(\phi,q^*)$ is root-inclusive.
We may of course assume that $E\neq k$.
Suppose also that $\tilde H$ is absolutely simple.
According to Proposition~\ref{prop:non-induced}, it follows that $E/k$ must be Galois and that
$\Psi(\tilde H,\tilde S)$ is $\Gal(E)$-isomorphic to
$\lsub{\phi}\Psi(\tilde G,\tilde T)\red$.

Apply Proposition~\ref{prop:decomp-fixed-pinning} to $(\phi, q^*)$ to obtain an action of $\Gamma$ on $\tilde G$ via
$k$-automorphisms that induces the action $\phi$ on $\Psi(\tilde G,\tilde T)$. (Replacing $(\phi,q^*)$ by
an equivalent parascopic datum, we may assume that
the torus $\tilde T'$ appearing in~\loccit is $\tilde T$.)
Letting $\bar G = (\tilde G^\Gamma)\conn$ and $\bar T = (\tilde G^\Gamma)\conn$,
we have that $(\phi, q^*)$ is a composition of the parascopic datum
$(\phi,\id )$ for $(\tilde G, \Gamma, \bar G)$ relative to $\tilde T$ and $\bar T$,
and the parascopic datum $(1, q^*)$ for $(\bar G, 1, G)$ relative to $\bar T$ and $T$.
(Here we identify $\Psi(\bar G,\bar T)$ and
$\lsub\phi\Psi(\tilde G,\tilde T) = \lsub\phi\Psi(\tilde G,\tilde T)\red$
as in Example~\ref{ex:parascopy-fixed-point}.)

Recall that the $\Gal(E)$-isomorphism from $\Psi(\tilde H,\tilde S)$ to
$\lsub{\phi}\Psi(\tilde G,\tilde T) = \Psi (\bar G,\bar T)$ from 
Proposition~\ref{prop:non-induced} is the composition of the quotient map
$\abmap{X^*(\tilde T)}{X^*(\bar T)}$ with the implicit embedding
$\abmap{X^*(\tilde S)}{X^*(\tilde T) = \Weil_{E/k}X^*(\tilde S)}$.
It is easily seen that this map on character lattices is induced by the $E$-isomorphism
$\abmap{\bar G}{\tilde H}$ obtained by composing the projection
$\map{\pr}{\tilde G}{\tilde H}$ with the inclusion $\abmap{\bar G}{\tilde G}$.
Composing the inverse of the latter isomorphism with $\pr$ gives an $E$-homomorphism from $\tilde G$ to $\bar G$
that satisfies the universal property of Weil restriction, and therefore
we can and will realize $\tilde G$ as $\Weil_{E/k}\bar G$.
Thus this situation is precisely that considered in
Example~\ref{ex:weil-restriction-parascopy}(\ref{ex:base-change-parascopy}).

Assuming that there exists a pointwise conorm function for $[1,q^*]$,
it follows from Lemma~\ref{lem:pointwise-conorm-factorization}
that the problem of finding such a function for $[\phi,q^*]$ reduces to that of finding one
for $[\phi,\id]$. We focus for now on the latter problem, i.e., we assume for now that $G = \bar G$
and $T=\bar T$.

Let $G^\wedge$ be a quasi-split $k$-group in $k$-duality with $G$.
Let $T^\wedge$ be a maximal $k$-torus of $G^\wedge$ corresponding to $T$ as in
\S\ref{para:duality}, and let $\map{\delta}{X^*(T)}{X_*(T^\wedge)}$ be a duality map.
Let $\tilde G^\wedge = \Weil_{E/k} G^\wedge$ and $\tilde T^\wedge =  \Weil_{E/k} T^\wedge$.
Then as in \S\ref{sub:conorm-restricted-datum}, $\tilde G^\wedge$ is in duality with $\tilde G$ via the duality map
$\tilde\delta := \map{\Weil_{E/k}\delta}{X^*(\tilde T)}{X_*(\tilde T^\wedge)}$.

\begin{prop}
\label{prop:conorm-base-change}
Using the notation and assumptions of \S\ref{sub:act-via-group-aut}, we have that
$\map{\diag_{E/k}}{G^\wedge}{\tilde G^\wedge}$
is a pointwise conorm for $[\phi,\id]$.
\end{prop}
\begin{proof}
Since $\diag_{E/k}$ is a homomorphism, from Proposition \ref{prop:conorm-func-hom},
it will be enough to show that
the restriction of $\diag_{E/k}$
to $T^\wedge$
agrees with the map
$\map{\dnorm[\phi,\id,\tilde\delta,\delta]}{T^\wedge}{\tilde{T}^\wedge}$.
Let $\hat\pr$ be the natural $E$-projection $\abmap{\tilde T^\wedge =  \Weil_{E/k} T^\wedge}{T^\wedge}$.
Then by definition, $\dnorm[\phi,\id,\tilde\delta,\delta]$ will equal $\diag_{E/k}$
if and only if $\hat\pr\circ \dnorm[\phi,\id,\tilde\delta,\delta] = \id_{T^\wedge}$.
This is equivalent to the equality
$\hat\pr_*\circ (\dnorm[\phi,\id,\tilde\delta,\delta])_* = \id_{X_*(T^\wedge)}$
of the corresponding maps on cocharacters.
Dualizing via $\delta$ and $\tilde\delta$ as in \eqref{eq:norm-char},
the latter condition becomes
\begin{equation}
\label{eq:diag-criterion}
p^* \circ \normchar[\phi] = \id_{X^*(T)},
\end{equation}
where $p^*$ is the natural projection $\abmap{X^*(\tilde T) = \Weil_{E/k}X^*(T)}{X^*(T)}$.
More precisely, if $\iota^*$ denotes the implicit $\Gal(E)$-embedding $\abmap{X^*(T)}{X^*(\tilde T)}$,
then $p^* \circ \iota^* = \id $ and $p^*$ is trivial on the conjugates
$\sigma (\iota^*(X^*(T))$ for $\sigma\in\Gal(E/k)\smallsetminus\{\id\}$.
Then \eqref{eq:diag-criterion} follows from the fact that $\normchar[\phi]$ is induced by the map
$\sum_{\gamma\in\Gamma}\, \phi(\gamma)$ on $X^*(\tilde T)$ and that 
the action $\phi$ on the conjugates of $\iota^* (X^*(T))$ in $X^*(\tilde T)$ is simply transitive.
\end{proof}

\section{Describing general pointwise conorms in terms of special cases}
\label{sec:reduction}

Suppose that $\tilde{G}$ and $G$ are quasisplit.
Then a parascopic datum $(\phi,q^*)$ for a triple $(\tilde{G},\Gamma, G)$
relative to tori $\tilde{T}\subseteq \tilde{G}$ and $T\subseteq G$
gives rise to a conorm $\dnormst[{[\phi,j_*]}]$,
a map from stable semisimple conjugacy classes
in $G^\wedge(k)$
to those in
$\tilde{G}^\wedge(k)$.

The goal of this section is to prove the following result:
\begin{prop}
\label{prop:explicit}
The conorm $\dnormst[{[\phi,j_*]}]$
can be described explicitly in terms of
compositions,
direct products, Weil restriction,
the following conorms:
\begin{itemize}
\item
the map in
Example \ref{ex:conormfunc-isogeny} arising from an isotypy,
\item
the map $\diag$ of
Example \ref{ex:permutation-action}
and \S\ref{para:product} arising from actions on products,
\item
the embeddings of dual groups in \S\ref{para:conorm-dual-closed}
and \S\ref{para:wrong-subsystem},
\item
the map $\diag_{E/k}$ from the top of \S\ref{sec:weil-parascopic} arising from Weil restriction,
\item
the map in Example \ref{ex:conorm-quasicentral} (below),
\end{itemize}
and
conorms for the specific examples 
in Tables \ref{table:quasi-central} and \ref{table:subdatum} below.
Moreover, if none of the conorms from these tables are involved,
then 
$\dnormst[{[\phi,j_*]}]$
has a pointwise conorm function associated to it.
\end{prop}

By Remark \ref{rem:decomp-canonical}
and Example \ref{ex:conormfunc-isogeny},
we may and will assume
that our parascopic datum $(\phi,q^*)$ is torus-inclusive.
By Propositions \ref{prop:conorm-factorization} and \ref{prop:decomposition-gamma},
we may and will assume that $\Gamma$ is simple or trivial.
In \S\ref{sec:reduction-abs-simple}, we reduce to the case where
$\tilde{G}$ is absolutely simple and adjoint,
and one of the following holds:
(a) 
$\Gamma$ is cyclic and acts nontrivially on the Dynkin
diagram of $\tilde{G}$;
or (b) $\Gamma$ is trivial.
In \S\ref{sec:reduction-not-A2n},
putting off case (b), 
and temporarily 
ignoring one situation where $\Gamma$ acts via an involution
and $\tilde{G}$ has type $A_{2n}$,
we further reduce case (a)
to the cases where
(a$^\prime$) $\Gamma$ acts quasicentrally on $\tilde{G}$;
and (b) $\Gamma$ is trivial (again).
In case (a$^\prime$), we can produce pointwise conorm functions in some cases,
and in others (listed in Table \ref{table:quasi-central})
such functions appear not to exist.
In \S\ref{sec:reduction-A2n},
we take care of one piece of unfinished business:
the cases where $\Gamma$ acts via an involution on $A_{2n}$.
In all such cases, we either produce a pointwise conorm function
or reduce to simpler cases.
Finally, in \S\ref{sec:subdatum},
we address the remaining piece of unfinished business:
case (b) above.
One usually has a pointwise conorm function,
and the exceptions are given in Table \ref{table:subdatum}.

Let $\Psi$ and $\tilde\Psi$ denote the root data
for $(G,T)$ and $(\tilde G, \tilde T)$,
and let $\Phi$ and $\tilde\Phi$ denote their root systems.

\subsection{Reduction to a simple case}
\label{sec:reduction-abs-simple}
We show how to reduce to the case where $\tilde{G}$
is absolutely simple and adjoint, and $\Gamma$ either
(a) is cyclic
and acts nontrivially on the Dynkin diagram,
or (b) is trivial.

By Proposition \ref{prop:conorm-isogeny},
we may assume that $\tilde{G}$
is a direct product of
adjoint, $k$-simple factors.
Since the conorm
behaves well with respect to direct products of $\Gamma$-invariant factors,
we may assume that
$\Gamma$ permutes
the $k$-simple factors of $\tilde{G}$ transitively.
By Proposition \ref{prop:product-conorm},
we may assume that $\tilde{G}$ is in fact $k$-simple.
Since $\tilde{G}$ is adjoint, we have that
$\tilde{G} = \Weil_{E/k} \tilde H$
for some finite separable extension $E/k$ and some adjoint,
$E$-quasisplit
absolutely simple $E$-group $\tilde H$.

From Proposition \ref{prop:conorm-weil-restriction-general},
we may assume that the action $\phi$ on $\tilde\Psi$ is not $K/k$-induced
for any intermediate field $K$ of $E/k$. If $E\neq k$,
then as in \S\ref{sub:act-via-group-aut}, $(\phi,q^*)$ is the composition of
a parascopic datum for $(\tilde G, \Gamma, \bar G)$, for which the conorm is given
in Proposition~\ref{prop:conorm-base-change} and one of the form
$(1,q^*)$ for $(\bar G, 1, G)$, which case is handled in \S\ref{sec:subdatum}.
Thus we may now assume that $E = k$ and $\tilde G = \tilde H$
is absolutely simple.

Since $\Gamma$ is simple or trivial, it acts either trivially or faithfully
on the absolute Dynkin diagram of $\tilde{G}$.
If faithfully, then $\Gamma$ is cyclic.

Suppose $\Gamma$ acts trivially on the Dynkin diagram.
Then $\Gamma$ acts trivially on the root datum $\tilde\Psi$.
From Proposition \ref{prop:decomposition-gamma},
our parascopic datum $(\phi,q^*)$
for $(\tilde{G},\Gamma,G)$
is a composition of
the parascopic datum $(\phi,\id)$
for $(\tilde{G},\Gamma,\tilde{G})$
and
the parascopic datum $(1,q^*)$
for $(\tilde{G},1,G)$.
The first datum has an associated conorm function given
by $s\mapsto s^{|\Gamma|}$,
as in Example \ref{ex:trivial-action}.
As indicated above, the second will be addressed in \S\ref{sec:subdatum}

\subsection{Root-inclusive cases with nontrivial $\Gamma$ action}
\label{sec:reduction-not-A2n}
Suppose that $\tilde{G}$ is absolutely simple
and adjoint,
and $\Gamma$ is a cyclic group that acts nontrivially
on the Dynkin diagram of $\tilde{G}$.
Additionally, let us assume that 
$(\phi,q^*)$ is root-inclusive.
Doing so only excludes certain cases where $\tilde{G}$ has type $A_{2n}$
for some $n$,
and we will treat these cases in \S\ref{sec:reduction-A2n}.

By Proposition \ref{prop:decomp-fixed-pinning}
we can decompose $(\phi,q^*)$ into
\begin{enumerate}[(a)]
\item[(a$^\prime$)]
a parascopic datum coming from a quasicentral
action of $\Gamma$ on $\tilde{G}$, in which case $\Psi = \lsub\phi\tilde\Psi\red$; and
\item[(b)]
a parascopic datum coming from a root subdatum.
\end{enumerate}
Again postponing discussion of (\refb)
to \S\ref{sec:subdatum},
we now give explicit pointwise conorm functions in two subcases of (\refaprime) below.
The remaining three subcases, where we do not know of explicit pointwise
conorm functions, are listed in Table \ref{table:quasi-central}.

\begin{table}
$$
\begin{array}{l | c | l l}
\tilde{\Phi} & |\Gamma| & \Phi &  \\
\cline{1-3}
A_{2n} & 2 & B_n^{(*)}  &
	^{(*)}\text{$C_n$ if $p = 2$} \\
A_{2n-1} & 2 & C_n  \\
D_n & 2 & B_{n-1}  \\
\end{array}
$$
\caption{Quasicentral actions for which we don't know of a pointwise
conorm function}
\label{table:quasi-central}
\end{table}

\begin{example}
\label{ex:conorm-quasicentral}
We now give pointwise conorm functions in the cases of
(\refaprime) where
$(\tilde\Phi, |\Gamma|, \Phi)$ is
$(D_4, 3, G_2)$
or
$(E_6, 2, F_4)$.
As in \S\ref{para:wrong-subsystem}, one can
define a $k$-action of $\Gamma$ on $\tilde G^\wedge$,
and it follows easily that $G$ is in $k$-duality with the $k$-group
$G^\wedge: = (\tilde G^{\wedge\,\Gamma})\conn$.
Let $\iota$ be the embedding $\abmap{G^\wedge}{\tilde G^\wedge}$.
It can be checked that the map
from $G^\wedge$ to $\tilde G^\wedge$ given by 
$s \mapsto \iota(s^{|\Gamma|})$ agrees with a conorm map
$\dnorm[S^\wedge]$
for every maximal $k$-torus $S^\wedge \subset G^\wedge$,
and so
this map is a pointwise conorm for the action of $\Gamma$ on $\tilde G$.
\end{example}

\subsection{Case of $A_{2n}$ with an involution}
\label{sec:reduction-A2n}
Suppose that $\tilde\Psi$ has type $A_{2n}$ for some $n$,
$\Gamma$ has order two, and $\Gamma$ acts nontrivially on a positive
system of roots in $\tilde\Psi$.
Our goal in this section is to express the associated parascopy
for $(\tilde\Psi,\Gamma,\Psi)$ in terms of simpler cases.

The restricted root datum
$\lsub\phi \tilde\Psi$ has type $BC_n$.
Our parascopic datum for $(\tilde\Psi,\Gamma,\Psi)$ decomposes
into a datum for $(\tilde \Psi, \Gamma, \Psi')$
and a datum for $(\Psi', 1, \Psi)$,
where $\Psi'$ is a maximal reduced subsystem of $\lsub\phi\tilde\Psi$
containing $\Psi$.
We will address the second datum in 
\S\ref{sec:subdatum},
so replace $\Psi$ by $\Psi'$,
and thus assume from now on that $\Phi$ is maximal.
In particular, $\Phi$ cannot have type $D_n$,
so it must contain a multipliable or divisible root.

Let
$R =
\{
\alpha \in \lsub\phi\tilde\Phi
	\mid \text{$\alpha$, $2\alpha$, or $\alpha/2 \in \Phi$}
\}$.
Then $R$ is a (not necessarily proper)
maximal subsystem of $\lsub\phi\tilde\Phi$.
By construction, there is some root $\alpha\in R$ such that
$2\alpha\in R$.
Let $R_1$ be the irreducible factor of $R$ containing $\alpha$.
Then $R_1$ has type $BC_m$ for some $0<m\leq n$.
By the maximality of $R$, we have that $R = R_1\sqcup R_2$,
where $R_2$ is the set of roots $\lsub\phi\tilde\Phi$ orthogonal to $R_1$.
Since $R_2$ has an orthogonal basis consisting of the multipliable
roots in $\lsub\phi\tilde\Phi$ that are not in $R_1$,
we see that $R_2$ has type $BC_{n-m}$.
By the maximality of $\Phi$ among reduced subsystems of $R$,
we must have that $\Psi = \Psi_1 \times \Psi_2$, where 
$\Psi_1$ and $\Psi_2$ each have type $B$ or $C$.

If $m=n$, then $\Psi_2$ is empty, and $\Psi$ has type $B_n$
or $C_n$.
Suppose $m<n$.
If $\Psi_1$ and $\Psi_2$ both have the same type (e.g., $B$ or $C$),
then $\Psi$ is properly contained in $B_n$ or $C_n$,
which contradicts the maximality of $\Psi$.
Therefore, we may assume without loss of generality that
$\Psi_1$ has type $B_m$ and $\Psi_2$ has type $C_{n-m}$.
Since
the root system $\Phi_1$ of $\Psi_1$
(resp.\ $\Phi_2$ of $\Psi_2$)
consists of all
nondivisible
(resp.\ nonmultipliable)
roots in $R_1$ (resp.\ $R_2$),
it is straightforward to check that the inverse images
of $\Phi_1$ and $\Phi_2$ under the quotient map
$\map{i^*}{X^*(\tilde T)}{\lsub\phi X^*(\tilde T)}$
are orthogonal sets in $\tilde\Phi$.
Since $\tilde\Phi$ has type $A$, the sum of two orthogonal
roots in $\tilde\Phi$ is never a root.
Therefore, the subsystem
generated by these two orthogonal sets
is a disjoint union $\tilde \Phi' := \tilde\Phi_1\sqcup \tilde\Phi_2$,
where each $\tilde\Phi_j$ contains the inverse image
under $i^*$ of $\Phi_j$.
The subdatum $\tilde\Psi'$ of $\tilde\Psi$ determined by $\tilde\Phi'$
corresponds to a Levi subgroup $\tilde G'$
of $\tilde G$.
By \S\ref{para:conorm-dual-closed},
the embedding $\abmap{\tilde G'^\wedge}{\tilde G^\wedge}$
is a pointwise conorm for the equivalence class $[1,\id]$
for $(\tilde G,1,\tilde G')$.
Therefore,
we may replace $\tilde\Psi$ and $\tilde G$ by $\tilde\Psi'$ and $\tilde G'$.
From Proposition \ref{prop:conorm-isogeny},
we may further replace $\tilde G$ by a direct product of semisimple groups
$\tilde G_1 \times \tilde G_2$,
where the root system of each $\tilde G_j$ is $\tilde\Phi_j$,
which must have type $A$.
Since $i^*(\tilde\Phi_1)$ contains $B_m$
as a maximal reduced subsystem,
$\tilde\Phi_1$ must have type $A_{2m}$,
and $\Gamma$ acts nontrivially on it.
Similarly,
since $i^*(\tilde\Phi_2) \supseteq C_{n-m}$,
$\tilde\Phi_2$ must have type $A_{2n-2m-1}$ or $A_{2n-2m}$,
and $\Gamma$ acts nontrivially on it.
But we cannot have $A_{2n-2m}$ since then $\tilde\Psi$
would not be contained in $A_{2n}$.
We can handle each factor $\tilde\Phi_j$ separately,
and $\tilde\Phi_2$ is addressed in
\S\ref{sec:reduction-not-A2n}.
Therefore, assume that $\tilde\Psi$ is a single factor,
and that $\Psi$ has type $B_n$ or $C_n$.

Recall
(from \S\ref{sec:basic})
that $(\lsub\phi\tilde\Phi)\red		= B_m$
and  $(\lsub\phi\tilde\Phi)\redwrong	= C_m$
if $p\neq 2$,
and they are reversed if $p=2$.
If $\Psi$ is contained in $(\lsub\phi\tilde\Psi)\red$, then $(\phi,q^*)$
is root-inclusive, and this case was covered in
\S\ref{sec:reduction-not-A2n}.
Otherwise,
by Remark \ref{rem:decomp-canonical},
we can decompose $(\phi,q^*)$ into
a parascopic datum
for $(\tilde\Psi,\Gamma,(\lsub\phi\tilde\Psi)\redwrong)$
(this is \S\ref{para:wrong-subsystem}, where a pointwise
conorm is given), 
and a parascopic datum $(1,\id)$
for $((\lsub\phi\tilde\Psi)\redwrong,1,\Psi)$,
which will be discussed in \S\ref{sec:subdatum}.

\subsection{Subdatum cases}
\label{sec:subdatum}
Suppose that $\Gamma$ is trivial.
Then $\Psi$ is a $\Gal(k)$-invariant subdatum of $\tilde\Psi$.
By induction,
we may assume that $\Psi$ is maximal among $\Gal(k)$-invariant
subdata.
From Proposition \ref{prop:max-subdatum},
$\Psi$ must be either closed or dual-closed in $\tilde\Psi$.
If it is dual-closed, then
\S\ref{para:conorm-dual-closed}
gives an explicit pointwise conorm.
Therefore, we may assume that our subdatum
is maximal, closed, but not dual-closed,
hence has the same rank as $\tilde\Phi$.

Replacing our parascopic datum $(1,q^*)$ by an equivalent
datum defined with respect to a $k$-torus $T\subset G$ 
that lies in a Borel $k$-subgroup of $G$,
we may assume that our $\Gal(k)$ action on $\Psi$
preserves a system of positive roots.

Using a result of
Borel and de Siebenthal~\cite{borel-desiebenthal}*{\S7},
it is not hard to list all maximal, closed, proper,
full-rank subsystems of $\tilde\Phi$.
Iterating this process, it is not hard to list all closed,
full-rank subsystems of $\tilde\Phi$.
In Table \ref{table:subdatum},
we list all such subsystems
that are maximal among those that are invariant under
the action of $\Gal(k)$.
It turns out
that there is only one case of a non-maximal subsystem 
that is nonetheless maximal among those that are $\Gal(k)$-invariant:
$D_4$ with a triality action, which is contained in $B_4$,
a subsystem of $F_4$ that is not preserved by the $\Gal(k)$ action.
Therefore, in Table \ref{table:subdatum},
we list all maximal, closed,
equal-rank subsystems of $\tilde\Phi$,
and also $D_4$,
omitting a maximal subsystem of $F_4$ that is both closed and dual-closed.

\begin{table}
$$
\begin{array}{l | l l}
\tilde{\Phi} &  \Phi &  \\
\cline{1-3}
B_n & B_r \times D\Long_s &
	(r+s = n) \\
C_n & C_r \times C_s & 
	(r+s = n) \\
F_4 & B_4  \\
&	A_3\Long \times A_1\Short \\
&	A_1\Long \times C_3 \\
&	^3D_4\Long, \: ^6D_4\Long  & \text{(not maximal; in $B_4$)}\\
G_2 & A_2\Long \\
\end{array}
$$
\caption{Cases where $\Phi$ is a subdatum of $\tilde\Phi$,
and we don't know of a pointwise conorm function}
\label{table:subdatum}
\end{table}

%
%

\end{document}